\documentclass[11pt]{amsart}

\usepackage{amscd,amssymb,amsmath,amsthm,mathrsfs}
\usepackage{amsfonts}
\usepackage{yfonts}
\usepackage{fullpage}
\usepackage{enumitem} 
\usepackage{pdfsync}
\usepackage{hyperref}

\usepackage{color}
\usepackage[backrefs, alphabetic]{amsrefs}
\usepackage[all]{xy} 

\newtheorem{maintheorem}{Theorem}

\newtheorem{proposition}{Proposition}[section]
\newtheorem{theorem}[proposition]{Theorem}
\newtheorem{lemma}[proposition]{Lemma}
\newtheorem{fact}[proposition]{Fact}
\newtheorem{corollary}[proposition]{Corollary}
\newtheorem*{claim*}{Claim}

\theoremstyle{definition}

\newtheorem*{problem*}{Problem}

\theoremstyle{remark}
\newtheorem{remark}[proposition]{Remark}
\newtheorem{example}[proposition]{Example}

\numberwithin{equation}{section}

\DeclareMathOperator{\Aut}{Aut}
\DeclareMathOperator{\Gal}{Gal}

\DeclareMathOperator{\Char}{Char}
\DeclareMathOperator{\cl}{cl}

\DeclareMathOperator{\divv}{div}
\DeclareMathOperator{\trdeg}{tr.deg}
\DeclareMathOperator{\Supp}{Support}
\DeclareMathOperator{\image}{Image}
\DeclareMathOperator{\Spec}{Spec}
\DeclareMathOperator{\Proj}{Proj}
\renewcommand{\H}{\operatorname{H}}

\DeclareMathOperator{\D}{D}
\DeclareMathOperator{\I}{I}
\DeclareMathOperator{\Hom}{Hom}

\DeclareMathOperator{\Out}{Out}
\DeclareMathOperator{\Inn}{Inn}

\newcommand{\Z}{\mathbb{Z}}
\newcommand{\Q}{\mathbb{Q}}

\newcommand{\G}{\mathbb{G}}
\newcommand{\U}{{\rm U}}

\newcommand{\Gbb}{\mathbb{G}}
\newcommand{\Abb}{\mathbb{A}}
\newcommand{\Pbb}{\mathbb{P}}

\newcommand{\Kbb}{\mathbb{K}}

\newcommand{\Vc}{\mathcal{V}}
\newcommand{\Gc}{\mathcal{G}}
\newcommand{\Dcal}{\mathscr{D}}
\newcommand{\Oc}{\mathcal{O}}

\newcommand{\Lcal}{\mathcal{L}}
\newcommand{\Ccal}{\mathcal{C}}
\newcommand{\Fcal}{\mathcal{F}}
\newcommand{\Rcal}{\mathcal{R}}
\newcommand{\Ucal}{\mathcal{U}}
\newcommand{\Hcal}{\mathcal{H}}

\newcommand{\abf}{\mathbf{a}}

\newcommand{\tbf}{\mathbf{t}}
\newcommand{\vbf}{\mathbf{v}}
\newcommand{\wbf}{\mathbf{w}}

\newcommand{\Kfrak}{\mathfrak{K}}
\newcommand{\Gfrak}{\mathfrak{G}}
\newcommand{\Ufrak}{\mathfrak{U}}
\newcommand{\mf}{\mathfrak{m}}
\newcommand{\Dfrak}{\mathfrak{D}}

\newcommand{\gfrak}{\mathfrak{g}}

\newcommand{\Vfrak}{\mathfrak{V}}

\newcommand{\Frob}{\operatorname{Frob}}
\newcommand{\Kummer}{\mathcal{K}}

\newcommand{\Autm}{\Aut^{\rm M}}
\newcommand{\Autc}{\Aut^{\rm c}}
\newcommand{\UAut}{\underline{\Aut}}
\newcommand{\UAutc}{\underline{\Aut}^{\rm c}}

\newcommand{\UAutm}{\underline{\Aut}^{\rm M}}

\renewcommand{\k}{\operatorname{k}^{\rm M}}
\newcommand{\K}{\operatorname{K}^{\rm M}}

\newcommand{\dimm}{\dim^{\rm M}}

\newcommand{\Profinite}{{\mathbf{Prof}_{\rm Out}}}

\newcommand{\Homc}{\Hom^{\rm c}}

\newcommand{\UHomc}{\underline{\Hom}^{\rm c}}

\newcommand{\one}{\mathbf{1}}
\newcommand{\Var}{\mathbf{Var}}
\newcommand{\AbC}{\mathbf{AbC}_{\ell}}
\newcommand{\ac}{{\rm ac}}
\newcommand{\msup}{{\sup}^{\rm M}}
\newcommand{\Galk}{{\Gal_{k_0}}}

\title{The Galois action on geometric lattices and the mod-$\ell$ I/OM}
\author{Adam Topaz}
\address{Adam Topaz \vskip 0pt
     Mathematical Institute \vskip 0pt
     University of Oxford \vskip 0pt
     Andrew Wiles Building \vskip 0pt
     Radcliffe Observatory Quarter \vskip 0pt
     Woodstock Road \vskip 0pt
     Oxford OX2 6GG \vskip 0pt
     United Kingdom 
     }
\email{topaz@maths.ox.ac.uk}
\urladdr{http://adamtopaz.com}
\date{\today}
\thanks{Research supported in part by NSF postdoctoral fellowship DMS-1304114, and in part by EPSRC programme grant EP/M024830/1 Symmetries and Correspondences}

\begin{document}


\maketitle

\begin{abstract}
  This paper studies the Galois action on a special lattice of geometric origin, which is related to mod-$\ell$ abelian-by-central quotients of geometric fundamental groups of varieties.
  As a consequence, we formulate and prove the mod-$\ell$ abelian-by-central variant/strengthening of a conjecture due to {\sc Ihara/Oda-Matsumoto}.
\end{abstract}

\tableofcontents

\section{Introduction}
\label{section: intro}


Our story begins with a question of {\sc Ihara} from the 1980s, which asked for a combinatorial description of the absolute Galois group of $\Q$.
More precisely, this combinatorial description should be in the spirit of {\sc Grothendieck's} \emph{Esquisse d'un Programme} \cite{Grothendieck:1997}, which suggested studying absolute Galois groups via their action on objects of ``geometric origin,'' and specifically the geometric fundamental group of algebraic varieties.
{\sc Ihara} asked whether the absolute Galois group of $\Q$ is isomorphic to the automorphism group of the geometric fundamental group functor on $\Q$-varieties, and {\sc Oda-Matsumoto} \cite{Matsumoto:1997} later conjectured that the answer is affirmative, based on motivic evidence.
We will henceforth refer to this question/conjecture (and its various variants) as the ``I/OM.''

The original I/OM conjecture, which deals with the full geometric fundamental group, and which we call ``the absolute I/OM'' below, was proven by {\sc Pop} in an unpublished manuscript from the 1990s.
A variant of the I/OM over $p$-adic fields, using \emph{tempered fundamental groups}, was then developed and proved by {\sc Andr\'e} \cite{Andre:2003}. 
Later on, {\sc Pop} formulated and proved a strengthening of the absolute I/OM, which instead deals with the maximal \emph{pro-$\ell$ abelian-by-central} quotient of the geometric fundamental group.
The pro-$\ell$ abelian-by-central I/OM implies the absolute I/OM, and both contexts are treated by {\sc Pop} in \cite{Pop:2014}.

In this paper, we develop and prove a further strengthening of I/OM, which deals with the \emph{mod-$\ell$ abelian-by-central} quotient of the geometric fundamental group.
This mod-$\ell$ context strengthens both the pro-$\ell$ abelian-by-central and the absolute situations.
Furthermore, the mod-$\ell$ abelian-by-central quotient is the \emph{smallest possible} functorial (pro-$\ell$) quotient which remains non-abelian.
In this sense, the mod-$\ell$ context yields the strongest possible results that one could hope for.

Most importantly however, the mod-$\ell$ abelian-by-central context gets much closer to the spirit of {\sc Ihara's} original question of finding a combinatorial description of absolute Galois groups.
Indeed, the geometric fundamental group of a variety and its pro-$\ell$ abelian-by-central quotient are both finitely-generated profinite resp. pro-$\ell$ groups, and the topology of such groups plays a crucial role in both situations.
In contrast to this, the mod-$\ell$ abelian-by-central quotient can be seen as a (discrete) finite-dimensional $\Z/\ell$-vector space endowed with some extra linear structure.
In other words, the mod-$\ell$ abelian-by-central quotient of a geometric fundamental group is an object of a \emph{purely combinatorial nature}, being a finite-dimensional linear object over $\Z/\ell$.

The precise notation and context of the paper is somewhat involved.
For the sake of the reader, we now give some brief (and mostly unmotivated) definitions in order to state the primary main theorem of the paper.
The rest of the Introduction will provide the full detailed notation and motivation.
Let $\ell$ be a fixed prime.
Let $k_0$ be a field of characteristic $\neq \ell$, and let $X$ be a normal, geometrically-integral $k_0$-variety.
For such an $X$, we write (see \S\ref{subsection: intro / mod-ell-category}, \S\ref{subsection: intro / mod-ell-IOM}):
\begin{enumerate}
  \item $\bar\pi_1(X) := \pi_1^\text{\'et}(\bar X,\bar x)$ for the geometric fundamental group of $X$, i.e. the \'etale fundamental group of the base-change $\bar X$ of $X$ to $\bar k_0$, with respect to some geometric point $\bar x$.
  \item $\pi^a(X)$ for the maximal \emph{mod-$\ell$ abelian} quotient of $\bar\pi_1(X)$.
  \item $\pi^c(X)$ for the maximal \emph{mod-$\ell$ abelian-by-central} quotient of $\bar\pi_1(X)$.
\end{enumerate}
We consider $\Autc(\pi^a(X))$, the set of automorphisms of $\pi^a(X)$ which arise\footnote{This condition will be made precise in \S\ref{subsection: intro / mod-ell-category} below.} from some automorphism of $\pi^c(X)$.

For an essentially small category $\Vc$ of normal geometrically-integral $k_0$-varieties, we consider the group $\Autc(\pi^a|_\Vc)$ which consists of systems $(\phi_X)_{X \in \Vc}$, where $X$ varies over the objects of $\Vc$, and the $\phi_X \in \Autc(\pi^a(X))$ are compatible with morphisms arising from $\Vc$.
Since $\pi^a(X)$ is a $\Z/\ell$-vector space for all $X \in \Vc$, and the morphisms $\pi^a(X) \rightarrow \pi^a(Y)$ arising from morphisms $X \rightarrow Y$ in $\Vc$ are $\Z/\ell$-linear, we obtain a canonical action of $(\Z/\ell)^\times$ on $\Autc(\pi^a|_\Vc)$ by left-multiplication.
We write $\UAutc(\pi^a|_\Vc) := \Autc(\pi^a|_\Vc)/(\Z/\ell)^\times$ for the quotient by this canonical action of $(\Z/\ell)^\times$.

The action of $\Galk := \Gal(\bar k_0|k_0)$ on $\bar X$ for $X$ as above yields a canonical (outer) action on $\pi^a(X)$ and $\pi^c(X)$, hence also a canonical Galois representation
\[ \rho_{k_0,X}^c : \Galk \rightarrow \Autc(\pi^a(X)). \]
Finally, by collecting the $\rho_{k_0,X}^c$ for all $X$ in $\Vc$, we obtain a canonical Galois representation
\[ \rho_{k_0,\Vc}^c : \Galk \rightarrow \Autc(\pi^a|_\Vc) \twoheadrightarrow \UAutc(\pi^a|_\Vc). \]
With this notation, the main theorem of the present paper reads as follows.

\vskip 5pt
\noindent{\bf Main Theorem:}
{\em Let $k_0$ be an infinite perfect field of characteristic $\neq \ell$, and let $\Vc$ be a ``sufficiently large'' category of normal geometrically-integral $k_0$-varieties.
Then the canonical map 
\[ \rho_{k_0,\Vc}^c : \Galk \rightarrow \UAutc(\pi^a|_\Vc) \]
is an isomorphism.
}
\vskip 5pt

We will define precisely what we mean by ``sufficiently large'' in \S\ref{subsection: intro / main-result-connected}, where the precise assumption/terminology is that $\Vc$ should be ``$5$-connected.''
However, we note here that both the full category of all normal geometrically-integral $k_0$-varieties, and the full category of geometrically-integral smooth quasi-projective $k_0$-varieties, are ``sufficiently large'' in this sense.
The above Main Theorem is proved for such $5$-connected categories in Theorem \ref{maintheorem: intro / main-result / iom-main-connected}.
In fact, our approach is \emph{birational}, and we obtain the above Main Theorem for \emph{many more} categories $\Vc$.
The following is a concrete example which follows directly from Theorem \ref{maintheorem: intro / main-result / iom-main-opens}.
\begin{example}
  Let $(H_i)_i$ be a cofinal system of (possibly non-integral) $\Q$-hypersurfaces in $\Pbb_\Q^5$, and put $U_i := \Pbb_\Q^5 \smallsetminus H_i$.
  For each $t \in \Q(\Pbb_\Q^5) \smallsetminus \Q$ and $U_i$ sufficiently small, let $p_t : U_i \rightarrow \Gbb_m$ be the dominant morphism corresponding to $t$.
  Let $\Vc$ be the category whose objects are $\{U_i\}_i \cup \{\Gbb_m\}$, and whose morphisms are the inclusions among the $U_i$, the identity on $\Gbb_m$, and the various $p_t : U_i \rightarrow \Gbb_m$ above.
  Then the canonical map 
  \[ \rho_{\Q,\Vc}^c : \Gal_\Q \rightarrow \UAutc(\pi^a|_\Vc) \]
  is an isomorphism.
\end{example}

The four main theorems of the paper (Theorems \ref{maintheorem: intro / main-result / iom-main-opens}, \ref{maintheorem: intro / main-result / iom-main-connected}, \ref{maintheorem: intro / galois-variant / galois-main} and \ref{maintheorem: intro / milnor-variant / milnor-main}) all require some variant of a ``dimension $\geq 5$'' assumption; this also leads to the ``$5$-connectedness'' condition mentioned above.
It is natural to ask whether these results hold true for dimension $d$, with $1 < d < 5$. 
However, the lower bound of $5$ seems to be the best that one can presently hope for, in the mod-$\ell$ context.
As explained below, this assumption first arises in the results of {\sc Evans-Hrushovski} \cite{Evans:1991}, \cite{Evans:1995} and {\sc Gismatullin} \cite{Gismatullin:2008}, which play a crucial role in our proofs.
In these works, this assumption is required in a technical step arising from the use of the \emph{group configuration theorem}.
The same assumption also arises, \emph{for different reasons,} in two key Lemmas \ref{lemma: mainproof / base-case / main-base-case}, \ref{lemma: mainproof / general-case / all-acceptable} of the present paper.
Nevertheless, the main \emph{large} categories of geometrically-integral normal $k_0$-varieties, which are of interest for the I/OM, are indeed ``$5$-connected,'' as we have mentioned above.

\subsection{Automorphism groups of functors}
\label{subsection: intro / automorphism-groups}

We first introduce some general notation and terminology which will help simplify the exposition.
Let $\Ccal$ be an essentially small category, and let $\Fcal : \Ccal \rightarrow \mathcal{D}$ be an arbitrary functor to another category $\mathcal{D}$.
We consider $\Aut(\Fcal)$, the \emph{automorphism group of the functor} $\Fcal$.
To be explicit, the elements of $\Aut(\Fcal)$ are systems $(\phi_X)_{X \in \Ccal}$ with $\phi_X \in \Aut(\Fcal(X))$ parameterized by the objects $X$ of $\Ccal$, such that for every morphism $f : X \rightarrow Y$ of $\Ccal$, one has the following commutative diagram in $\mathcal{D}$
\[
\xymatrix{
  \Fcal(X) \ar[r]^{\phi_X} \ar[d]_{\Fcal(f)} & \Fcal(X) \ar[d]^{\Fcal(f)} \\
  \Fcal(Y) \ar[r]_{\phi_Y} & \Fcal(Y). 
}
\]

For a subcategory $\Ccal_0$ of $\Ccal$, we write $\Fcal|_{\Ccal_0}$ for the restriction of $\Fcal$ to $\Ccal_0$, and we will frequently consider the automorphism group $\Aut(\Fcal|_{\Ccal_0})$ of this restriction.
Note that this restriction yields a canonical restriction morphism
\[\phi \mapsto \phi|_{\Ccal_0} : \Aut(\Fcal) \rightarrow \Aut(\Fcal|_{\Ccal_0}). \]
In explicit terms, this restriction morphism sends a system $(\phi_X)_{X \in \Ccal}$, as above, to the restricted system $(\phi_X)_{X \in \Ccal_0}$.

We will frequently consider several different target categories $\mathcal{D}$.
To keep the notation consistent throughout, if some symbol/notation is used to denote the automorphism groups of objects in $\mathcal{D}$, then we will use the same symbol/notation to denote the automorphism group of a functor whose values are in $\mathcal{D}$.

\subsection{The Absolute I/OM}
\label{subsection: intro / absolute-IOM}

Throughout the paper, we will work with a fixed infinite perfect field $k_0$, and $k = \bar k_0$ will denote the algebraic closure of $k_0$.
Furthermore, $\Galk := \Gal(k|k_0)$ will denote the absolute Galois group of $k_0$.
We will only consider normal geometrically-integral $k_0$-varieties, and we denote the category of all such varieties by $\Var_{k_0}$.
Namely, the objects of $\Var_{k_0}$ are schemes which are normal, geometrically-integral, separated and of finite-type over $k_0$.
The morphisms in $\Var_{k_0}$ are just morphisms of $k_0$-schemes.

For every $X \in \Var_{k_0}$, we write 
\[ \bar X = X \otimes_{k_0} k = X \times_{\Spec k_0} \Spec k \]
for the base-change of $X$ to $k$.
Since $k_0$ is perfect, it follows that $\bar X$ is a normal $k$-variety for every $X \in \Var_{k_0}$.
Furthermore, recall that $\Galk$ acts on $\bar X$ in the usual canonical way.
Namely, $\sigma \in \Galk$ acts on $\bar X$ via the automorphism $\one \times \Spec(\sigma^{-1})$ of 
\[ X \times_{\Spec k_0} \Spec k = \bar X. \]

We let $\Profinite$ denote the category of profinite groups with \emph{outer-morphisms}.
Namely, the objects of $\Profinite$ are just the profinite groups, but the set of morphisms $G \rightarrow H$ in $\Profinite$ is given by 
\[ \Hom_{\Out}(G,H) := \Hom_{\rm cont}(G,H)/\Inn(H). \]
In particular, the automorphism group of a profinite group $G$ in $\Profinite$ is precisely $\Out(G)$, the outer-automorphism group of $G$.

For every $X \in \Var_{k_0}$, consider the \emph{geometric (\'etale) fundamental group} of $X$,
\[ \bar\pi_1(X) := \pi_1^\text{\'et}(\bar X,\bar x) \]
with respect to some geometric point $\bar x$.
We will consider $\bar\pi_1(X)$ as an object of $\Profinite$, and because of this, the choice of geometric point becomes irrelevant.
We therefore omit the geometric point $\bar x$ from the notation.
To summarize, we obtain a canonical functor
\[ \bar\pi_1 : \Var_{k_0} \rightarrow \Profinite. \]

Observe that the action of $\Galk$ on $\bar X$ and the functoriality of $\pi_1^\text{\'et}$ with values in $\Profinite$, yields a canonical \emph{outer} Galois representation
\[ \rho_{k_0,X} =: \rho_{k_0} : \Galk \rightarrow \Out(\bar\pi_1(X)). \]
Naturally, this outer representation agrees with the one arising from the \emph{fundamental short exact sequence}
\[ 1 \rightarrow \bar\pi_1(X) \rightarrow \pi_1^\text{\'et}(X,\bar x) \rightarrow \Galk \rightarrow 1. \]
Moreover, it is clear that if $X \rightarrow Y$ is a morphism in $\Var_{k_0}$, then the induced morphism $\bar\pi_1(X) \rightarrow \bar\pi_1(Y)$ in $\Profinite$ is compatible with the action of $\Galk$.
In particular, $\Galk$ acts on \emph{the functor} $\bar\pi_1$ itself.

Similarly, if $\Vc$ is any (essentially small) subcategory of $\Var_{k_0}$, then $\Galk$ acts on $\bar\pi_1|_\Vc$.
Following our notational convention mentioned in \S\ref{subsection: intro / automorphism-groups}, we denote the automorphism group of $\bar\pi_1|_\Vc$ by $\Out(\bar\pi_1|_\Vc)$.
We therefore obtain a canonical \emph{outer} Galois representation
\[ \rho_{k_0,\Vc} : \Galk \rightarrow \Out(\bar\pi_1|_\Vc). \]
With this notation, the \emph{Absolute I/OM} refers to the following general question.

\vskip 5pt
\noindent{\bf The Absolute I/OM:}
{\em For which fields $k_0$ and subcategories $\Vc$ of $\Var_{k_0}$ as above, is the Galois representation 
\[ \rho_{k_0,\Vc} : \Galk \rightarrow \Out(\bar\pi_1|_\Vc) \]
an isomorphism?}
\vskip 5pt

The nature of the map $\rho_{k_0,\Vc}$ in general is still quite mysterious.
Nevertheless, for $k_0 = \Q$, the \emph{injectivity} of this morphism has been extensively studied.
For instance {\sc Drinfeld} \cite{Drinfeld:1990} observed that Belyi's theorem \cite{Belyi:1980} implies $\rho_{\Q,\Vc}$ is injective as soon as $\Vc$ contains \emph{the tripod}, $\Pbb^1_\Q \smallsetminus \{0,1,\infty\}$.
More generally, it follows from the work of {\sc Voevodsky} \cite{Voevodsky:1991} and {\sc Matsumoto} \cite{Matsumoto:1996} in the affine case, and {\sc Hoshi-Mochizuki} \cite{Hoshi:2011} in general, that $\rho_{\Q,\Vc}$ is injective as soon as $\Vc$ contains a (possibly affine) hyperbolic curve.

The surjectivity of $\rho_{k_0,\Vc}$ is much less understood, even in the case $k_0 = \Q$.
For instance, if one takes $\Vc = \{\mathcal{M}_{0,n}\}_n$ with the ``connecting morphisms'' (or certain smaller subcategories) then $\Out(\bar\pi_1|_\Vc)$ is the intensively studied \emph{Grothendieck-Teichm\"uller group}.
The surjectivity of $\rho_{\Q,\Vc}$ in this case (and for other subcategories of the Teichm\"uller modular tower) is still a \emph{major open question} in modern Galois theory.

In any case, the original question of {\sc Ihara}, and the subsequent conjecture of {\sc Oda-Matsumoto} \cite{Matsumoto:1997}, predict that $\rho_{\Q,\Vc}$ is an isomorphism in the case where $\Vc = \Var_\Q$.
In 1999, {\sc Pop} proved in an unpublished manuscript that $\rho_{k_0,\Vc}$ is an isomorphism for more general fields $k_0$ in the case where $\Vc = \Var_{k_0}$.
{\sc Pop's} proof was eventually released in \cite{Pop:2014}, along with a stronger \emph{pro-$\ell$ abelian-by-central} variant which we now summarize.

\subsection{The Pro-$\ell$ Abelian-by-Central I/OM}
\label{subsection: intro / pro-ell-IOM}

In order to get closer to the spirit of {\sc Ihara's} original question of finding a combinatorial description of the absolute Galois group of $\Q$, it makes sense to replace the geometric fundamental group by certain \emph{smaller functorial quotients.}
The first such strengthening was formulated and proved by {\sc Pop} \cite{Pop:2014} who uses the maximal \emph{pro-$\ell$ abelian-by-central} quotient of $\bar\pi_1$.
We will give only a very brief summary of the pro-$\ell$ abelian-by-central context, since the purpose of this paper is to develop a \emph{mod-$\ell$ variant/strengthening} of loc.~cit.

Let $\ell$ be a prime which is different from $\Char k_0$.
For $X \in \Var_{k_0}$, we consider $\bar\pi_1^\ell(X)$ the maximal pro-$\ell$ quotient of $\bar\pi_1(X)$.
Next, consider the maximal {\bf pro-$\ell$ abelian} resp. {\bf pro-$\ell$ abelian-by-central} quotients of $\bar\pi_1(X)$, which are defined and denoted as follows:
\[ \Pi^A(X) := \frac{\bar\pi_1^\ell(X)}{[\bar\pi_1^\ell(X),\bar\pi_1^\ell(X)]} \ \text{ resp. } \ \Pi^C(X) = \frac{\bar\pi_1^\ell(X)}{[\bar\pi_1^\ell(X),[\bar\pi_1^\ell(X),\bar\pi_1^\ell(X)]]}. \]
Note that both $\Pi^A(X)$ and $\Pi^C(X)$ are functorial in $X$.
Moreover, note that $\Z_\ell^\times$ acts on $\Pi^A(X)$ by left-multiplication.
It turns out (by general group-theoretical facts) that this action lifts to $\Pi^C(X)$.
We consider the set of \emph{liftable automorphisms} of $\Pi^A(X)$, defined by
\[ \Aut^{\rm C}(\Pi^A(X)) := \image(\Aut(\Pi^C(X)) \rightarrow \Aut(\Pi^A(X))), \]
as well as the quotient $\UAut^{\rm C}(\Pi^A(X)) := \Aut^{\rm C}(\Pi^A(X))/\Z_\ell^\times$ by the canonical action of $\Z_\ell^\times$.

Similar to the absolute context, for every $X \in \Var_{k_0}$, one has canonical representations
\[ \rho^{\rm C}_{k_0,X} : \Galk \rightarrow \Aut^{\rm C}(\Pi^A(X)) \rightarrow \UAut^{\rm C}(\Pi^A(X)). \]
Given a subcategory $\Vc$ of $\Var_{k_0}$, one defines $\Aut^{\rm C}(\Pi^A|_\Vc)$ as the set of systems $(\phi_X)_{X \in \Vc}$, $\phi_X \in \Aut^{\rm C}(\Pi^A(X))$ which are compatible with morphisms from $\Vc$, similar to the absolute context.
We may also consider the quotient by the canonical action of $\Z_\ell^\times$:
\[ \UAut^{\rm C}(\Pi^A|_\Vc) := \Aut^{\rm C}(\Pi^A|_\Vc)/\Z_\ell^\times \]
similar to our definition of $\UAut^{\rm C}(\Pi^A(X))$.
As before, we obtain canonical Galois representations
\[ \rho^{\rm C}_{k_0,\Vc} : \Galk \rightarrow \Aut^{\rm C}(\Pi^A|_\Vc) \rightarrow \UAut^{\rm C}(\Pi^A|_\Vc). \]
With this notation, the \emph{pro-$\ell$ abelian-by-central I/OM} is completely analogous to the absolute I/OM, as it refers to the following general question.

\vskip 5pt
\noindent{\bf The Pro-$\ell$ Abelian-by-Central I/OM:}
{\em For which fields $k_0$ and subcategories $\Vc$ of $\Var_{k_0}$ as above, is the Galois representation 
\[ \rho^{\rm C}_{k_0,\Vc} : \Galk \rightarrow \UAut^{\rm C}(\Pi^A|_\Vc) \]
an isomorphism?}
\vskip 5pt

{\sc Pop} shows in \cite{Pop:2014} that $\rho^{\rm C}_{k_0,\Vc}$ is an isomorphism for certain categories $\Vc = \Vc_X$ which are similar to the categories of the form $\Ucal_\abf$ (see \S\ref{subsection: intro / main-result-birational-systems} below).
Loc. cit. also shows that $\rho^{\rm C}_{k_2,\Vc}$ is an isomorphism for so-called ``connected'' subcategories $\Vc$ of $\Var_{k_0}$.
This ``connectedness'' condition holds, in particular, for $\Var_{k_0}$ itself, as well as for the category of geometrically-integral smooth quasi-projective $k_0$-varieties.
Although this notion of connectedness is somewhat technical, we note that it is similar to what we call ``$2$-connected'' in \S\ref{subsection: intro / main-result-connected} below.

The pro-$\ell$ abelian-by-central context gets closer to a truly combinatorial description of absolute Galois groups than the absolute context.
However, the groups considered in this context, $\Pi^A(X)$ and $\Pi^C(X)$, are still quite large, as they still have a non-trivial profinite topology which plays a very significant role.

In this paper, we develop a further strengthening of the I/OM, by considering the {\bf mod-$\ell$ abelian-by-central} quotient of $\bar\pi_1(X)$.
As mentioned before, this quotient of $\bar\pi_1$ which we consider is the \emph{smallest} functorial pro-$\ell$ quotient of $\bar\pi_1$ which remains non-abelian; in particular, it is a quotient of $\Pi^C(X)$, and it can be seen as a purely combinatorial (i.e. finite and discrete) object.
Thus, the mod-$\ell$ abelian-by-central context is essentially the best one could hope for, in the profinite context.
In more broad terms, considering the I/OM with other variants of the geometric fundamental group could lead to further substantial developments in various facets of Galois theory.

In this paper, we will prove that the mod-$\ell$ abelian-by-central I/OM holds for so-called ``$5$-connected'' subcategories $\Vc$ of $\Var_{k_0}$.
Similar to {\sc Pop's} notion of connectedness, our notion of $5$-connectedness applies to $\Var_{k_0}$ itself, as well as to the full category of geometrically-integral smooth quasi-projective $k_0$-varieties.
However, for the time being, it is unclear whether the mod-$\ell$ I/OM holds true for $d$-connected categories with $1 < d < 5$.

It is particularly important to note that {\sc Pop} \cite{Pop:2014} uses ideas related to  \emph{Bogomolov's Program} \cite{Bogomolov:1991} in birational anabelian geometry, which considers pro-$\ell$ abelian-by-central Galois groups of higher-dimensional function fields over $k$.
In a few words, the proof of the pro-$\ell$ abelian-by-central I/OM first reduces to a \emph{birational} context, and eventually uses both the \emph{local theory} \cite{Bogomolov:2002}, \cite{Pop:2010} and \emph{global theory} \cite{Pop:2007} from pro-$\ell$ abelian-by-central birational anabelian geometry.

The initial step in our proof of the mod-$\ell$ abelian-by-central I/OM is more-or-less the same as the pro-$\ell$ context, in the sense that we will reduce the mod-$\ell$ I/OM to a birational context.
We will then use techniques from the \emph{mod-$\ell$ abelian-by-central variant} of Bogomolov's Program, including both the \emph{mod-$\ell$ local theory} \cite{Pop:2010a}, \cite{Topaz:2013a}, \cite{Topaz:2015} and the \emph{mod-$\ell$ global theory} \cite{Topaz:2014}.

Because of this strategy, we run into precisely the same problems/difficulties that arise when one passes from the pro-$\ell$ to the mod-$\ell$ abelian-by-central variants of Bogomolov's Program.
These fundamental differences between the pro-$\ell$ and mod-$\ell$ context were described in detail in the introduction of \cite{Topaz:2014}, and we refer the reader there for these details.
Nevertheless, we mention here that, in the pro-$\ell$ context, one eventually uses the \emph{Fundamental Theorem of Projective Geometry} applied to an infinite-dimensional $k$-projective space embedded in $\H^1(K,\Z_\ell(1))$, where $K$ is a function field over $k$.
The main difficulty in the mod-$\ell$ context is that $\H^1(K,\Z/\ell(1))$ \emph{contains no such $k$-projective space}.
Therefore, our proof of the mod-$\ell$ I/OM is fundamentally different than the proof of the pro-$\ell$ variant.
See \S\ref{section: summary} for a detailed summary of the proof of the mod-$\ell$ I/OM, and see the introduction of \cite{Topaz:2014} for more on the comparison between the pro-$\ell$ and mod-$\ell$ contexts.
We now introduce the mod-$\ell$ abelian-by-central context in detail.

\subsection{The Mod-$\ell$ Abelian-by-Central Quotients}
\label{subsection: intro / mod-ell-category}

For a profinite group $\Gc$, we let $\Gc^{(i)}$ denote the $i$-th term of the {\bf mod-$\ell$ Zassenhaus filtration} of $\Gc$.
We will only need to consider the first two non-trivial terms of this filtration, which are defined explicitly as follows:
\begin{enumerate}
  \item $\Gc^{(2)} := [\Gc,\Gc] \cdot \Gc^\ell$.
  \item $\Gc^{(3)} := [\Gc,\Gc^{(2)}] \cdot \Gc^{\delta \cdot \ell}$ where $\delta = 1$ if $\ell \neq 2$ and $\delta = 2$ if $\ell = 2$.
\end{enumerate}
We will consistently denote the quotients $\Gc/\Gc^{(*)}$ for $* = 2,3$ as follows:
\[ \Gc^a := \Gc/\Gc^{(2)}, \ \ \Gc^c := \Gc/\Gc^{(3)}. \]
We call $\Gc^a$ resp. $\Gc^c$ the {\bf mod-$\ell$ abelian} resp. {\bf mod-$\ell$ abelian-by-central} quotient of $\Gc$.
Indeed, note that $\Gc^a$ is an $\ell$-elementary abelian pro-$\ell$ group and $\Gc^c$ is a central extension of $\Gc^a$ by an $\ell$-elementary abelian pro-$\ell$ group.

\begin{remark}
  The primary reason for the distinction between odd/even $\ell$ in the definition of $\Gc^{(3)}$ is that we require $\Gc^c = \Gc/\Gc^{(3)}$ to be \emph{non-abelian}.
  Nevertheless, $\Gc^c$ is the smallest pro-$\ell$ quotient of $\Gc$ which is \emph{functorial} in $\Gc$, and which is non-abelian.
  In cohomological terms, this distinction between odd/even $\ell$ is related to the fact that the mod-$2$ Bockstein morphism agrees with the Steenrod square ${\rm Sq}^1$, whereas no such relationship exists for odd $\ell$.
  For more on the mod-$\ell$ Zassenhaus filtration and its connection with mod-$\ell$ (group) cohomology, we refer the reader to \cite{Efrat:2014} and \cite{Efrat:2015}.
\end{remark}

Suppose now that $\sigma,\tau \in \Gc^a$ are given, and choose lifts $\tilde\sigma,\tilde\tau \in \Gc^c$ of $\sigma,\tau$.
Since $\Gc^c$ is a central extension of $\Gc^a$, it follows that the commutator
\[ [\sigma,\tau] := \tilde\sigma^{-1}\tilde\tau^{-1}\tilde\sigma\tilde\tau \]
depends only on $\sigma,\tau \in \Gc^a$ and not on the choice of lifts $\tilde\sigma,\tilde\tau \in \Gc^c$ of $\sigma,\tau$.
Next, define the {\bf completed wedge product} of $\Gc^a$ with itself as 
\[ \widehat\wedge^2 (\Gc^a) = \varprojlim_H \wedge^2(\Gc^a/H), \]
where $H$ varies over the open subgroup of $\Gc^a$, and where $\wedge^2 (M) = M \otimes M/\langle x \otimes x : x \in M \rangle$ for a discrete $\Z/\ell$-module $M$.
Then the commutator defined above extends linearly to define a canonical morphism 
\[ [\bullet,\bullet] : \widehat\wedge^2(\Gc^a) \rightarrow \Gc^{(2)}/\Gc^{(3)}, \]
and we define $\Rcal(\Gc)$ as the kernel of this canonical map.
Note that one has $\Gc^a = (\Gc^c)^a$ and $\Rcal(\Gc) = \Rcal(\Gc^c)$, so the datum $(\Gc^a,\Rcal(\Gc))$ is completely determined by the quotient $\Gc^c$.

Suppose now that $\Gc_1,\Gc_2$ are two profinite groups, and that $f : \Gc_1^a \rightarrow \Gc_2^a$ is a morphism.
In this context, we say that $f$ is {\bf compatible with $\Rcal$} if the induced map 
\[ \widehat\wedge^2 (f) : \widehat\wedge^2(\Gc_1^a) \rightarrow \widehat\wedge^2(\Gc_2^a) \]
restricts to a map $\Rcal(\Gc_1) \rightarrow \Rcal(\Gc_2)$.
We write $\Homc(\Gc_1^a,\Gc_2^a)$ for the collection of morphisms $f : \Gc_1^a \rightarrow \Gc_2^a$ which are compatible with $\Rcal$.
Similarly, for a profinite group $\Gc$, we write $\Autc(\Gc^a)$ for the collection of automorphisms of $\Gc^a$ which are compatible with $\Rcal$ and whose inverse is also compatible with $\Rcal$.

The definitions above can be summarized by defining the {\bf mod-$\ell$ abelian-by-central category}, denoted $\AbC$, to be the category whose objects consist of pairs $(\Gc^a,\Rcal)$ where $\Gc^a$ is a profinite group such that $(\Gc^a)^{(2)} = 1$, and $\Rcal$ is a closed subgroup of $\widehat\wedge^2(\Gc^a)$.
A morphism from $(\Gc_1^a,\Rcal_1)$ to $(\Gc_2^a,\Rcal_2)$ in $\AbC$ is simply a morphism $f : \Gc_1^a \rightarrow \Gc_2^a$ of profinite groups such that the induced map 
\[ \widehat\wedge^2 (f) : \widehat\wedge^2(\Gc_1^a) \rightarrow \widehat\wedge^2(\Gc_2^a) \]
restricts to a map $\Rcal_1 \rightarrow \Rcal_2$.
Finally, we have a canonical functor 
\[ (\bullet)^\ac : \Profinite \rightarrow \AbC \]
defined on objects by $\Gc^\ac := (\Gc^a,\Rcal(\Gc))$.
The $\Rcal$-compatible morphisms/automorphisms are then given by the morphisms/automorphism in $\AbC$:
\begin{itemize}
  \item $\Homc(\Gc_1^a,\Gc_2^a) = \Hom_{\AbC}(\Gc_1^\ac,\Gc_2^\ac)$.
  \item $\Autc(\Gc^a) = \Aut_{\AbC}(\Gc^\ac)$.
\end{itemize}
Finally, as noted above, for any profinite group, one has $\Gc^\ac = (\Gc^c)^\ac$ as objects of $\AbC$.
In other words, the functor $(\bullet)^\ac$ factors through the endofunctor $\Gc \mapsto \Gc^c$ of $\Profinite$.

Similar to the pro-$\ell$ context, we will consistently use \emph{underlines} to denote the process of modding out by the left-multiplication action of $(\Z/\ell)^\times$.
For instance, note that $(\Z/\ell)^\times$ acts on $\Gc^a$ by multiplication on the left, and that this action is always compatible with $\Rcal$.
Thus, $(\Z/\ell)^\times$ acts on $\Homc(\Gc_1^a,\Gc_2^a)$ and we write
\[ \UHomc(\Gc_1^a,\Gc_2^a) = \Homc(\Gc_1^a,\Gc_2^a)/(\Z/\ell)^\times. \]
Similarly, $(\Z/\ell)^\times$ acts on $\Autc(\Gc^a)$ and we write $\UAutc(\Gc^a) := \Autc(\Gc^a)/(\Z/\ell)^\times$.

\subsection{The Mod-$\ell$ Abelian-by-Central I/OM}
\label{subsection: intro / mod-ell-IOM}

For any $X \in \Var_{k_0}$, we consider the {\bf mod-$\ell$ abelian} resp. {\bf mod-$\ell$ abelian-by-central} geometric fundamental groups of $X$, which are defined and denoted as follows:
\[ \pi^a(X) := (\bar\pi_1(X))^a \ \text{ resp. } \ \pi^c(X) := (\bar\pi_1(X))^c. \]
We will also consider the associated \emph{abelian-by-central datum}
\[ (\bar\pi_1(X))^\ac = (\pi^a(X),\Rcal(\bar\pi_1(X))) = (\pi^a(X),\Rcal(\pi^c(X))) = (\pi^c(X))^\ac \]
as discussed above.

\begin{remark}
  \label{remark: intro / etale-cup-combinatorial}
  In practice, the object $(\bar\pi_1(X))^\ac$ can be explicitly computed using cohomology with $\Z/\ell$ coefficients.
  Indeed, for any $X \in \Var_{k_0}$, one has $\pi^a(X) = \H_{\text{\'et}}^1(\bar X,\Z/\ell)^\vee$, and the inclusion $\Rcal := \Rcal(\pi^c(X)) \hookrightarrow \wedge^2 (\pi^a(X))$ is dual to (the image of) 
  \[ \wedge^2 (\H_{\text{\'et}}^1(\bar X,\Z/\ell)) \xrightarrow{\cup} \H_{\text{\'et}}^2(\bar X,\Z/\ell). \]
  In other words, one can completely compute the object $(\bar\pi_1(X))^\ac$ using the \emph{finite-dimensional} $\Z/\ell$-vector space $\H_{\text{\'et}}^1(\bar X,\Z/\ell)$, along with the kernel of the cup-product map above. 
  In particular, $\Aut^c(\pi^a(X))$ is a subgroup of the \emph{finite group} $\Aut(\pi^a(X)) = \operatorname{GL}_{\Z/\ell}(\pi^a(X))$.
  In this respect, one may view $(\bar\pi_1(X))^\ac$ as a combinatorial object, defined by a \emph{finite} vector space with some additional linear data.
\end{remark}

Note that for every $X \in \Var_{k_0}$, the Galois group $\Galk$ acts on $\pi^c(X)$ by outer-automorphisms, since $\pi^c(X)$ is a characteristic quotient of $\bar\pi_1(X)$.
Since $\pi^a(X)$ is a further characteristic quotient of $\pi^c(X)$, we see that $\Galk$ acts on $(\pi^c(X))^\ac = (\pi^a(X),\Rcal(\pi^c(X)))$ as an object of $\AbC$, so that we obtain canonical Galois representations
\[ \rho_{k_0,X}^c : \Galk \rightarrow \Autc(\pi^a(X)) \rightarrow \UAutc(\pi^a(X)).\]

Suppose now that $\Vc$ is an essentially small subcategory of $\Var_{k_0}$.
In this context, we define $\Autc(\pi^a|_\Vc)$ to be the automorphism group of the functor
\[ \Vc \hookrightarrow \Var_{k_0} \xrightarrow{\bar\pi_1} \Profinite \xrightarrow{(\bullet)^\ac} \AbC. \]
In other words, $\Autc(\pi^a|_\Vc)$ consists of system $(\phi_X)_{X \in \Vc} \in \Aut(\pi^a|_\Vc)$ with $\phi_X \in \Autc(\pi^a(X))$ for all $X \in \Vc$, such that the $\phi_X$ are compatible with morphisms arising from $\Vc$.
As before, $(\Z/\ell)^\times$ acts on $\Autc(\pi^a|_\Vc)$ in a canonical way by left-multiplication, and we will denote the quotient by this action with an underline as $\UAutc(\pi^a|_\Vc) := \Autc(\pi^a|_\Vc)/(\Z/\ell)^\times$.
Finally, we can combine all of the $\rho_{k_0,X}^c$ as before to obtain canonical Galois representations
\[ \rho_{k_0,\Vc}^c : \Galk \rightarrow \Autc(\pi^a|_\Vc) \twoheadrightarrow \UAutc(\pi^a|_\Vc). \]
With this notation, the \emph{mod-$\ell$ abelian-by-central I/OM} is completely analogous to the absolute and pro-$\ell$ contexts, and it refers to the following general question.

\vskip 5pt
\noindent{\bf The Mod-$\ell$ Abelian-by-Central I/OM:}
{\em For which fields $k_0$ and subcategories $\Vc$ of $\Var_{k_0}$ as above, is the Galois representation 
\[ \rho^c_{k_0,\Vc} : \Galk \rightarrow \UAutc(\pi^a|_\Vc) \]
an isomorphism?}
\vskip 5pt

\subsection{The Main Result -- Birational Systems}
\label{subsection: intro / main-result-birational-systems}

With the preparation above, we will now introduce the subcategories of $\Var_{k_0}$ which we consider in this paper.
Let $X \in \Var_{k_0}$ have dimension $\geq 1$, and let $\Ucal_X$ be a basis of open neighborhoods of the generic point of $X$.
We will always consider $\Ucal_X$ as a subcategory of $\Var_{k_0}$ whose objects are the elements of $\Ucal_X$ and whose morphisms are the inclusions among them as open subsets of $X$.
Moreover, we write $\Ucal_X^+ = \Ucal_X \cup \{X\}$ for the basis of open neighborhoods of the generic point of $X$ which also includes $X$ as a terminal object.

Let $X \in \Var_{k_0}$ be an object.
To simplify the exposition, we will say that a subcategory $\Ucal_X$ of $\Var_{k_0}$ is a {\bf birational system of $X$} if $\Ucal_X$ is a basis of open neighborhoods of the generic point of $X$.
We will use the notation $\Ucal_X^+$ as above to denote the existence of $X$ as a terminal object.
In other words, while a birational system $\Ucal_X$ of $X$ need not have a terminal object, the birational system $\Ucal^+_X$ always has $X$ as a terminal object.
We say that $\Ucal$ is a {\bf birational system} if $\Ucal$ is a birational system of $X$ for some $X \in \Var_{k_0}$.
The {\bf dimension} of a birational system $\Ucal$, denoted $\dim \Ucal$, is defined to be the dimension of one (hence all) of the objects in $\Ucal$.
In particular, $\dim\Ucal_X = \dim X$.

Now let $r \geq 0$ be given, and let $\abf = (a_1,\ldots,a_r)$ be a (possibly empty) finite tuple of elements $a_i \in k_0^\times$.
We denote the complement of $\abf$ in $\Gbb_m$ (over $k_0$) as
\[ \U_\abf := \Gbb_{m,k_0} \smallsetminus \{a_1,\ldots,a_r\}. \]
For instance, $\U_\abf = \Gbb_m$ if $\abf = \varnothing$ is empty, and $\U_\abf = \Pbb^1 \smallsetminus \{0,1,\infty\}$ is the tripod if $\abf = (1)$.

We will furthermore consider a small category $\Ucal_\abf$ which is constructed from a positive-dimensional birational system $\Ucal$ and $\U_\abf$ as follows:
\begin{enumerate}
  \item The objects of $\Ucal_\abf$ are given by $\Ucal \cup \{\U_\abf\}$.
  \item The morphisms in $\Ucal_\abf$ are the inclusions among the objects in $\Ucal$, the identity on $\U_\abf$, and all of the dominant morphisms $U \rightarrow \U_\abf$ for $U \in \Ucal$.
\end{enumerate}
Our first main theorem concerning the mod-$\ell$ I/OM as stated above shows the bijectivity of the Galois representation for such categories, if the birational system has sufficiently large dimension.

\begin{maintheorem}
\label{maintheorem: intro / main-result / iom-main-opens}
  Let $k_0$ be an infinite perfect field of characteristic $\neq \ell$.
  Let $\Ucal$ be a birational system of dimension $\geq 5$, and let $\abf = (a_1,\ldots,a_r)$ be a (possibly empty) finite tuple of elements of $k_0^\times$.
  Then the canonical Galois representation 
  \[ \rho^c_{k_0,\Ucal_\abf} : \Galk \rightarrow \UAutc(\pi^a|_{\Ucal_\abf}) \]
  is an isomorphism.
\end{maintheorem}

\subsection{The Main Result -- Connected Categories}
\label{subsection: intro / main-result-connected}

We will now introduce the precise notion of a ``$d$-connected category.''
Suppose that $\Vc$ is an essentially small subcategory of $\Var_{k_0}$ which contains a positive-dimensional object.
Let $\Ucal_1,\Ucal_2$ be two positive-dimensional birational systems.
In this context, we say that {\bf $\Ucal_1$ dominates $\Ucal_2$ in $\Vc$} provided that $\Vc$ contains $\Ucal_1$ and $\Ucal_2$, and that the following holds:
\begin{itemize}
  \item If $\dim \Ucal_2 > 1$: For all $V \in \Ucal_2$, there exists some $U \in \Ucal_1$ such that $\Vc$ contains a dominant morphism $U \rightarrow V$.
  \item If $\dim \Ucal_2 = 1$: For all $V \in \Ucal_2$, there exists some $U \in \Ucal_1$ such that $\Vc$ contains a dominant morphism $U \rightarrow V$ with geometrically integral fibers.
\end{itemize}

Next suppose that $\Ucal,\Ucal_1,\Ucal_2$ are three positive-dimensional birational systems.
In this context, we say that {\bf $\Ucal$ attaches $\Ucal_1$ to $\Ucal_2$ in $\Vc$} if the following hold:
\begin{enumerate}
  \item The category $\Vc$ contains $\Ucal_\abf$ for some finite tuple $\abf$ of elements of $k_0^\times$.
  \item The birational system $\Ucal$ dominates both $\Ucal_1$ and $\Ucal_2$ in $\Vc$.
\end{enumerate}

Now let $d \geq 1$ be given, and let $\Ucal_0$ and $\Ucal_{2r}$ be two birational systems.
We will say that $\Ucal_0$ and $\Ucal_{2r}$ are {\bf $d$-connected in $\Vc$} if there exist birational systems $\Ucal_1,\ldots,\Ucal_{2r-1}$ such that, for all $i = 0,\ldots,r-1$, the following conditions hold:
\begin{enumerate}
  \item One has $\dim \Ucal_{2i+1} \geq d$. 
  \item The birational system $\Ucal_{2i+1}$ attaches $\Ucal_{2i}$ to $\Ucal_{2i+2}$ in $\Vc$.
\end{enumerate}
Finally, we say that $\Vc$ is {\bf $d$-connected} if the following conditions hold:
\begin{enumerate}
  \item $\Vc$ is essentially small and it contains a positive-dimensional object.
  \item For every object $X$ of $\Vc$, there exists some birational system $\Ucal_X^+$ of $X$ which contains $X$ as the terminal object, such that $\Ucal_X^+$ is contained in $\Vc$.
  \item Any two birational systems $\Ucal_0,\Ucal_{2r}$ which are contained in $\Vc$ are $d$-connected in $\Vc$.
\end{enumerate}

Although the precise definition of a $d$-connected category is somewhat complicated, we note that, for example, both the full category $\Var_{k_0}$, and the full subcategory of all geometrically-integral smooth quasi-projective $k_0$-varieties, are $d$-connected for all $d \geq 1$.
Furthermore, if $d' \geq d \geq 1$, we note that $\Vc$ being $d'$-connected implies that $\Vc$ is $d$-connected.
Our next main theorem concerns the mod-$\ell$ I/OM for $5$-connected varieties.

\begin{maintheorem}
\label{maintheorem: intro / main-result / iom-main-connected}
  Let $k_0$ be an infinite perfect field of characteristic $\neq \ell$, and let $\Vc$ be a subcategory of $\Var_{k_0}$ which is $5$-connected.
  Then the canonical Galois representation 
  \[ \rho^c_{k_0,\Vc} : \Galk \rightarrow \UAutc(\pi^a|_\Vc) \]
  is an isomorphism.
\end{maintheorem}

Theorems \ref{maintheorem: intro / main-result / iom-main-opens} and \ref{maintheorem: intro / main-result / iom-main-connected} together form the mod-$\ell$ variant/strengthening (in dimension $\geq 5$) of the main results from \cite{Pop:2014}, where the absolute and pro-$\ell$ I/OM are proven.
And as mentioned above, this mod-$\ell$ context is optimal with respect to functorial pro-$\ell$ quotients of $\bar\pi_1$ which remain non-abelian.
The recent work \cite{Pop:2016} proves yet another refinement of \cite{Pop:2014} by considering \emph{coarser} categories of varieties, but still necessarily remaining in the pro-$\ell$ context.
Therefore, the present paper and \cite{Pop:2016} both refine the results of \cite{Pop:2014}, while these two refinements seem to be in orthogonal directions.

\subsection{Birational-Galois Variant}
\label{subsection: intro / galois-variant}

Let $X \in \Var_{k_0}$ be given and let $\Ucal = \Ucal_X$ be a birational system for $X$, as defined above.
Recall that elements of $\Autc(\pi^a|_\Ucal)$ consist of systems of elements $(\phi_U)_{U \in \Ucal}$, where $\phi_U \in \Autc(\pi^a(U))$ are compatible with the morphisms arising from $\Ucal$.
By taking the projective limit over $\Ucal$, one obtains an element of $\Autc((\Gal_{k(X)})^a)$.
Moreover, since $X$ is geometrically normal, it follows that the induced canonical map $\Autc(\pi^a(X)) \rightarrow \Autc((\Gal_{k(X)})^a)$ is injective.
Therefore, in order to prove Theorem \ref{maintheorem: intro / main-result / iom-main-opens}, it makes sense to first develop a \emph{birational variant} of that theorem, which deals with quotients of \emph{absolute Galois groups of function fields} as opposed to quotients of fundamental groups of varieties.
Therefore, the main focus of this paper is to develop and prove \emph{birational variants} of our main theorems, and we now introduce the appropriate notation and terminology.

Suppose that $K_0$ is a regular function field over $k_0$, and let $K = K_0 \cdot k = K_0 \otimes_{k_0} k$ denote the base-change of $K_0$ to $k$.
Recall that $\Galk$ acts on $K = K_0 \otimes_{k_0} k$ in the obvious way, and that one has a canonical isomorphism
\[ \Galk \xrightarrow{\cong} \Gal(K|K_0). \]

We denote by $\Gc_K := \Gal(K(\ell)|K)$ the maximal pro-$\ell$ Galois group of $K$.
We also consider its mod-$\ell$ abelian resp. mod-$\ell$ abelian-by-central quotients $\Gc_K^a$ resp. $\Gc_K^c$, and the associated object 
\[ (\Gc_K)^\ac = (\Gc_K^c)^\ac = (\Gc_K^a,\Rcal(\Gc_K^c)) \]
of $\AbC$, as introduced in \S\ref{subsection: intro / mod-ell-category}.
Following the notation above, we denote the automorphism group of $\Gc_K^\ac$ by $\Autc(\Gc_K^a)$, and we write
\[ \UAutc(\Gc_K^a) := \Autc(\Gc_K^a)/(\Z/\ell)^\times \]
for its quotient by the canonical action of $(\Z/\ell)^\times$.
Since the projection $\Gc_K^c \rightarrow \Gc_K^a$ is functorial in $K$, it follows that $\Galk$ acts on $\Gc_K^\ac$ as an object of $\AbC$.
In other words, we obtain canonical Galois representations
\[ \rho^c_{k_0,K_0} : \Galk \rightarrow \Autc(\Gc_K^a) \rightarrow \UAutc(\Gc_K^a). \]

Suppose now that $\abf$ is a (possibly empty) finite tuple of elements of $k_0^\times$, and recall that we write $\U_\abf := \Gbb_m \smallsetminus \abf$.
Note that every non-constant $t \in K_0^\times$ induces a dominant morphism $U \rightarrow \U_\abf$ for some $U \in \Var_{k_0}$ such that $k_0(U) = K_0$.
Thus, for every non-constant $t \in K_0$, we obtain a canonical morphism of $\ell$-elementary abelian pro-$\ell$ groups
\[ \pi_t : \Gc_K^a \twoheadrightarrow \pi^a(U) \rightarrow \pi^a(\U_\abf). \]
Let $\Hcal_t$ denote the kernel of $\pi_t : \Gc_K^a \rightarrow \pi^a(\U_\abf)$.
We will write $\Autc_\abf(\Gc_K^a)$ for the subgroup of $\Autc(\Gc_K^a)$ consisting of elements $\phi \in \Autc(\Gc_K^a)$ such that $\phi\Hcal_t = \Hcal_t$ for all non-constant $t \in K_0^\times$.
Note that the canonical action of $(\Z/\ell)^\times$ on $\Autc(\Gc_K^a)$ restricts to an action on the subgroup $\Autc_\abf(\Gc_K^a)$, and we will write 
\[ \UAutc_\abf(\Gc_K^a) := \Autc_\abf(\Gc_K^a)/(\Z/\ell^\times) \]
for the quotient of this action.

Since the Galois action is clearly compatible with the morphisms $\pi_t$ for non-constant $t \in K_0^\times$, we see that any element of $\Autc(\Gc_K^a)$ resp. $\UAutc(\Gc_K^a)$ which arises from $\Galk$ must actually be contained in $\Autc_\abf(\Gc_K^a)$ resp. $\UAutc_\abf(\Gc_K^a)$.
In other words, we obtain canonical Galois representations
\[ \rho_{k_0} : \Galk \rightarrow \Autc_\abf(\Gc_K^a) \rightarrow \UAutc_\abf(\Gc_K^a). \]
Our Birational-Galois variant of Theorem \ref{maintheorem: intro / main-result / iom-main-opens} is about this canonical morphism.

\begin{maintheorem}
\label{maintheorem: intro / galois-variant / galois-main}
  Let $k_0$ be an infinite perfect field of characteristic $\neq \ell$.
  Let $K_0$ be a regular function field over $k_0$ of transcendence degree $\geq 5$, and put $K = K_0 \cdot k$.
  Let $\abf$ be an arbitrary (possibly-empty) finite tuple of elements of $k_0^\times$.
  Then the canonical map
  \[ \rho_{k_0} : \Galk \rightarrow \UAutc_\abf(\Gc_K^a) \]
  is an isomorphism.
\end{maintheorem}

\subsection{Birational-Milnor Variant}
\label{subsection: intro / milnor-variant}

It turns out that it will be more convenient to work with the \emph{Kummer Dual} of Theorem \ref{maintheorem: intro / galois-variant / galois-main}.
While the Kummer dual of $\Gc_K^a$ is $K^\times/\ell$, it will be a consequence of the Merkurjev-Suslin Theorem \cite{Merkurjev:1982} that the ``dual'' of the object $\Gc_K^\ac$ can be considered as the mod-$\ell$ Milnor K-ring of $K$, which we denote by $\k_*(K)$.
Thus, the primary focus of this paper will be to prove a \emph{Milnor variant} of Theorem \ref{maintheorem: intro / galois-variant / galois-main}, which deals with the mod-$\ell$ Milnor K-ring of the function field $K$.

We will recall the definition of $\k_*(K)$ in \S\ref{section: milnor}, but we note here that one has 
\[ \k_1(K) = K^\times/\ell \]
and that one has a canonical surjective morphism of $\Z/\ell$-algebras
\[ \operatorname{T}_*(K^\times/\ell) \twoheadrightarrow \k_*(K) \]
where $\operatorname{T}_*(K^\times/\ell)$ denotes the tensor algebra of $K^\times/\ell$ considered as a vector space over $\Z/\ell$ which is concentrated in degree $1$.
We denote by $\Autm(\k_1(K))$ the collection of automorphisms of $\k_1(K)$ which extend to an automorphism of $\k_*(K)$.
Similar to the above, we have a canonical action of $(\Z/\ell)^\times$ on $\k_1(K)$ by left multiplication, and we put
\[ \UAutm(\k_1(K)) = \Autm(\k_1(K))/(\Z/\ell)^\times. \]

Let $\abf = (a_1,\ldots,a_r)$ be a possibly empty finite tuple of elements of $k_0^\times$ as above.
For $x \in K^\times$, we write $\{x\}_K$ for the image of $x$ in $\k_1(K) = K^\times/\ell$.
We write $\Autm_\abf(\k_1(K))$ for the subgroup of all automorphisms $\phi \in \Autm(\k_1(K))$ such that for all non-constant $t \in K_0^\times$, the automorphism $\phi$ restricts to an automorphism of the subgroup
\[ \langle \{t\}_K,\{t-a_1\}_K, \ldots, \{t-a_r\}_K \rangle \]
of $\k_1(K)$.
Finally, we define $\UAutm_\abf(\k_1(K)) := \Autm_\abf(\k_1(K))/(\Z/\ell)^\times$ similar to the above.
It turns out that the group $\UAutm_\abf(\k_1(K))$ can be viewed as the ``Kummer Dual'' of the group $\UAutc_\abf(\Gc_K^a)$ considered in \S\ref{subsection: intro / galois-variant}.

As before, we have a canonical action of $\Galk$ on $\k_*(K)$.
Moreover, this action is compatible with subgroups of $\k_1(K)$ of the form 
\[ \langle \{t\}_K,\{t-a_1\}_K,\ldots,\{t-a_r\}_K \rangle \]
for all non-constant $t \in K_0^\times$, and all $a_1,\ldots,a_r \in k_0^\times$.
To summarize, for a (possibly empty) finite tuple $\abf$ of elements of $k_0^\times$, we obtain canonical Galois representations
\[ \rho^{\rm M}_{k_0,K_0} : \Galk \rightarrow \Autm_\abf(\k_1(K)) \rightarrow \UAutm_\abf(\k_1(K)) \]
which are the primary focus of the following ``Milnor-Variant'' of Theorem \ref{maintheorem: intro / galois-variant / galois-main}.

\begin{maintheorem}
\label{maintheorem: intro / milnor-variant / milnor-main}
  Let $k_0$ be an infinite perfect field of characteristic $\neq \ell$.
  Let $K_0$ be a regular function field over $k_0$ of transcendence degree $\geq 5$, and put $K = K_0 \cdot k$.
  Let $\abf$ be an arbitrary (possibly-empty) finite tuple of elements of $k_0^\times$.
  Then the canonical map
  \[ \rho^{\rm M}_{k_0,K_0} : \Galk \rightarrow \UAutm_\abf(\k_1(K)) \]
  is an isomorphism.
\end{maintheorem}

We give a brief description of the proof of Theorem \ref{maintheorem: intro / milnor-variant / milnor-main}, as it pertains to the mod-$\ell$ anabelian tools mentioned above.
A much more detailed summary is given in \S\ref{section: summary}.
In the above context, let $\sigma \in \Autm_\abf(\k_1(K))$ be given.
First, we will use the mod-$\ell$ \emph{local theory} from \cite{Topaz:2015}, \cite{Topaz:2013a}, along with the compatibility of $\sigma$ with $\abf$, to show that $\sigma$ is compatible with certain special \emph{one-dimensional geometric subgroups} (see \S\ref{section: summary} for this terminology).
The majority of the work is then devoted to showing that $\sigma$ is compatible with \emph{all} such one-dimensional geometric subgroups, and these include the \emph{rational subgroups} considered in \cite{Topaz:2014}.
One then concludes, along similar lines to the mod-$\ell$ \emph{global theory} from loc. cit., that $\sigma$ arises from some automorphism of $K$, while some additional arguments show that this automorphism fixes $K_0$ pointwise.
In other words, $\sigma$ arises from an element of $\Gal_{k_0} = \Gal(K|K_0)$.

\subsection{A Guide Through the Paper}
\label{subsection: guide}

This paper contains a total of 11 sections, including \S\ref{section: intro} which is the introduction, and \S\ref{section: summary} which introduces some notation, and includes a summary of the proof of the main theorems.

Sections \ref{section: milnor}, \ref{section: cohom} and \ref{section: localthy} contain mostly generalities, appropriately translated to our context.
More specifically, in \S\ref{section: milnor}, we recall some basic facts about the Mod-$\ell$ Milnor K-ring of fields.
In \S\ref{section: cohom}, we recall the cohomological framework which allows us to translate back and forth between mod-$\ell$ abelian-by-central Galois groups and mod-$\ell$ Milnor K-rings -- this can be seen as a group-theoretical formulation of the Merkurjev-Suslin Theorem \cite{Merkurjev:1982}.
Such cohomological results have seen a recent resurgence in \cite{Chebolu:2012}, \cite{Efrat:2011}, \cite{Efrat:2015}, \cite{Topaz:2015b}, especially in connection with the Merkurjev-Suslin Theorem \cite{Merkurjev:1982} and/or the Bloch-Kato conjecture, which is now a highly-celebrated theorem due to {\sc Voevodsky-Rost} et al. \cite{Voevodsky:2011}, \cite{Rost:1998}, \cite{Weibel:2009}.
Nevertheless, the Merkurjev-Suslin Theorem is sufficient for the considerations in \S\ref{section: cohom}, as we summarize the appropriate results for our context in Theorem \ref{theorem: cohom / kummer / galois-to-milnor}.

In \S\ref{section: localthy}, we recall the required results concerning \emph{the local theory} in mod-$\ell$ abelian-by-central birational anabelian geometry.
These results have been developed incrementally over the last several years by \cite{Bogomolov:2002}, \cite{Mahe:2004}, \cite{Pop:2010}, \cite{Pop:2010a}, \cite{Efrat:2012}, \cite{Topaz:2013a}, \cite{Topaz:2015}.
We summarize the applicable results for our context in Theorem \ref{theorem: localthy / qpd / main-qpd}.

The core of the paper begins in \S6, where we discuss the mod-$\ell$ Milnor K-theory of function fields.
The ideas in this section are similar to \cite{Topaz:2014}*{\S3}, although the results themselves refine loc.~cit. somewhat.
Perhaps the most important result in \S6 is Corollary \ref{corollary: milnorff / geometric / infinite-subsets} which shows how to reconstruct a \emph{geometric subgroup} given sufficiently many of its elements.

In \S\ref{section: lattice}, we summarize (see Theorem \ref{theorem: lattice / acl / fund-thm-field-acl}) the main results from {\sc Evans-Hrushovski} \cite{Evans:1991}, \cite{Evans:1995} and {\sc Gismatullin} \cite{Gismatullin:2008}, translated appropriately to the context of the present paper.
In \S7 we also prove Corollary \ref{corollary: lattice / galois-mod-ell / galois-lattice-iso}, which shows that the absolute Galois group $\Galk$ can be canonically identified with a Galois group of certain geometric lattices associated to $K|k$ and $K_0|k_0$; this corollary will be used in a fundamental way in the proof of Theorem \ref{maintheorem: intro / milnor-variant / milnor-main}.

In \S\ref{section: essram}, we introduce the so-called \emph{essential branch locus}, and the notion of an \emph{essentially unramified point}.
We use this concept of essential ramification in several technical results in coordination with the local theory, in order to ensure that certain divisorial valuations can be ``detected'' in the mod-$\ell$ setting.

In \S\ref{section: strgen}, we recall the notion of a general element, and introduce the notion of a \emph{strongly-general} element.
In this section we also recall the so-called \emph{Birational-Bertini} theorem for general elements.
We also prove a Birational-Bertini theorem for strongly-general elements, which uses the ``yoga'' of essential ramification in a fundamental way.

In \S\ref{section: mainproof}, we give the detailed proof of Theorem \ref{maintheorem: intro / milnor-variant / milnor-main}, and note that Theorem \ref{maintheorem: intro / galois-variant / galois-main} follows from this by applying Theorem \ref{theorem: cohom / kummer / galois-to-milnor} from \S\ref{section: cohom}.
Finally, in \S\ref{section: finalproof}, we conclude the proofs of Theorems \ref{maintheorem: intro / main-result / iom-main-opens} and \ref{maintheorem: intro / main-result / iom-main-connected} by using Theorem \ref{maintheorem: intro / galois-variant / galois-main}.

To summarize, the following diagram indicates the logical relationships between Theorems \ref{maintheorem: intro / main-result / iom-main-opens}, \ref{maintheorem: intro / main-result / iom-main-connected}, \ref{maintheorem: intro / galois-variant / galois-main} and \ref{maintheorem: intro / milnor-variant / milnor-main}: 
\[ 
  \xymatrix{
    \text{Theorem \ref{maintheorem: intro / milnor-variant / milnor-main}} \ar@{<=>}[d]_{\text{Theorem \ref{theorem: cohom / kummer / galois-to-milnor}}} & & {} & & {} \\
    \text{Theorem \ref{maintheorem: intro / galois-variant / galois-main}} \ar@{=>}[rr]^{\S\ref{subsection: finalproof / ion-main-opens}} & & \text{Theorem \ref{maintheorem: intro / main-result / iom-main-opens}} \ar@{=>}[rr]^{\S\ref{subsection: finalproof / ion-main-connected}} & & \text{Theorem \ref{maintheorem: intro / main-result / iom-main-connected}.}
  }
\]
\subsection*{Acknowledgments}

First and foremost, the author would like to thank Florian Pop and Thomas Scanlon for numerous technical discussions concerning the topics in this paper.
The author also thanks all who expressed interest in this paper, and especially Y. Hoshi, E. Hrushovski, M. Kim, J. Min\'a\v{c}, A. Obus, A. Silberstein and A. Tamagawa.
Finally, the author thanks the anonymous referee, whose thoughtful comments helped improve this paper in various ways.

\section{Notation and a Summary}
\label{section: summary}


Throughout the whole paper, we will work with a fixed prime $\ell$ and a fixed base field $k_0$, such that $k_0$ is infinite, perfect, and of characteristic $\neq \ell$.
We denote by $\Galk$ the absolute Galois group of $k_0$.
We will also fix a function field $K_0$ over $k_0$ which is {\bf regular}, which means that $K_0$ has a separating transcendence base and that $k_0$ is relatively algebraically closed in $K_0$.
We denote by $k := \bar k_0$ the algebraic closure of $k_0$, and we write 
\[ K = K_0 \otimes_{k_0} k = K_0 \cdot k \]
for the base-change of $K_0$ to $k$.

When we discuss other fields which are potentially unrelated to $K_0|k_0$ and/or $K|k$ and which might have characteristic $\ell$, we will use letters such as $F,L,M$, etc.
The perfect closure of a field $F$ will be denoted by $F^i$.
If $\Char F = p > 0$, then we will write $\Frob_F$ for the usual Frobenius map on $F$, and we note that $\Frob_{F^i}$ is an automorphism of $F^i$.
In order to keep the notation consistent, if $\Char F = 0$, then $\Frob_F$ is defined to be the identity on $F$.
Finally, the absolute Galois group of a field $F$ will be denoted by $\Gal_F$, and the maximal pro-$\ell$ Galois group of $F$ will be denoted by $\Gc_F$.

For a valuation $v$ of $F$, we will use the following standard notation associated with $v$.
We denote the valuation ring of $v$ by $\Oc_v$ and the maximal ideal of $\Oc_v$ is denoted by $\mf_v$.
We will also write $\U_v := \Oc_v^\times$ for the $v$-units and $\U_v^1 := (1+\mf_v)$ for the principal $v$-units.
Finally, we write $vF$ for the value group of $v$ and $Fv$ for the residue field of $v$.
The residue map $\Oc_v \twoheadrightarrow Fv$ will usually be denoted by $t \mapsto \bar t$.
If $L$ is a subfield of $F$, we will abuse the notation and write $vL$ resp. $Lv$ for the value group resp. residue field of the restriction $v|_L$ of $v$ to $L$.

The majority of this paper deals with $\k_*(K)$, the mod-$\ell$ Milnor K-ring of $K$, the definition of which is recalled in \S\ref{section: milnor}.
We note now that for a field $F$, one has $\k_1(F) = F^\times/\ell$, and that the canonical projection $F^\times \twoheadrightarrow F^\times/\ell = \k_1(F)$ is denoted by $x \mapsto \{x\}_F$.

We now introduce some important notation which will be used consistently throughout the whole paper.
In particular, we introduce the notion of a \emph{geometric subgroup} of $\k_1(K)$, which is the primary object we study in this paper.
For a subset $S$ of $K$, we may consider $k(S)$, the subextension of $K|k$ generated by $S$, and we write:
\begin{enumerate}
  \item $\Kbb(S) := \overline{k(S)} \cap K$ for the relative algebraic-closure of $k(S)$ in $K$.
  \item $\Kfrak(S) := \{\Kbb(S)^\times\}_K$ for the image of $\Kbb(S)^\times$ in $\k_1(K)$.
\end{enumerate}
A subgroup $A$ of $\k_1(K)$ is called a {\bf geometric subgroup of $\k_1(K)$} provided that there exists some subset $S$ of $K$ such that $\Kfrak(S) = A$.
Since $\Kbb(S)$ is relatively algebraically closed in $K$, we note that the canonical map $\k_1(\Kbb(S)) \rightarrow \k_1(K)$ is injective, and its image is $\Kfrak(S)$.
In particular, the map $\k_1(\Kbb(S)) \rightarrow \k_1(K)$ induces a canonical isomorphism $\k_1(\Kbb(S)) \cong \Kfrak(S)$.

We will also frequently work with valuations of $K$ via the images of their (principal) units in $\k_1(K)$, and so we must introduce some more important notation here.
For a valuation $v$ of $K$, we will consistently write:
\begin{enumerate}
  \item $\Ufrak_v := \{\U_v\}_K$ for the image of the $v$-units in $\k_1(K)$.
  \item $\Ufrak_v^1 := \{\U_v^1\}_K$ for the image of the principal $v$-units in $\k_1(K)$.
\end{enumerate}
Note in particular that one has $\Ufrak_v^1 \subset \Ufrak_v$, and, since $vK$ is torsion-free, the quotient $\Ufrak_v/\Ufrak_v^1$ is canonically isomorphic to $\k_1(Kv)$.

We will frequently consider affine and projective spaces over $k$, which are given by a certain set of algebraically independent parameters.
More precisely, let $\tbf = (t_1,\ldots,t_r)$ be a collection of $r$ algebraically independent indeterminants over $k$.
In this case, we write 
\[ \Abb^r_{t_1,\ldots,t_r} = \Abb^r_\tbf := \Spec k[t_1,\ldots,t_r] \]
for affine $r$-space with parameters $t_1,\ldots,t_r$.
Similarly, we write 
\[ \Pbb^r_{t_1,\ldots,t_r} = \Pbb^r_\tbf := \Proj k[T_0,\ldots,T_r], \ \ T_i/T_0 = t_i \]
for the associated projective $r$-space which contains $\Abb^r_\tbf$ as a standard open subset.
In other words, one has a canonical open embedding $\Abb^r_\tbf \hookrightarrow \Pbb^r_\tbf$, and the function field of $\Abb^r_\tbf$ and/or $\Pbb^r_\tbf$ can be canonically identified with the rational function field $k(\tbf)$ generated by $\tbf$.

We will also identify the closed points of $\Abb^r_\tbf$ resp. $\Pbb^r_\tbf$ with the set $\Abb^r_\tbf(k) = k^r$ resp. $\Pbb^r_\tbf(k) = (k^{r+1} \smallsetminus \{0\})/k^\times$ of $k$-rational points.
We will use affine coordinates $(a_1,\ldots,a_r)$ to denote elements of $\Abb^r_\tbf(k) = k^r$, and we will use homogeneous coordinates $(a_0:\cdots:a_r)$ to denote elements of $\Pbb^r_\tbf(k) = (k^{r+1} \smallsetminus \{0\})/k^\times$.
In particular, we identify $\Abb^r_\tbf(k) = k^r$ with the elements of the form $(1:a_1:\cdots:a_r)$ in $\Pbb^r_\tbf(k) = (k^{r+1} \smallsetminus \{0\})/k^\times$.

Since we will consider various representations of $\Galk$, in order to simplify the notation, we will denote all such representations by $\rho_{k_0}$ if no confusion is possible.
This convention holds in particular for the representations $\rho^c_{k_0,*}$ ($* = X,\Vc,K_0$) and $\rho^{\rm M}_{k_0,K_0}$, which were defined in \S\ref{section: intro}.
We will also use this implicit terminology when describing the compatibility of certain morphisms with $\rho_{k_0}$.
To be precise, if $\bullet,\circ$ are two objects endowed with two $\Galk$-representations $\rho_{k_0} : \Galk \rightarrow \Aut(\bullet)$ resp. $\rho_{k_0} : \Galk \rightarrow \Aut(\circ)$, we say that a morphism $f : \Aut(\bullet) \rightarrow \Aut(\circ)$ of automorphism groups is {\bf compatible with $\rho_{k_0}$} provided that the following diagram commutes:
\[ 
\xymatrix{
  \Galk \ar[r]^{\rho_{k_0}} \ar[dr]_{\rho_{k_0}} & \Aut(\bullet) \ar[d]^f \\
  {} & \Aut(\circ). 
}
\]
Furthermore, we will say that an element $\sigma \in \Aut(\bullet)$ {\bf arises from $\Galk$} if $\sigma$ is in the image of $\rho_{k_0}$.
If we wish to make precise the element $\tau \in \Galk$ which maps to this $\sigma$, we will say that $\sigma$ is {\bf defined by $\tau$}.

Some of the proofs in this paper are fairly technical, although the overall idea is quite natural, and can be briefly described using the following three key steps:
\begin{enumerate}
  \item First, reduce all the main theorems to Theorem \ref{maintheorem: intro / milnor-variant / milnor-main}.
  \item Second, prove that any element of $\UAutm_\abf(\k_1(K))$ induces an automorphism of a certain lattice $\Gbb^*(K|k)$ which is of geometric origin. This step is the most difficult and takes up the majority of the paper.
  \item Finally, we use an analogue of the \emph{Fundamental Theorem of Projective Geometry} for this lattice $\Gbb^*(K|k)$, to deduce that any element of $\UAutm_\abf(\k_1(K))$ arises in a unique way from $\Galk$.
  This analogue of the fundamental theorem of projective geometry comes from the work of {\sc Evans-Hrushovski} \cite{Evans:1991}, \cite{Evans:1995} and {\sc Gismatullin} \cite{Gismatullin:2008}, and it relies on the so-called \emph{group-configuration theorem} from geometric stability theory.
\end{enumerate}
For the sake of the reader, we now provide a fairly detailed summary of the proofs of the main theorems, to act as a guide for reading the details which are found in the body of the paper.

\subsection{Reduction to Theorem \ref{maintheorem: intro / main-result / iom-main-opens}}

In the terminology introduced above, suppose that $\Ucal_X$ and $\Ucal_Y$ are birational systems for $X$ resp. $Y$.
Furthermore, suppose that $\Ucal_X$ dominates $\Ucal_Y$ in $\Vc$.
Note that any element $\phi$ of $\Autc(\pi^a|_\Vc)$ defines an element of $\phi|_{\Ucal_X} \in \Autc(\pi^a|_{\Ucal_X})$ and an element of $\phi|_{\Ucal_Y} \in \Autc(\pi^a|_{\Ucal_Y})$ by restriction.
The condition that $\Ucal_X$ dominates $\Ucal_Y$ implies the following property: If $\phi|_{\Ucal_X}$ is defined by $\tau \in \Galk$, then $\phi|_{\Ucal_Y}$ is defined by $\tau$ as well.
The ``$5$-connectedness'' assumption on $\Vc$ is then used to reduce Theorem \ref{maintheorem: intro / main-result / iom-main-connected} to Theorem \ref{maintheorem: intro / main-result / iom-main-opens}.

\subsection{Reduction to Theorem \ref{maintheorem: intro / galois-variant / galois-main}}

Let $X$ be a normal $k_0$-variety with $K = k(X)$, and let $\Ucal$ be a birational system of $X$.
By passing to the projective limit over $\Ucal$, one obtains a canonical \emph{injective} map
\[ \UAutc(\pi^a|_\Ucal) \rightarrow \UAutc(\Gc_K^a). \]
Given a finite tuple $\abf$ of elements of $k_0^\times$, this injective map above induces a map
\[ \UAutc(\pi^a|_{\Ucal_\abf}) \rightarrow \UAutc_\abf(\Gc_K^a), \]
by first restricting to $\Ucal$, then taking projective limits over $\Ucal$.
This induced map turns out to be injective as well, as long as $\dim \Ucal \geq 2$.
Since this injection is compatible with $\rho_{k_0}$, we see that Theorem \ref{maintheorem: intro / main-result / iom-main-opens} follows from Theorem \ref{maintheorem: intro / galois-variant / galois-main}.

\subsection{Reduction to Theorem \ref{maintheorem: intro / milnor-variant / milnor-main}}

Kummer theory yields a canonical perfect pairing
\[ \Gc_K^a \times \k_1(K) \rightarrow \mu_\ell. \]
Moreover, using the well-known duality between $\H^2(\Gc_K^c,\Z/\ell)$ and the relations in a minimal free presentation of $\Gc_K^c$ (in the category of pro-$\ell$ groups), along with the fact that cup-products correspond to commutators in this duality, it is then a consequence of the Merkurjev-Suslin Theorem \cite{Merkurjev:1982} that one has a canonical isomorphism
\[ \Autc(\Gc_K^a) \cong \Autm(\k_1(K)). \]
This isomorphism is obtained by dualizing an automorphism of $\Gc_K^a$, via the Kummer pairing above, to obtain an automorphism of $\k_1(K)$.
Although the isomorphism above is not compatible with $\rho_{k_0}$ exactly (since we didn't introduce the appropriate cyclotomic twist), the induced isomorphism 
\[ \UAutc(\Gc_K^a) \cong \UAutm(\k_1(K)) \] 
is actually compatible with $\rho_{k_0}$.
See Theorem \ref{theorem: cohom / kummer / galois-to-milnor}.
By Kummer theory, it follows that the isomorphism above restricts to an isomorphism
\[ \UAutc_\abf(\Gc_K^a) \cong \UAutm_\abf(\k_1(K)) \]
which is again compatible with $\rho_{k_0}$.
Thus, Theorem \ref{maintheorem: intro / galois-variant / galois-main} is equivalent to Theorem \ref{maintheorem: intro / milnor-variant / milnor-main}.

\subsection{The Mod-$\ell$ Geometric Lattice}

The primary focus of the proof of Theorem \ref{maintheorem: intro / milnor-variant / milnor-main} is to show that an element $\phi \in \UAutm_\abf(\k_1(K))$ induces an automorphism of a certain \emph{graded lattice}
\[ \Gfrak^*(K|k) := \coprod_{r \geq 0} \Gfrak^r(K|k) \]
which is contained in the lattice of subgroups of $\k_1(K)$.
The elements of $\Gfrak^*(K|k)$ are the {\bf geometric subgroups} of $\k_1(K)$, as introduced above, and the grading is induced by the so-called {\bf Milnor-dimension} of subsets of $\k_1(K)$. 
Moreover, as a consequence of the construction, it will also follow that such an induced automorphism of $\Gfrak^*(K|k)$ fixes all geometric subgroups which arise from $K_0|k_0$.
We denote the collection of all such automorphisms of $\Gfrak^*(K|k)$ by $\Aut^*(\Gfrak^*(K|k)|K_0)$.

We show in Proposition \ref{proposition: lattice / mod-ell / mod-ell-lattice} that the map $\Kbb(S) \mapsto \Kfrak(S)$ (see the notation introduced above) induces an isomorphism of graded lattices $\Gbb^*(K|k) \cong \Gfrak^*(K|k)$, where $\Gbb^*(K|k)$ is the lattice of relatively-algebraically closed subextensions of $K|k$ graded by transcendence degree.
Thus, any element of $\UAutm_\abf(\k_1(K))$ will define an automorphism of $\Gbb^*(K|k)$ which fixes subextensions arising from $K_0|k_0$.
We then use the results of {\sc Evans-Hrushovski} \cite{Evans:1991}, \cite{Evans:1995} and {\sc Gismatullin} \cite{Gismatullin:2008} to show that any such automorphism of $\Gbb^*(K|k)$ arises from some element of $\Galk$.
See Theorem \ref{theorem: lattice / acl / fund-thm-field-acl}, Proposition \ref{proposition: lattice / galois-acl / galois-acl} and Corollary \ref{corollary: lattice / galois-mod-ell / galois-lattice-iso} for more details.

\subsection{Generic Generators of $\Gfrak^*(K|k)$}

The idea of the proof is to ``produce'' elements of $\Gfrak^*(K|k)$, i.e. geometric subgroups of $\k_1(K)$, using the ``given'' data of the mod-$\ell$ Milnor K-ring $\k_*(K)$ endowed with some extra structure which is compatible with all automorphisms in $\UAutm_\abf(\k_1(K))$, and also to ensure that this process is compatible with such automorphisms.
The reconstruction process of $\Gfrak^*(K|k)$ relies on a certain ``closure operation'' called the {\bf Milnor Supremum}, which takes place entirely in the ring $\k_*(K)$, and which takes in a set of geometric subgroups as an input and returns a geometric subgroup as an output.

The fact that this closure operation produces geometric subgroups follows from some explicit vanishing and non-vanishing results in $\k_*(K)$.
The vanishing results say that $\k_*(K) = 0$ for $* > \trdeg(K|k)$, and this follows from well-known cohomological dimension calculations of $K$ and the Bloch-Kato conjecture, which is now a theorem of {\sc Voevodsky-Rost} et al. \cite{Voevodsky:2011}, \cite{Rost:1998}, \cite{Weibel:2009}.
The non-vanishing results say that there are ``many'' elements of $\k_1(K)$ which have non-trivial products.
The ``many'' above refers to the fact that these non-vanishing results all involve some open condition on some model of $K|k$ (or some subextension of $K|k$) which is usually the complement of some branch locus.

\subsection{Fixing Elements of $\Gfrak^*(K|k)$ Which Arise from $K_0$}

The fact that an automorphism $\sigma\in\UAutm_\abf(\k_1(K))$ is compatible with the \emph{tuple} $\abf$ implies that $\sigma$ fixes the elements of $\Gfrak^1(K|k)$ which come from $K_0$.
Applying the ``closure operation'' described above shows that $\sigma$ fixes all of the elements of $\Gfrak^*(K|k)$ which come from $K_0$.
However, this is still very far from what we need, because at this point there is absolutely nothing we can ``construct/produce'' which is moved around by the action of $\Galk$.

\subsection{Fixing Elements of $K_0^\times$}

A key step in the proof is to show that any element of $\UAutm_\abf(\k_1(K))$ has a representative $\sigma \in \Autm(\k_1(K))$ such that $\sigma \{x\}_K = \{x\}_K$ for all $x \in K_0^\times$.
To show this, we introduce the concept of a \emph{strongly-general element} of $K|k$, which is related to the concept of a general element from \cite{Pop:2007}, \cite{Pop:2012} but has further assumptions.
Another key input comes from the \emph{local theory} in abelian-by-central birational anabelian geometry for function fields over algebraically closed fields.
In this context, the local theory says that $\sigma$ is compatible with \emph{quasi-divisorial valuations}.
But using the previous step, one can show that $\sigma$ is actually compatible with \emph{divisorial valuations}.
The literature concerning the local theory in abelian-by-central birational anabelian geometry is quite rich, and it includes the following works among others \cite{Bogomolov:2002}, \cite{Mahe:2004}, \cite{Pop:2010}, \cite{Pop:2010a}, \cite{Topaz:2013a}, \cite{Topaz:2015}.
See the introduction of \cite{Topaz:2013a} for a detailed overview of the history of the local theory.

To show that $\sigma \in \UAutm_\abf(\k_1(K))$ fixes elements from $K_0^\times$, we first show this for elements of $K_0^\times$ which are strongly-general in $K|k$, and this uses the local theory in an essential way.
To deduce that $\sigma$ fixes all elements arising from $K_0^\times$, we prove a \emph{Birational-Bertini} type result for strongly-general elements, which shows that there are ``sufficiently many'' strongly-general elements in higher-dimensional function fields.

\subsection{The Base Case}

By using our ``closure operation'' described above, in order to show that $\sigma \in \UAutm_\abf(\k_1(K))$ induces an automorphism of the lattice $\Gfrak^*(K|k)$, it suffices to show that $\sigma$ induces a permutation of $\Gfrak^1(K|k)$, the set of \emph{1-dimensional geometric subgroups}.
Using the notation introduced above, a one-dimensional geometric subgroup is a subgroup of $\k_1(K)$ which is of the form $\Kfrak(t)$ for some $t \in K^\times \smallsetminus k^\times$.
Note that every element $t \in K = K_0 \otimes_{k_0} k$ can be written as a sum $a_0 x_0 + \cdots + a_r x_r$ for some $x_i \in K_0$ and $a_i \in k$.
The proof now proceeds by induction on the length $r$ of such an expression.
The case $r = 0$ was discussed above, and so the base case for the induction is $r = 1$.

The base case works as follows.
Using the concept of \emph{essential ramification}, we show that there are ``many'' elements of the form $x_0 + a_1 x_1$ with $x_i \in K_0$ and $a_i \in k$ which are ``acceptable'' with respect to $\sigma$.
The term ``acceptable'' means that there exists some $t \in K$ such that $\sigma$ sends the geometric subgroup associated to $x_0 + a_1 x_1$ to the geometric subgroup associated to $t$.
As before, the term ``many'' is related to a precise open condition on a certain model of a subextension of $K|k$, with the condition being related to the \emph{essential branch locus}.

Once we have ``many'' acceptable elements of the form $x_0 + a_1 x_1$, we use the ``yoga'' of the ``generic generators'' mentioned above to show that \emph{all pairs} $(t_0,t_1)$, with $t_0 \in K_0$ and $t_1 = a_0 x_0 + a_1 x_1$, $x_i \in K_0$, $a_i \in k$, are acceptable with respect to $\sigma$ (acceptability is defined similarly for pairs as it was for elements of $K$).
We then take appropriate intersections of certain \emph{two-dimensional} geometric subgroups to deduce that \emph{every element} of the form $a_0 x_0 + a_1 x_1$ is acceptable.

\subsection{Inductive Case}

To conclude the proof, one proceeds by induction on $r$ as above, with the inductive hypothesis being that every element of $K$ of the form $a_0 x_0 + \cdots + a_s x_s$ with $s < r$, $x_i \in K_0$, $a_i \in k$, is acceptable with respect to $\sigma$.
The proof is now similar in nature to the proof of the base case.
Indeed, first we show that \emph{pairs} of certain elements are acceptable, then take intersections of certain two-dimensional geometric subgroups to deduce that all elements of the form above are acceptable.

\subsection{Concluding the Proof}

The argument outlined above shows that $\sigma \in \UAutm_\abf(\k_1(K))$ induces a permutation of $\Gfrak^1(K|k)$.
Since $\sigma$ is compatible with the ``closure operation'' described above, it follows that $\sigma$ induces an automorphism of the lattice $\Gfrak^*(K|k)$.
Namely, one obtains a canonical homomorphism $\UAutm_\abf(\k_1(K)) \rightarrow \Aut^*(\Gfrak^*(K|k)|K_0) \subset \Aut^*(\Gfrak^*(K|k))$.
On the other hand, we prove that the map $\Galk \rightarrow \Aut^*(\Gfrak^*(K|k)|K_0)$ is an isomorphism by using the results of {\sc Evans-Hrushovski} \cite{Evans:1991}, \cite{Evans:1995} and {\sc Gismatullin} \cite{Gismatullin:2008}, as noted above.
Thus, to conclude the proof of Theorem \ref{maintheorem: intro / milnor-variant / milnor-main}, it remains to show that the map
\[ \UAutm_\abf(\k_1(K)) \rightarrow \Aut^*(\Gfrak^*(K|k)) \]
is \emph{injective}.
The argument here again uses the theory of \emph{strongly-general elements}.

Indeed, any element $\sigma$ in the kernel of the map above must fix \emph{all geometric subgroups} of $\k_1(K)$.
First, we show this implies that the restriction of $\sigma$ to any \emph{strongly-general} geometric subgroup looks like some element of $(\Z/\ell)^\times \cdot \one$.
Finally, one uses a \emph{Birational-Bertini} type argument again to deduce that $\sigma$ is indeed an element of $(\Z/\ell)^\times \cdot \one_{\k_1(K)}$.
This thereby proves the injectivity of the map above, hence concluding the proof of Theorem \ref{maintheorem: intro / milnor-variant / milnor-main}.

\section{Milnor K-Theory}
\label{section: milnor}


Let $F$ be a field.
We recall that the {\bf $r$-th Milnor K-group of $F$} is defined as follows:
\[ \K_r(F) = \frac{(F^\times)^{\otimes r}}{\langle x_1 \otimes \cdots \otimes x_r \ : \ \exists i < j, \ x_i + x_j = 1 \rangle}. \]
The tensor product makes $\K_*(F) := \bigoplus_{r \geq 0} \K_r(F)$ into a graded-commutative algebra over the ring $\Z = \K_0(F)$, and we call $\K_*(F)$ the {\bf Milnor K-ring of $F$}.
It is customary to denote the product of $a_1,\ldots,a_r \in \K_1(F) = F^\times$ in this ring by $\{a_1,\ldots,a_r\}$.

We will use the standard notation $\k_r(F) := \K_r(F)/\ell$ and call $\k_r(F)$ the {\bf $r$-th mod-$\ell$ Milnor K-group} of $F$.
As with $\K_*(F)$, the tensor product makes $\k_*(F) := \bigoplus_{r \geq 0} \k_r(F)$ into a graded commutative algebra over $\k_0(F) =\Z/\ell$, and we call $\k_*(F)$ the {\bf mod-$\ell$ Milnor K-ring of $F$}.
Given $r$ elements $a_1,\ldots,a_r$ of $\k_1(F) = F^\times/\ell$, we will denote their product in $\k_*(F)$ by 
\[ \{a_1,\ldots,a_r\}_F \in \k_r(F).\]
For $b_1,\ldots,b_r \in \K_1(F) = F^\times$, we will abuse the notation and write
\[ \{b_1,\ldots,b_r\}_F := \{b_1 \cdot F^{\times \ell},\ldots,b_r \cdot F^{\times\ell}\}_F. \] 
In particular, $\{\bullet\}_F$ denotes the canonical projection $F^\times \twoheadrightarrow F^\times/\ell = \k_1(F)$.

Note that $\k_*(F)$ is functorial in $F$.
The notation $\{\bullet,\ldots,\bullet\}_F$ will also be used to indicate this functoriality.
Namely, if $b_1,\ldots,b_r \in F^\times$ are given, and $F \hookrightarrow L$ is a field extension, then $\{b_1,\ldots,b_r\}_L$ denotes the image of $\{b_1,\ldots,b_r\}_F$ under the canonical map $\k_r(F) \rightarrow \k_r(L)$.

\subsection{Purely Inseparable Extensions}
\label{subsection: milnor / insep}

We will frequently reduce some arguments concerning finite field extensions to the case where the extension is separable.
This will be possible because a purely-inseparable extension of fields of characteristic $\neq \ell$ induces an isomorphism on the mod-$\ell$ Milnor K-ring which is also compatible with valuations, as the following two lemmas show.

\begin{lemma}
\label{lemma: milnor / insep / insep-iso}
  Let $L|F$ be a finite and purely inseparable extension of fields, such that $\Char F \neq \ell$.
  Then the canonical map $\k_*(F) \rightarrow \k_*(L)$ is an isomorphism for all $* \geq 0$.
\end{lemma}
\begin{proof}
  Put $p = \Char F$ and assume that $p > 0$.
  Since $L$ is finite and purely inseparable over $F$, one has $L \subset F^{1/p^n}$ for sufficiently large $n$.
  Since $p$ is invertible in $\Z/\ell$, the canonical map $\k_*(F) \rightarrow \k_*(F^{1/p^n})$ is an isomorphism.
  As this map factors through $\k_*(L)$, we deduce that the map $\k_*(F) \rightarrow \k_*(L)$ is injective.
  To deduce that $\k_*(F) \rightarrow \k_*(L)$ is also surjective, it suffices to prove that $\k_*(L) \rightarrow \k_*(F^{1/p^n})$ is injective.
  
  For any $\eta$ in the kernel of $\k_*(L) \rightarrow \k_*(F^{1/p^n})$, there exists some intermediate extension $M$ of $F^{1/p^n}|L$ such that $M|L$ is finite and such that $\eta$ is in the kernel of $\k_*(L) \rightarrow \k_*(M)$.
  But such a subextension $M|L$ is purely inseparable, so the argument above shows that $\eta = 0$.
  Thus $\k_*(L) \rightarrow \k_*(F^{1/p^n})$ is injective, as required.
\end{proof}

\begin{lemma}
\label{lemma: milnor / insep / insep-val}
  Let $(L,w)|(F,v)$ be a finite and purely inseparable extension of valued fields, such that $\Char F \neq \ell$.
  Then the canonical map $\k_1(F) \rightarrow \k_1(L)$ restricts to an isomorphism $\{\U_v\}_F \xrightarrow{\cong} \{\U_w\}_L$.
\end{lemma}
\begin{proof}
  Put $p = \Char F$ and assume that $p > 0$.
  Since $L|F$ is purely inseparable, the index $[wL:vF]$ is a power of $p$.
  Therefore, the canonical map $vF/\ell \rightarrow wL/\ell$ is an isomorphism.
  Next, note that one has a commutative diagram with exact rows
  \[
  \xymatrix{
  1 \ar[r] & \{\U_v\}_F \ar[d]\ar[r] & \k_1(F) \ar[d]^{\cong}\ar[r]^v & vF/\ell \ar[d]^{\cong}\ar[r] & 1 \\
  1 \ar[r] & \{\U_w\}_L \ar[r] & \k_1(L) \ar[r]_w & wL/\ell \ar[r] & 1.
  }
  \]
  As indicated on the diagram, the middle vertical arrow is an isomorphism by Lemma \ref{lemma: milnor / insep / insep-iso}, and the right vertical map is an isomorphism as noted above.
  It follows that the left vertical map is also an isomorphism, as required.
\end{proof}

\subsection{Tame Symbols}
\label{subsection: milnor / tame}

Suppose that $(F,v)$ is a discretely valued field of rank $1$, so that one has $vF \cong \Z$.
We recall that the ($r$-th) {\bf tame symbol} associated to $v$ is a morphism $\{\bullet\}^v_F : \k_{r+1}(F) \rightarrow \k_r(Fv)$ which is uniquely defined by the condition
\[ \{\pi,u_1,\ldots,u_r\}_F^v = \{\bar u_1,\ldots,\bar u_r\}_{Fv}, \]
where $\pi$ is any uniformizer of $v$ (i.e. $v(\pi) = 1$), the elements $u_1,\ldots,u_r \in \U_v$ are $v$-units, and $\bar u_i$ denotes the image of $u_i$ in $(Fv)^\times$.

We will primarily use tame symbols to prove that the mod-$\ell$ Milnor K-ring of a function field contains many non-trivial elements.
Most such ``non-vanishing'' results will essentially follow from the following fact concerning the field of Laurent series.

\begin{fact}
\label{fact: milnor / tame / laurent-coords}
  Let $F$ be any field, and let $t_1,\ldots,t_d$ be algebraically independent indeterminants over $F$.
  Consider the field of Laurent series $L = F((t_1,\ldots,t_d))$ over $F$.
  Let $f_0$ be a non-trivial element of $\k_r(F)$ for some $r \geq 0$, and let $f := \{f_0\}_L$ denote the image of $f_0$ in $\k_r(L)$.
  Then the product
  \[ \{t_1,\ldots,t_d,f\}_L \]
  is non-trivial in $\k_{r+d}(L)$.
\end{fact}
\begin{proof}
  Note that we have an $F$-embedding of $L$ into $L' := F((t_d))((t_{d-1}))\cdots((t_1))$, so it suffices to prove that $\{t_1,\ldots,t_d,f\}_{L'}$ is non-trivial in $\k_{r+d}(L')$.
  We proceed to prove this by induction on $d$, with the base case $d = 0$ being trivial.
  For the inductive case, let $v$ be the $t_1$-adic valuation on $L'$.
  Then the residue field of $v$ can be canonically identified with $F((t_d))\cdots((t_2)) =: M$ via the obvious map sending $t_i \in \U_v$ to $t_i \in M$ for $i \geq 2$.
  Finally, applying the tame symbol associated to $v$, we obtain
  \[ \{t_1,\ldots,t_d,f\}_{L'}^v = \{t_2,\ldots,t_d,f\}_M \neq 0 \]
  by the inductive hypothesis.
  Hence $\{t_1,\ldots,t_d,f\}_{L'} \neq 0$, as required.
\end{proof}

\section{Galois Cohomology}
\label{section: cohom}


In this section, we recall the basic framework which allows us to translate back and forth between the ``Galois'' context (i.e. mod-$\ell$ abelian-by-central Galois groups) and the ``Milnor'' context (i.e. mod-$\ell$ Milnor K-rings) by using Kummer theory.
This theory is more-or-less well known, as it follows from the fact that $\H^2$ of a pro-$\ell$ group is ``dual'' to the relations in a minimal free presentation of the group, while the cup product in $\H^*$ is ``dual'' to the commutator $[\bullet,\bullet]$ as defined in \S\ref{subsection: intro / mod-ell-category}.

The essential calculations concerning cup products and commutators were first carried out by {\sc Labute} \cite{Labute:1967} (see also \cite{Neukirch:2013}*{\S3.9}).
These calculations have seen a recent resurgence of interest in \cite{Chebolu:2012}, \cite{Efrat:2015}, \cite{Efrat:2011}, \cite{Topaz:2015b}, especially in connection with the Merkurjev-Suslin Theorem \cite{Merkurjev:1982} and the Bloch-Kato conjecture, which is now a theorem of {\sc Voevodsky-Rost} et al. \cite{Voevodsky:2011}, \cite{Rost:1998}, \cite{Weibel:2009}.
The Merkurjev-Suslin Theorem is sufficient for our considerations here.
In fact, the discussion in this section can be seen as a summary of \cite{Topaz:2014}*{\S8}, appropriately translated to the context of this paper.

First we introduce a bit of notation.
We let $\H^* = \H^*_{\rm cont}$ denote continuous-cochain group-cohomology.
For a pro-$\ell$ group $\Gc$, we write $\H^*(\Gc) := \H^*(\Gc,\Z/\ell)$, but we will specify the coefficient module if it is different from $\Z/\ell$.
We will denote the {\bf decomposable part} of $\H^*(\Gc)$ by $\H^*(\Gc)_{\rm dec}$.
We will only need to consider $\H^2(\Gc)_{\rm dec}$, which is defined as the image of the cup-product map
\[ \cup : \H^1(\Gc)^{\otimes 2} \rightarrow \H^2(\Gc). \]
Also, recall that the {\bf Bockstein morphism}, denoted $\beta : \H^1(\Gc) \rightarrow \H^2(\Gc)$, is the connecting morphism in the cohomological long exact sequence associated to the short exact sequence of coefficient modules $0 \rightarrow \Z/\ell \rightarrow \Z/\ell^2 \rightarrow \Z/\ell \rightarrow 0$.

Finally, recall that for a field $F$, we let $\Gc_F$ denote the maximal pro-$\ell$ quotient of $\Gal_F$.
Most of the general results we say in this section hold true for general fields $F$ such that $\Char F \neq \ell$ and $\mu_\ell \subset F$.
However, in order to simplify the discussion, we will restrict our attention to the function field $K|k$ which is the main focus of the paper.

\subsection{Generalities on Cohomology of Pro-$\ell$ Groups}
\label{subsection: cohom / generalities}

In this subsection we summarize the basic connection between the group-theoretical structure of a pro-$\ell$ group $\Gc$ and the mod-$\ell$ cohomology of its various quotients.
We will use the notation introduced in \S\ref{subsection: intro / mod-ell-category}.
First of all, recall that the inflation map induces canonical isomorphisms
\[ \H^1(\Gc^a) \xrightarrow{\cong} \H^1(\Gc^c) \xrightarrow{\cong} \H^1(\Gc). \]
We will tacitly identify these three cohomology groups.
Since $\H^1(\Gc^a) = \Hom(\Gc^a,\Z/\ell)$, the definition of $\Gc^{(2)}$ implies that $\Gc^a$ and $\H^1(\Gc^a)$ are in perfect $\Z/\ell$-duality.
This perfect duality thereby induces a perfect pairing
\[ \widehat\wedge^2(\Gc^a) \times \wedge^2(\H^1(\Gc^a)) \rightarrow \Z/\ell \]
in the usual way.

Next, recall that the Lyndon-Hochschild-Serre (LHS) spectral sequence
\[ {\rm E}_2^{i,j} = \H^i(\Gc^a,\H^j(\Gc^{(2)})) \Rightarrow \H^{i+j}(\Gc) \]
induces a canonical exact sequence:
\[ 0 \rightarrow \H^1(\Gc^{(2)})^\Gc \xrightarrow{d_2} \H^2(\Gc^a) \rightarrow \H^2(\Gc) \]
where $d_2$ is the differential on the ${\rm E}_2$-page of the spectral sequence and $\H^2(\Gc^a) \rightarrow \H^2(\Gc)$ is the canonical inflation map.
Moreover, since $\Gc^c = \Gc/\Gc^{(3)}$ is a central extension of $\Gc^a$ by $\Gc^{(2)}/\Gc^{(3)}$, we see that the image of the canonical \emph{injective} inflation map $\H^1(\Gc^{(2)}/\Gc^{(3)}) \rightarrow \H^1(\Gc^{(2)})$ is contained in the $\Gc$-invariants $\H^1(\Gc^{(2)})^\Gc$.
We will henceforth identify $\H^1(\Gc^{(2)}/\Gc^{(3)})$ with its isomorphic image in $\H^1(\Gc^{(2)})^\Gc$.
Using the discussion and notation introduced above, the following fact summarizes the general group-theoretical properties which we will need.

\begin{fact}
\label{fact: cohom / generalities / cup-vs-commutator}
  The following hold:
  \begin{enumerate}
    \item The preimage of $\H^2(\Gc^a)_{\rm dec}$ under the map $d_2 : \H^1(\Gc^{(2)})^\Gc \rightarrow \H^2(\Gc^a)$ is precisely $\H^1(\Gc^{(2)}/\Gc^{(3)})$.
    \item The cup product induces a canonical isomorphism $\wedge^2(\H^1(\Gc^a)) \rightarrow \H^2(\Gc^a)/\beta\H^1(\Gc^a)$.
    \item Identify $\H^2(\Gc^a)/\beta\H^1(\Gc^a)$ with $\wedge^2(\H^1(\Gc^a))$ using the isomorphism from (2) above.
    Then the map 
    \[ [\bullet,\bullet] : \widehat\wedge^2(\Gc^a) \rightarrow \Gc^{(2)}/\Gc^{(3)} \]
    is $\Z/\ell$-dual (hence Pontryagin-dual) to the map 
    \[ \H^1(\Gc^{(2)}/\Gc^{(3)}) \xrightarrow{-d_2} \H^2(\Gc^a) \twoheadrightarrow \H^2(\Gc^a)/\beta\H^1(\Gc^a) = \wedge^2(\H^1(\Gc^a)). \]
  \end{enumerate}
\end{fact}
\begin{proof}
  See \cite{Topaz:2014}*{Lemma 8.2} for assertion (1).
  Assertion (2) follows from the K\"unneth formula along with the fact that $\Gc^a$ is isomorphic to a direct power of $\Z/\ell$; see \cite{Topaz:2014}*{Fact 8.1} and the surrounding discussion for more details.

  Assertion (3) is the standard ``duality'' between the commutator and the cup-product.
  This ``duality'' has been well-known for some time (see e.g. \cite{Neukirch:2013}*{Proposition 3.9.13}), but \cite{Topaz:2015b}*{Theorem 2} can also be used as a reference.
  See \cite{Topaz:2014}*{Fact 8.3} for more details.
\end{proof}

\subsection{Kummer Theory}
\label{subsection: cohom / kummer}

Recall that Kummer theory yields a canonical perfect pairing 
\[ \Gc_K^a \times \k_1(K) \rightarrow \mu_\ell \]
which is defined by $(\sigma,\{x\}_K) \mapsto \sigma(\sqrt[\ell]{x})/\sqrt[\ell]{x}$.
For an automorphism $\phi$ of $\k_1(K)$, we denote by $\phi^*$ the automorphism of $\Gc_K^a$ which is dual to $\phi$ via the pairing above.
This perfect duality yields a canonical isomorphism of automorphism groups
\[\Kummer : \Aut(\Gc_K^a) \xrightarrow{\cong} \Aut(\k_1(K)), \]
which is given by mapping $\phi \in \Aut(\Gc_K^a)$ to $(\phi^{-1})^*$.
This isomorphism further induces a canonical isomorphism on $(\Z/\ell)^\times$-classes of automorphism groups
\[ \Kummer : \UAut(\Gc_K^a) \xrightarrow{\cong} \UAut(\k_1(K)). \]

These isomorphisms are actually compatible with the Galois action, once we introduce the appropriate twist.
Hence the last isomorphism is \emph{exactly} compatible with $\rho_{k_0}$.
To make things precise, let $\chi_\ell : \Galk \rightarrow \Z_\ell^\times$ denote the $\ell$-adic cyclotomic character of the base field $k_0$.
We can then define the {\bf $i$-th cyclotomic twist} $\rho_{k_0}(i)$ of the representation 
\[ \rho_{k_0} : \Galk \rightarrow \Aut(\k_1(K)) \]
in the obvious way, as $\rho_{k_0}(i)(\tau) = \chi_\ell(\tau)^i \cdot \rho_{k_0}(\tau)$.
Next, recall that Kummer theory produces a canonical isomorphism 
\[ \k_1(K) = \H^1(\Gc_K^a,\mu_\ell) = \H^1(\Gc_K^a)\otimes \mu_\ell \]
which is compatible with the action of $\Galk$.
It therefore follows that the following diagram commutes
\[\xymatrix{
  \Galk \ar[r]^{\rho_{k_0}} \ar[dr]_{\rho_{k_0}(-1)} & \Aut(\Gc_K^a) \ar[d]^{\Kummer} \\
  {} & \Aut(\k_1(K))
}\]
because we introduced the appropriate cyclotomic twist.
But the definition of $\rho_{k_0}(*)$ ensures that the induced isomorphism 
\[ \Kummer : \UAut(\Gc_K^a) \xrightarrow{\cong} \UAut(\k_1(K)) \]
is compatible with $\rho_{k_0}$ since the cyclotomic twist becomes completely irrelevant after modding out by $(\Z/\ell)^\times$.

\subsection{Galois vs. Milnor}
\label{subsection: cohom / galois-to-milnor}
With this discussion, we may now present the main theorem of this section which allows us to pass back and forth between the Galois-setting and the Milnor-setting.
The following theorem is essentially a group-theoretical interpretation of the Merkurjev-Suslin Theorem \cite{Merkurjev:1982}.

\begin{theorem}
\label{theorem: cohom / kummer / galois-to-milnor}
  The following hold:
  \begin{enumerate}
    \item Let $\phi \in \Aut(\k_1(K))$ be given and consider $\phi^* \in \Aut(\Gc_K^a)$ the Kummer-dual of $\phi$.
    Then one has $\phi \in \Autm(\k_1(K))$ if and only if $\phi^* \in \Autc(\Gc_K^a)$.
    \item The canonical isomorphism $\Kummer : \Aut(\Gc_K^a) \xrightarrow{\cong} \Aut(\k_1(K))$ restricts to an isomorphism $\Kummer : \Autc(\Gc_K^a) \xrightarrow{\cong} \Autm(\k_1(K))$.
    \item The canonical isomorphism $\Kummer : \UAut(\Gc_K^a) \xrightarrow{\cong} \UAut(\k_1(K))$ restricts to an isomorphism $\Kummer : \UAutc(\Gc_K^a) \xrightarrow{\cong} \UAutm(\k_1(K))$ which is compatible with $\rho_{k_0}$.
  \end{enumerate}
\end{theorem}
Concerning the proof of Theorem \ref{theorem: cohom / kummer / galois-to-milnor}, the implication $(1) \Rightarrow (2)$ is clear, while the implication $(2) \Rightarrow (3)$ follows from the observations about cyclotomic twists made above.
It therefore remains to prove assertion (1), which was essentially already proven in \cite{Topaz:2014}*{Theorem 8.6}; alternatively, one can deduce (1) from the main results of \cite{Topaz:2015b}.
We give a summary of the proof of Theorem \ref{theorem: cohom / kummer / galois-to-milnor} below.
First, we need to recall some calculations in Galois cohomology.

\begin{fact}
\label{fact: cohom / kummer / generalities-kummer}
  The following hold:
  \begin{enumerate}
    \item The Bockstein morphism $\beta : \H^1(\Gal_K) \rightarrow \H^2(\Gal_K)$ is trivial.
    \item The inflation map $\H^*(\Gc_K) \rightarrow \H^*(\Gal_K)$ is an isomorphism for $* = 0,1,2$.
  \end{enumerate}
\end{fact}
\begin{proof}
  Assertion (2) is clear for $* = 0,1$, and for $* = 2$ it follows by considering the LHS spectral sequence associated to the extension
  \[ 1 \rightarrow \Gal_{K(\ell)} \rightarrow \Gal_K \rightarrow \Gc_K \rightarrow 1. \]
  See \cite{Topaz:2014}*{Fact 8.4} for the details.
  Assertion (1) follows from the fact that $\mu_{\ell^2} \subset K$; see \cite{Topaz:2014}*{Lemma 8.5} for the details.
\end{proof}

\begin{proof}[Proof of Theorem \ref{theorem: cohom / kummer / galois-to-milnor}]
  As noted above, it suffices to prove assertion (1) as assertions (2) and (3) follow from this.
  First, note that the Kummer pairing canonically induces a perfect pairing
  \[ \widehat\wedge^2(\Gc_K^a) \times \wedge^2 (\k_1(K)) \rightarrow \mu_\ell^{\otimes 2}. \]
  Let $\Rcal(\Gc_K)$ denote the kernel of the map
  \[ [\bullet,\bullet] : \widehat\wedge^2(\Gc_K^a) \rightarrow \Gc_K^{(2)}/\Gc_K^{(3)}. \]
  Since $\k_*(K)$ is a quadratic $\Z/\ell$-algebra, in order to prove assertion (1), it suffices to prove that the dual of the inclusion $\Rcal(\Gc_K) \hookrightarrow \widehat\wedge^2(\Gc_K^a)$ with respect to the pairing above is the (surjective) multiplication map 
  \[ \wedge^2 (\k_1(K)) \twoheadrightarrow \k_2(K) \]
  in the mod-$\ell$ Milnor K-ring.
  This follows easily by combining Fact \ref{fact: cohom / generalities / cup-vs-commutator}(3) with the Merkurjev-Suslin Theorem \cite{Merkurjev:1982} and Fact \ref{fact: cohom / kummer / generalities-kummer}, as follows.
  
  First, since $\mu_\ell \subset K$, we may simplify the notation and choose a fixed isomorphism $\mu_\ell \cong \Z/\ell$ of $\Gal_K$-modules.
  In particular, we may identify $\k_1(K)$ with $\H^1(\Gc_K^a) = \H^1(\Gc_K)$ via the Kummer pairing, and $\k_2(K)$ with $\H^2(\Gc_K)$ via the Merkurjev-Suslin Theorem \cite{Merkurjev:1982} and Fact \ref{fact: cohom / kummer / generalities-kummer}(2).

  Next, recall from Fact \ref{fact: cohom / kummer / generalities-kummer} that the Bockstein morphism $\beta : \H^1(\Gc_K) \rightarrow \H^2(\Gc_K)$ is trivial.
  Hence, the inflation map $\H^2(\Gc_K^a) \rightarrow \H^2(\Gc_K)$ factors through $\H^2(\Gc_K^a)/\beta\H^1(\Gc_K^a)$, which is isomorphic to $\wedge^2(\k_1(K))$ by Fact \ref{fact: cohom / generalities / cup-vs-commutator} and our identification $\H^1(\Gc_K^a) = \k_1(K)$ above.

  Finally, using the LHS spectral sequence and the fact that $\H^2(\Gc_K)_{\rm dec} = \H^2(\Gc_K)$ by the Merkurjev-Suslin Theorem and Fact \ref{fact: cohom / kummer / generalities-kummer}(2), we have a canonical short exact sequence
  \[ 0 \rightarrow \H^1(\Gc_K^{(2)})^{\Gc_K} \xrightarrow{d_2} \H^2(\Gc_K^a) \rightarrow \H^2(\Gc_K) \rightarrow 0, \]
  which induces an exact sequence
  \[ \H^1(\Gc_K^{(2)}/\Gc_K^{(3)}) \xrightarrow{-d_2} \wedge^2(\H^1(\Gc_K^a)) \rightarrow \H^2(\Gc_K) \rightarrow 0 \]
  by Fact \ref{fact: cohom / generalities / cup-vs-commutator}.
  Once we identify $\H^1(\Gc_K^a)$ with $\k_1(K)$ and $\H^2(\Gc_K)$ with $\k_2(K)$, the right-hand map of this sequence is precisely the multiplication map.
  Finally, the dual of this sequence is precisely 
  \[ 0 \rightarrow \Rcal(\Gc_K) \rightarrow \widehat\wedge^2(\Gc_K^a) \xrightarrow{[\bullet,\bullet]} \Gc_K^{(2)}/\Gc_K^{(3)} \]
  by Fact \ref{fact: cohom / generalities / cup-vs-commutator}(3).
\end{proof}

\section{The Local Theory}
\label{section: localthy}


In this section, we recall the required results from the \emph{Local Theory} in ``almost-abelian'' anabelian geometry for function fields over algebraically closed fields.
All such results have been generally stated for abelian-by-central Galois groups in the literature.
But since we work primarily with the ``Kummer dual'' of this context, i.e. with the mod-$\ell$ Milnor K-ring, we will need to translate these results to the context of Milnor K-theory via Theorem \ref{theorem: cohom / kummer / galois-to-milnor}.

\subsection{Minimized Galois Theory and C-Pairs}
\label{subsection: localthy / c-pairs}

For a field $F$, we define the {\bf $\ell$-minimized Galois group} of $F$ as
\[ \gfrak(F) := \Hom(F^\times,\Z/\ell), \]
endowed with its canonical structure of an abelian pro-$\ell$ group arising from the point-wise convergence topology.
If $\Char F \neq \ell$ and $\mu_\ell \subset F$, then Kummer theory together with a choice of isomorphism of $\Gal_F$-modules $\mu_\ell \cong \Z/\ell$ induces canonically an isomorphism of pro-$\ell$ groups
\[ \gfrak(F) \cong \Gc_F^a. \]
However, the main benefit of working with the minimized context is that it applies also to fields of characteristic $\ell$.
Since we will need to consider valuations of $K$ which \emph{a priori} have residue characteristic $\ell$, we \emph{must} work in the minimized context.

A pair of elements $\sigma,\tau \in \gfrak(F)$ is called a {\bf C-pair} provided that for all $x \in F \smallsetminus \{0,1\}$, one has
\[ \sigma(x)\cdot\tau(1-x) = \sigma(1-x)\cdot\tau(x). \]
If $F_1,F_2$ are two fields and $\phi : \gfrak(F_1) \rightarrow \gfrak(F_2)$ is an isomorphism of pro-$\ell$ groups, we say that $\phi$ is {\bf compatible with C-pairs} if, for all $\sigma,\tau \in \gfrak(F_1)$, the following are equivalent
\begin{enumerate}
  \item $\sigma,\tau$ forms a C-pair in $\gfrak(F_1)$.
  \item $\phi\sigma,\phi\tau$ forms a C-pair in $\gfrak(F_2)$.
\end{enumerate}

In the case where $\Char F \neq \ell$ and $\mu_{2\ell} \subset F$, it turns out that C-pairs can be determined in a Galois-theoretical manner, as the following fact describes.
For a detailed proof of this fact, see \cite{Topaz:2013a}*{Theorem 12} and \cite{Topaz:2015b}*{Theorem 4}.

\begin{fact}
\label{fact: localthy / c-pairs / c-vs-cl}
  Let $F$ be a field such that $\Char F \neq \ell$ and $\mu_{2\ell} \subset F$, and identify $\gfrak(F)$ with $\Gc_F^a$ via some isomorphism $\mu_\ell \cong \Z/\ell$, as above.
  Let $\sigma,\tau \in \gfrak(F) = \Gc_F^a$ be given.
  Then the following are equivalent:
  \begin{enumerate}
    \item $\sigma,\tau$ form a C-pair.
    \item One has $[\sigma,\tau] = 0$.
  \end{enumerate}
  In particular, if $\phi$ is an element of $\Autc(\Gc_F^a)$, then $\phi$ is compatible with C-pairs.
\end{fact}

\subsection{Minimized Decomposition Theory}
\label{subsection: localthy / mindec}

For a valuation $v$ of $F$, we consider the {\bf minimized inertia and decomposition} groups of $v$, which are defined as follows
\[ \I_v = \Hom(F^\times/\U_v,\Z/\ell), \ \D_v = \Hom(F^\times/\U_v^1,\Z/\ell). \]
Note that one has $\I_v \subset \D_v \subset \gfrak(F)$, and both $\I_v$ and $\D_v$ are closed subgroups of $\gfrak(F)$.

For $\sigma \in \D_v$, considered as a homomorphism $\sigma : F^\times/\U_v^1 \rightarrow \Z/\ell$, we write $\sigma_v$ for the restriction of $\sigma$ to $\U_v/\U_v^1 = Fv^\times$.
In particular, the map $\sigma \mapsto \sigma_v$ yields a canonical morphism $\D_v \rightarrow \gfrak(Fv)$.

\begin{fact}
\label{fact: localthy / mindec / mindec}
  Let $(F,v)$ be a valued field, and let $w$ be a valuation of $Fv$.
  Then the following hold:
  \begin{enumerate}
    \item The canonical map $\sigma \mapsto \sigma_v : \D_v \rightarrow \gfrak(Fv)$ induces an isomorphism $\D_v/\I_v \cong \gfrak(Fv)$.
    \item One has $\I_v \subset \I_{w \circ v} \subset \D_{w \circ v} \subset \D_v$.
    \item Identifying $\D_v/\I_v$ with $\gfrak(Fv)$ as in (1), one has $\D_{w \circ v}/\I_v = \D_w$ and $\I_{w \circ v}/\I_v = \I_w$.
  \end{enumerate}
\end{fact}
\begin{proof}
  Assertion (2) is clear, while assertions (1) and (3) are \cite{Topaz:2015}*{Lemma 2.1(1)(2)}.
\end{proof}

It turns out that the C-pair condition in the minimized decomposition group of a valuation can actually be completely determined from the residue field.
We summarize this property in the following fact.

\begin{fact}
\label{fact: localthy / mindec / c-pair-mindec}
  Let $(F,v)$ be a valued field and let $\sigma,\tau \in \D_v$ be given such that $\sigma(-1)=\tau(-1) = 0$.
  Then the following are equivalent
  \begin{enumerate}
    \item $\sigma,\tau$ form a C-pair in $\gfrak(F)$.
    \item $\sigma_v,\tau_v$ form a C-pair in $\gfrak(Fv)$.
  \end{enumerate}
\end{fact}
\begin{proof}
  See \cite{Topaz:2015}*{Lemma 2.1(3)}.
\end{proof}

\subsection{Quasi-Divisorial Valuations}
\label{subsection: localthy / qpd}

A valuation $v$ of $K$ is called a {\bf quasi-divisorial valuation of $K|k$} provided that $v$ satisfies the following properties:
\begin{enumerate}
  \item $vK$ contains no non-trivial $\ell$-divisible convex subgroups.
  \item $vK/vk$ is isomorphic to $\Z$ as an abstract group.
  \item $\trdeg(K|k)-1 = \trdeg(Kv|kv)$.
\end{enumerate}
For a quasi-divisorial valuation $v$ of $K|k$, it turns out that the residue field $Kv$ is a function field of transcendence degree $\trdeg(K|k)-1$ over $kv$ (cf. \cite{Pop:2006}*{Fact 5.5}).
In particular we can consider quasi-divisorial valuations $w$ of $Kv|kv$ as well, as long as $\trdeg(Kv|kv) = \trdeg(K|k)-1 \geq 1$.
For $r \leq \trdeg(K|k)$, we say that a valuation $v$ is an {\bf $r$-quasi-divisorial valuation of $K|k$} if $v$ is the valuation-theoretic composition of $r$ quasi-divisorial valuations as described above.

We will sometimes want to keep track of each term in the composition defining an $r$-quasi-divisorial valuation.
In such cases, we will generally consider the {\bf flag of quasi-divisorial valuations}:
\[ \vbf = (v_1,\ldots,v_r) \]
associated to an $r$-quasi-divisorial valuation $v_r$ of $K|k$.
Namely, $v_i/v_{i-1}$ is a quasi-divisorial valuation on $Kv_{i-1}|kv_{i-1}$ for all $i = 1,\ldots,r$; here and throughout, we write $v_0$ for the trivial valuation so that $Kv_0 = K$.
In particular, for all $i =1,\ldots,r$, the valuation $v_i$ is an $i$-quasi-divisorial valuation of $K|k$ which refines $v_{i-1}$.

For any valuation $v$ of $K$, recall that we write
\[ \Ufrak_v = \{\U_v\}_K, \ \ \Ufrak_v^1 = \{\U_v^1\}_K. \]
This notation is compatible with the notation/terminology from \S\ref{subsection: localthy / mindec}, as follows.
Under the Kummer pairing $\Gc_K^a \times \k_1(K) \rightarrow \mu_\ell$, the subgroup $\Ufrak_v$ resp. $\Ufrak_v^1$ of $\k_1(K)$ is precisely the orthogonal to the minimized inertia group $\I_v$ resp. minimized decomposition group $\D_v$ of $v$.

The main results concerning the \emph{local theory} (in the abelian-by-central setting) state that the minimized inertia and decomposition groups of $r$-quasi-divisorial valuations are preserved under the action of elements of $\Autc(\Gc_K^a)$, and that the partially ordered structure of the associated valuations is preserved as well.
We present the following theorem merely as a translation of these results which replaces $\I_v$ resp. $\D_v$ with $\Ufrak_v$ resp. $\Ufrak_v^1$, and $\Autc(\Gc_K^a)$ with $\Autm(\k_1(K))$.
The fact that elements of $\Autc(\Gc_K^a)$ preserve minimized decomposition/inertia groups of quasi-divisorial valuations was first proven by {\sc Pop} in \cite{Pop:2010a}, which uses the Galois context and the usual inertia/decomposition groups.
Nevertheless, the key arguments from loc.~cit. can be made to work in the minimized context as well.
In any case, in order to keep things precise, we will instead use the reference \cite{Topaz:2015} which uses the minimized context exclusively.

\begin{theorem}
\label{theorem: localthy / qpd / main-qpd}
  Let $\sigma \in \Autm(\k_1(K))$ be given.
  Then the following hold:
  \begin{enumerate}
    \item Let $v$ be an $r$-quasi-divisorial valuation of $K|k$, with $r < \trdeg(K|k)$.
    Then there exists a unique $r$-quasi-divisorial valuation $v^\sigma$ of $K|k$ such that
    \[ \sigma\Ufrak_v = \Ufrak_{v^\sigma}, \ \ \sigma\Ufrak_v^1 = \Ufrak_{v^\sigma}^1. \]
    \item Let $v$ be an $r$-quasi-divisorial valuation of $K|k$ and $w$ an $s$-quasi-divisorial valuation of $K|k$, with $r,s < \trdeg(K|k)$, and let $v^\sigma,w^\sigma$ be as in (1) above.
    Then $v$ is a coarsening of $w$ if and only if $v^\sigma$ is a coarsening of $w^\sigma$.
  \end{enumerate}
\end{theorem}
\begin{proof}
  Before we begin the proof, we choose an isomorphism $\mu_\ell \cong \Z/\ell$ which will be fixed throughout, and which allows us to identify $\Gc_K^a$ with $\gfrak(K)$ by Kummer theory.

  \vskip 5pt
  \emph{Proof of (1):}
  Let $\sigma^*$ denote the element of $\Autc(\Gc_K^a)$ which is the dual of $\sigma$ under the Kummer pairing (see Theorem \ref{theorem: cohom / kummer / galois-to-milnor}(1)).
  Recall that $\Ufrak_v$ resp. $\Ufrak_v^1$ is the orthogonal of $\I_v$ resp. $\D_v$ under this Kummer pairing
  \[ \Gc_K^a \times \k_1(K) \rightarrow \mu_\ell \cong \Z/\ell. \]
  Therefore, it suffices to show that there exists an $r$-quasi-divisorial valuation $v^\sigma$ such that
  \[ \sigma^*\I_v = \I_{v^\sigma}, \ \ \sigma^*\D_v = \D_{v^\sigma}, \]
  as the uniqueness of this $v^\sigma$ will be discussed at the end of the proof.

  Recall from Theorem \ref{theorem: cohom / kummer / galois-to-milnor}(1) that $\sigma^*$ is compatible with $\Rcal$, hence $\sigma^*$ is compatible with C-pairs by Fact \ref{fact: localthy / c-pairs / c-vs-cl}.
  The proof now proceeds by induction on $r$.
  The base case $r = 1$ follows immediately from \cite{Topaz:2015}*{Theorem D} and/or \cite{Pop:2010a}*{Theorem 1}.

  The inductive case proceeds as follows.
  First, write an $r$-quasi-divisorial valuation $v$ as $v = w \circ w_0$ with $w_0$ an $(r-1)$-quasi-divisorial valuation and $w$ a quasi-divisorial valuation of $Kw_0|kw_0$.
  By induction, there exists an $(r-1)$-quasi-divisorial valuation $w_1 = w_0^\sigma$ such that
  \[ \sigma^* \D_{w_0} = \D_{w_1},\ \ \sigma^*\I_{w_0} = \I_{w_1}. \]
  Thus, $\sigma^*$ induces an isomorphism
  \[ \sigma^* : \gfrak(Kw_0) = \D_{w_0}/\I_{w_0} \xrightarrow{\cong} \D_{w_1}/\I_{w_1} = \gfrak(Kw_1) \]
  by Fact \ref{fact: localthy / mindec / mindec}.
  Also, note that $\trdeg(Kw_0|kw_0) \geq 2$ by assumption.
  By Fact \ref{fact: localthy / mindec / c-pair-mindec}, this isomorphism $\sigma^* : \gfrak(Kw_0) \rightarrow \gfrak(Kw_1)$ is compatible with C-pairs.
  Using \cite{Topaz:2015}*{Theorem D} again, it follows that there exists a quasi-divisorial valuation $w'$ of $Kw_1|kw_1$ such that $\sigma^* \I_w = \I_{w'}$ and $\sigma^* \D_w = \D_{w'}$.
  Letting $v^\sigma := w' \circ w_1$, it follows from Fact \ref{fact: localthy / mindec / mindec} that $\sigma^* \I_v = \I_{v^\sigma}$ and $\sigma^*\D_w = \D_{v^\sigma}$, as required.

  \vskip 5pt
  \emph{Proof of (2):}
  Since $vK$ and $wK$ contain no non-trivial $\ell$-divisible convex subgroups, it follows from \cite{Topaz:2015}*{Lemma 3.1} or \cite{Topaz:2013a}*{Lemma 3.4} that $v$ is a coarsening of $w$ if and only if $\I_v \subset \I_w$.
  By assertion (1), this is true if and only if $\I_{v^\sigma} \subset \I_{w^\sigma}$, which is similarly equivalent to $v^\sigma$ being a coarsening of $w^\sigma$.
  This observation also implies that $v^\sigma$ is uniquely determined by $v$ and by $\sigma$.
\end{proof}

We will continue to use the notation $v^\sigma$ which was implicitly introduced in Theorem \ref{theorem: localthy / qpd / main-qpd}.
In other words, for an $r$-quasi-divisorial valuation $v$ of $K|k$ with $r < \trdeg(K|k)$, and $\sigma \in \Autm(\k_1(K))$, we denote by $v^\sigma$ the unique $r$-quasi-divisorial valuation of $K|k$ such that
\[ \sigma \Ufrak_v  = \Ufrak_{v^\sigma}, \ \ \sigma\Ufrak_v^1 = \Ufrak_{v^\sigma}^1.\]
Since $v^\sigma$ is compatible with coarsening/refinement of $v$, we may also use this notation for flags of quasi-divisorial valuations.
Namely, if $(v_1,\ldots,v_r)$ is a flag of quasi-divisorial valuations of $K|k$, then $(v_1^\sigma,\ldots,v_r^\sigma)$ is again a flag of quasi-divisorial valuations of $K|k$.

\subsection{Divisorial Valuations}
\label{subsection: localthy / pd}

An $r$-quasi-divisorial valuation is called {\bf $r$-divisorial} provided that the restriction of $v$ to $k$ is trivial.
As in the quasi-divisorial setting, we may wish to consider a \emph{flag} of divisorial valuations
\[ \vbf = (v_1,\ldots,v_r) \]
where $v_i/v_{i-1}$ is a divisorial valuation of $Kv_{i-1}|kv_{i-1}$ for all $i \geq 1$; as before, here $v_0$ denotes the trivial valuation of $K$.

It turns out (using the valuative criterion for properness) that any divisorial valuation indeed arises from some Weil-prime-divisor on some normal model of $K|k$.
More generally, any $r$-divisorial valuation arises from a \emph{flag} of Weil-prime-divisors on some normal model.
See \cite{Pop:2006}*{\S5} for more details on this.

In general, it is still a major open question to determine which $r$-quasi-divisorial valuations are actually $r$-divisorial, using the group-theoretical structure of $\Gc_K^c$ resp. the ring structure of $\k_*(K)$.
Nevertheless, in our context, we can use \emph{geometric subgroups} (as defined in \S\ref{section: summary}) which arise from $K_0$ to distinguish the $r$-divisorial valuations among the $r$-quasi-divisorial ones.
See also \cite{Pop:2010a}*{Theorem 19} for a related result which uses a similar argument.

\begin{proposition}
\label{proposition: localthy / pd / pd-detection}
  Let $v$ be an $r$-quasi-divisorial valuation of $K|k$, with $r < \trdeg(K|k)$.
  Then the following are equivalent:
  \begin{enumerate}
    \item The valuation $v$ is an $r$-divisorial valuation of $K|k$.
    \item There exists some $t \in K_0 \smallsetminus k_0$ such that the intersection $\Kfrak(t) \cap \Ufrak_v^1$ is finite.
  \end{enumerate}
\end{proposition}
\begin{proof}
  First of all, note that $K|K_0$ and $k|k_0$ are both algebraic extensions.
  Thus, the associated residue extensions, $Kv|K_0v$ and $kv|k_0v$, are also algebraic.
  Moreover, since $k$ is algebraically closed, the residue field $kv$ is also algebraically closed.
  In particular, if $\trdeg(Kv|kv) \geq 1$, then $K_0v$ cannot be contained in $kv$, and therefore there exist (many) elements $t$ in $\U_v \cap K_0^\times$ whose image $\bar t$ in $Kv$ is transcendental over $kv$.

  Now assume that $v$ is $r$-divisorial, hence $v$ is trivial on $k$.
  Since $r < \trdeg(K|k)$, it follows that $Kv$ is a function field of transcendence degree $\geq 1$ over $k$.
  Arguing as above, we see that there exists some element
  \[ t \in (\U_v \cap K_0^\times) \smallsetminus k_0 \]
  such that the image $\bar t$ of $t$ in $Kv$ is transcendental over $k$.
  This implies that $\Kbb(t)^\times \subset \U_v$ and that the image of $\Kbb(t)$ in $Kv$ is a function field of transcendence degree $1$ over $k$.
  Let $L'$ denote the image of $\Kbb(t)$ in $Kv$ and let $L$ denote the relative algebraic closure of $L'$ in $Kv$.
  Note that $L$ is a finite extension of $L'$, and that $\Kbb(t) \rightarrow L'$ is an isomorphism.

  If we consider the canonical map
  \[ \Kfrak(t) \hookrightarrow \Ufrak_v \twoheadrightarrow \Ufrak_v / \Ufrak_v^1 = \k_1(Kv) \]
  then its kernel is the same as the kernel of $\Kfrak(t) \cong \k_1(L') \rightarrow \k_1(L)$, which is finite by Kummer theory.
  On the other hand, this kernel is precisely $\Ufrak_v^1 \cap \Kfrak(t)$.

  Conversely, assume that $v$ is not $r$-divisorial, hence $v$ is non-trivial on $k$, and let $t \in K_0 \smallsetminus k_0$ be given.
  We must show that $\Kfrak(t) \cap \Ufrak_v^1$ is infinite.
  By replacing $t$ with $t^{-1}$ if needed, we may assume, without loss of generality, that $v(t) \geq 0$.

  Now let $a \in k^\times$ be any element such that $v(a) < 0$.
  For all such $a$, one has $v(a^{-1} t) > 0$, and therefore
  \[ \{a-t\}_K = \{1-a^{-1}t\}_K \in \Ufrak_v^1 \cap \Kfrak(t).\]
  On the other hand, there are infinitely many such $a \in k^\times$, hence the set
  \[ \{\{a-t\}_{k(t)} \ : \ a \in k^\times, \ v(a) < 0 \} \]
  is an infinite linearly-independent subset of $\k_1(k(t))$.
  Since the map given by the composition $\k_1(k(t)) \rightarrow \k_1(\Kbb(t)) \cong \Kfrak(t)$ has a finite kernel by Kummer theory, it follows that $\Ufrak_v^1 \cap \Kfrak(t)$ is infinite, as required.
\end{proof}

\section{Milnor K-Theory of Function Fields}
\label{section: milnorff}


In this section, we will prove some important vanishing and non-vanishing results for the mod-$\ell$ Milnor K-ring of a function field.
In particular, we will show that a significant portion of the ``algebraic-independence'' structure of a function field $K$ over an algebraically closed field $k$ is encoded in the mod-$\ell$ Milnor K-ring of $K$.

\subsection{Vanishing}
\label{subsection: milnorff / vanishing}

We begin by recalling the following fact which follows from some well-known cohomological dimension bounds, combined with the highly celebrated Voevodsky-Rost Theorem \cite{Voevodsky:2011} \cite{Rost:1998} \cite{Weibel:2009}.

\begin{fact}
\label{fact: milnorff / vanishing / vanishing-B-K}
  One has $\k_s(K) = 0$ for all $s > \trdeg(K|k)$.
\end{fact}
\begin{proof}
  It is well-known that the field $K$ has $\ell$-cohomological-dimension $\leq \trdeg(K|k)$ (cf. \cite{Serre:2002}*{II.\S4 Proposition 11}).
  Therefore, it follows from the Voevodsky-Rost Theorem \cite{Voevodsky:2011} \cite{Rost:1998} \cite{Weibel:2009} that
  \[ \k_s(K) \cong \H^s(K,\mu_\ell^{\otimes s}) = 0\]
  for all $s > \trdeg(K|k)$.
\end{proof}

\subsection{Non-Vanishing}
\label{subsection: milnorff / nonvanishing}

The non-vanishing results in mod-$\ell$ Milnor K-theory of function fields follow the ``yoga'' that algebraically independent elements should have non-trivial products in Milnor K-theory.
This turns out to be true, but with various exceptions which arise from modding out by $\ell$-th powers.
Nevertheless, it turns out that these exceptions can be avoided because they are related to some ramification phenomena which are concentrated in codimension one.

\begin{lemma}
\label{lemma: milnorff / nonvanishing / alg-indep-coords}
  Let $t_1,\ldots,t_r \in K$ be algebraically independent over $k$.
  Then there exists a non-empty open subset $U$ of $\Abb^r_k$ such that, for all $(a_1,\ldots,a_r) \in U(k)$, one has 
  \[ \{t_1-a_1,\ldots,t_r-a_r\}_K \neq 0. \]
\end{lemma}
\begin{proof}
  By extending $t_1,\ldots,t_r$ to a transcendence base $\tbf = (t_1,\ldots,t_d)$ for $K|k$, we may assume without loss of generality that $r = d = \trdeg(K|k)$.
  Let $X = \Abb^d_\tbf$ denote affine $d$-space over $k$ with parameters $\tbf$, so that $k(\tbf)$ is canonically identified with the function field of $X$.
  Let $K'$ be the maximal separable subextension of $K|k(\tbf)$, and let $Y$ denote the normalization of $X$ in $K'$.

  Then $Y \rightarrow X$ is a finite separable (possibly branched) cover of $k$-varieties.
  Since the branch locus of $Y \rightarrow X$ has codimension one, there exists a non-empty open subset $U$ of $X$ such that $Y$ is unramified over $U$.

  Let $x \in U$ be a given closed point.
  Then $x$ corresponds to a $k$-rational point
  \[ (a_1,\ldots,a_d) \in U(k) \subset X(k) = k^d, \]
  and in this case $(t_1-a_1,\ldots,t_d-a_d)$ is a system of regular parameters at $x$.

  If $y$ is any closed point of $Y$ lying above $x$, then $(t_1-a_1,\ldots,t_d-a_d)$ is also a system of regular parameters at $y$, since $y$ is unramified over $x$.
  By taking the $\mf_y$-adic completion at $y$, we obtain a $k$-embedding of $K'$ into the field of Laurent series $k((T_1,\ldots,T_d)) =: L$ which sends $t_i-a_i$ to $T_i$.
  But $\{T_1,\ldots,T_d\}_L \neq 0$ by Fact \ref{fact: milnor / tame / laurent-coords}, so it follows that
  \[ \{t_1-a_1,\ldots,t_d-a_d\}_{K'} \neq 0. \]
  Finally, one has $\{t_1-a_1,\ldots,t_d-a_d\}_K \neq 0$ as well by Lemma \ref{lemma: milnor / insep / insep-iso}.
\end{proof}

The following proposition is our second main ``non-vanishing'' result, and it will be crucial in describing geometric subgroups of $\k_1(K)$.
This will play a primary role in the reconstruction of the ``geometric lattice'' in the proof of the main theorems.

\begin{proposition}
\label{proposition: milnorff / geometric / geometric-maximal}
  Let $t_1,\ldots,t_r \in K$ be algebraically independent over $k$.
  Let $z \in K^\times$ be such that $\{z\}_K \notin \Kfrak(t_1,\ldots,t_r)$.
  Then there exists a non-empty open subset $U$ of $\Abb^r_k$ such that for all $(a_1,\ldots,a_r) \in U(k)$, one has
  \[ \{t_1-a_1,\ldots,t_r-a_r,z\}_K \neq 0. \]
\end{proposition}
\begin{proof}
  Put $d = \trdeg(K|k)$ and $s := d-r$.
  Let $\tbf = (t_1,\ldots,t_r,y_1,\ldots,y_s)$ be a transcendence base of $K|k$ extending $t_1,\ldots,t_r$.
  Let $K'$ denote the maximal separable subextension of $K|k(\tbf)$.
  By Lemma \ref{lemma: milnor / insep / insep-iso}, there exists some $z' \in K'$ such that $\{z'\}_K = \{z\}_K$.
  By replacing $z$ with such a $z'$, we may assume without loss of generality that $z \in K'$.

  Let $F'$ denote the relative algebraic closure of $k(t_1,\ldots,t_r)$ in $K'$.
  In particular, note that $F'$ is separable over $k(t_1,\ldots,t_r)$ and that $F := \Kbb(t_1,\ldots,t_r)$ is purely inseparable over $F'$.
  Therefore, one has $\{z\}_{K'} \notin \{(F')^\times\}_{K'}$ by Lemma \ref{lemma: milnor / insep / insep-iso}.
  This implies that the field $F'$ is relatively algebraically closed also in $K'(\sqrt[\ell]{z}) =: L$.
  And since $L$ is separable over $k(\tbf)$, it follows that $L$ is regular over $F'$.

  Let $B$ denote the normalization of $\Abb^r_{t_1,\ldots,t_r}$ in $F'$.
  Moreover, we let $X_0$ denote $\Abb^s_{B;y_1,\ldots,y_s}$, an affine $s$-space over $B$ with coordinates $(y_1,\ldots,y_s)$.
  Finally, let $X_1$ denote the normalization of $X_0$ in $K'$, and let $X_2$ denote the normalization of $X_0$ in $L$.

  We will now pass to a sufficiently small non-empty open subset $V$ of $B$ which has the six properties listed below.
  This is possible since each one of these properties holds on a dense open subset of $B$, as described in each item below.
  \begin{enumerate}[leftmargin=*]
    \item The fibers of $X_i \rightarrow B$, $i = 0,1,2$, over points of $V$ are all geometrically integral.
    This is an open condition on $B$ since $k(X_i)$ is regular over $F' = k(B)$, hence $X_i \rightarrow B$ has generically geometrically integral fibers.
    \item Any point of $V$ is unramified over $\Abb^r_{t_1,\ldots,t_r}$.
    This is an open condition on $B$ because the extension $F'|k(t_1,\ldots,t_r)$ is finite and separable, hence the ramification locus of $B \rightarrow \Abb^r_{t_1,\ldots,t_r}$ has codimension one.
    \item For all $x \in V$, the function $z$ is regular and non-zero on the generic point of $(X_1)_x$, the fiber of $X_1 \rightarrow B$ over $x$.
    This is clearly an open condition on $B$, since the support of $z$ has codimension one in $X_1$, and $X_0$ is flat over $B$ of relative dimension $s$.
    \item For all $x \in V$, letting $\bar z$ denote the image of $z$ in $k((X_1)_x)$, one has
    \[ k((X_2)_x) = k((X_1)_x)[\sqrt[\ell]{\bar z}]. \]
    This is an open condition on $B$ because $L = K'[\sqrt[\ell]{z}]$, hence for a sufficiently small non-empty affine open subset of $X_1$, say $\Spec A$, one has $z \in A$ and the normalization of $A$ in $L$ is precisely $A[\sqrt[\ell]{z}]$.
    \item For all $x \in V$, one has $[k(X_2):k(X_0)] = [k((X_2)_x):k((X_0)_x)]$, and therefore one also has $\ell = [k((X_2)_x):k((X_1)_x)]$ and $[k(X_1):k(X_0)] = [k((X_1)_x):k((X_0)_x)]$.
    This is an open condition on $B$ by a standard application of the Bertini-Noether theorem (cf. \cite{Fried:2006}*{Proposition 9.4.3}).
    \item For all $x \in V$, letting $\eta^i_x$ denote the generic point of the fiber $(X_i)_x$ over $x$, the point $\eta^1_x$ is unramified over $\eta^0_x$.
    This is clearly an open condition on $B$, since the branch locus of $X_1 \rightarrow X_0$ is a proper closed subset of $X_0$, and $X_0$ is flat over $B$.
  \end{enumerate}

  Let $U$ denote the image of $V$ in $\Abb^r_{t_1,\ldots,t_r}$, and note that $U$ is a non-empty (hence dense) open subset of $\Abb^r_{t_1,\ldots,t_r}$.
  Suppose that $x \in U$ is a closed point with associated rational point $(a_1,\ldots,a_r) \in U(k) \subset k^r$, and let $y \in V$ be a closed point lying above $x$.
  Then $(t_1-a_1,\ldots,t_r-a_r)$ is a system of regular parameters at $x$.
  Since $y$ is unramified over $x$ by condition (2) above, the tuple $(t_1-a_1,\ldots,t_r-a_r)$ is also a system of regular parameters at $y$.
  Since $X_0 = \Abb^s_B$, we see that $(t_1-a_1,\ldots,t_r-a_r)$ is a system of regular parameters for the generic point of the fiber $(X_0)_y$.
  By condition (6), $(t_1-a_1,\ldots,t_r-a_r)$ is also a system of regular parameters for the generic point $\eta_y$ of the fiber $(X_1)_y$, which is integral by condition (1).
  Finally, by (4) and (5), we know that $\{\bar z\}_{k(\eta_y)} \neq 0$ as an element of $\k_1(k(\eta_y))$.
  Taking the completion with respect to the maximal ideal of $\Oc_{X_1,\eta_y}$, we obtain an embedding of $K'$ into the field of Laurent series
  \[ M := k(\eta_y)((T_1,\ldots,T_r)) \]
  which sends $t_i-a_i$ to $T_i$.
  Fact \ref{fact: milnor / tame / laurent-coords} implies that $\{T_1,\ldots,T_r,z\}_M \neq 0$, hence
  \[ \{t_1-a_1,\ldots,t_r-a_r,z\}_{K'} \neq 0.\]
  Finally, one has $\{t_1-a_1,\ldots,t_r-a_r,z\}_K \neq 0$ by Lemma \ref{lemma: milnor / insep / insep-iso}, and this concludes the proof of the proposition.
\end{proof}

As a corollary, we obtain our primary criterion for constructing geometric subgroups of $\k_1(K)$.
Since it appears in the statement of the corollary, we briefly note that, given $r$ elements $f_1,\ldots,f_r$ of $\k_1(K)$, one obtains a canonical morphism
\[ \{f_1,\ldots,f_r,\bullet\}_K : \k_1(K) \rightarrow \k_{r+1}(K). \]
The description of geometric subgroups of $\k_1(K)$ will use the \emph{kernels} of such morphisms.

\begin{corollary}
\label{corollary: milnorff / geometric / infinite-subsets}
  Let $t_1,\ldots,t_r \in K$ be algebraically independent over $k$.
  Let $S_1,\ldots,S_r$ be arbitrary infinite subsets of $k$.
  Then one has
  \[ \Kfrak(t_1,\ldots,t_r) = \bigcap_{(a_1,\ldots,a_r)} \ker\{t_1-a_1,\ldots,t_r-a_r,\bullet\}_K \]
  where $(a_1,\ldots,a_r)$ varies over the elements of $S_1 \times \cdots \times S_r$.
\end{corollary}
\begin{proof}
  It is clear that
  \[ \Kfrak(t_1,\ldots,t_r) \subset \bigcap_{(a_1,\ldots,a_r)} \ker\{t_1-a_1,\ldots,t_r-a_r,\bullet\}_K, \]
  where $(a_1,\ldots,a_r)$ varies over $S_1 \times \cdots \times S_r$, by Fact \ref{fact: milnorff / vanishing / vanishing-B-K}.
  The converse follows immediately from Proposition \ref{proposition: milnorff / geometric / geometric-maximal}, since the fact that $S_i$ are all infinite, implies that the set
  \[ (S_1 \times \cdots \times S_r) \cap U(k) \]
  is non-empty for any non-empty open subset $U$ of $\Abb^r_k$.
\end{proof}

\subsection{Milnor Dimension}
\label{subsection: milnorff / milnordim}

For a subset $S$ of $\k_1(K)$, we define the {\bf Milnor dimension} of $S$, denoted $\dimm(S)$, as the minimal element of $\Z_{\geq 0}$ which satisfies
\[ \dimm(S) \geq r \Longleftrightarrow \exists s_1,\ldots,s_r \in S, \ \{s_1,\ldots,s_r\}_K \neq 0. \]
Since $\k_r(K) = 0$ for all $r > \trdeg(K|k)$ by Fact \ref{fact: milnorff / vanishing / vanishing-B-K}, it follows that $\dimm(S)$ is a well-defined element of $\Z_{\geq 0}$ and that $\dimm(S) \leq \trdeg(K|k)$, for any subset $S$ of $\k_1(K)$.
The following fact follows immediately from Fact \ref{fact: milnorff / vanishing / vanishing-B-K} and Lemma \ref{lemma: milnorff / nonvanishing / alg-indep-coords}.

\begin{fact}
\label{fact: milnorff / milnordim / geometric-dim}
  Let $S$ be a subset of $K$.
  Then one has $\dimm(\Kfrak(S)) = \trdeg(\Kbb(S)|k)$.
\end{fact}

\subsection{Milnor-Supremum}
\label{subsection: milnorff / milnorsup}

We can now use the vanishing/non-vanishing results above and the notion of Milnor-dimension to show how to reconstruct large geometric subgroups from smaller ones.
Let $S$ be a subset of $\k_1(K)$ with $d = \dimm(S)$.
We define the {\bf Milnor-supremum} of $S$, denoted $\msup(S)$, as
\[ \msup(S) = \bigcap_{s_1,\ldots,s_d \in S} \ker\{s_1,\ldots,s_d,\bullet\}_K.\]
The following lemma shows how the Milnor-supremum can be used to construct ``large'' geometric subgroups from smaller geometric subgroups.

\begin{lemma}
\label{lemma: milnorff / milnordim / milnorsup}
  Let $(S_i)_i$ be a collection of subsets of $K$, and put $\Kfrak_i = \Kfrak(S_i)$.
  Moreover, put
  \[ S := \bigcup_i S_i, \ \ \Kfrak := \bigcup_i \Kfrak_i. \]
  Then one has $\msup(\Kfrak) = \Kfrak(S)$.
\end{lemma}
\begin{proof}
  Put $d' := \trdeg(\Kbb(S)|k)$ and $d := \dimm(\Kfrak)$.
  First we show that $d = d'$.
  The inequality $d \leq d'$ follows from Fact \ref{fact: milnorff / vanishing / vanishing-B-K} since $\Kfrak \subset \Kfrak(S)$.
  Conversely, we note that there exist $s_1,\ldots,s_{d'} \in S$ which are algebraically independent over $k$.
  Thus, $d \geq d'$ by Lemma \ref{lemma: milnorff / nonvanishing / alg-indep-coords}.

  Since $d = \trdeg(\Kbb(S)|k) = \dimm(\Kfrak(S))$ and $\Kfrak \subset \Kfrak(S)$, it follows again from Fact \ref{fact: milnorff / vanishing / vanishing-B-K} that
  \[ \Kfrak(S) \subset \bigcap_{x_1,\ldots,x_d \in \Kfrak} \ker\{x_1,\ldots,x_d,\bullet\}_K.\]
  Conversely, we again use the fact that there exist $t_1,\ldots,t_d \in S$ which are algebraically independent over $k$.
  Then the inclusion
  \[ \Kfrak(S) \supset \bigcap_{x_1,\ldots,x_d \in \Kfrak} \ker\{x_1,\ldots,x_d,\bullet\}_K\]
  follows from Corollary \ref{corollary: milnorff / geometric / infinite-subsets}.
\end{proof}

\section{Geometric Lattices}
\label{section: lattice}


A {\bf graded lattice} is a $\Z_{\geq 0}$-graded set $\Lcal^* = \coprod_{i \geq 0} \Lcal^i$ endowed with a partial ordering $\leq$, such that the following two conditions hold:
\begin{enumerate}
  \item Every subset $S$ of $\Lcal^*$ has a \emph{greatest lower bound} $\wedge S$ and a \emph{least upper bound} $\vee S$ in $\Lcal^*$. 
  Namely, the partially ordered set $(\Lcal^*,\leq)$ is a \emph{complete lattice}.
  \item If $A \in \Lcal^s$ and $B \in \Lcal^t$ are such that $A < B$, then one has $s < t$.
  Namely, the partial ordering $\leq$ is \emph{strictly compatible} with the grading of $\Lcal^*$.
\end{enumerate}

An isomorphism between two graded lattices $\Lcal_1^*$ and $\Lcal_2^*$ is a graded bijection
\[ f^* = \coprod_{i \geq 0} f^i : \Lcal_1^* \rightarrow \Lcal_2^*\] 
which is an isomorphism on the level of partially ordered sets. 
We will denote the automorphism group of a graded lattice $\Lcal^*$ by $\Aut^*(\Lcal^*)$.

\subsection{Relative Algebraic Closure}
\label{subsection: lattice / acl}

Assume now that $F$ is a perfect field and that $L$ is an extension of finite transcendence degree over $F$, such that $F$ is relatively algebraically closed in $L$.
We denote by $\G^*(L|F)$ the graded lattice of relatively algebraically closed subextensions of $L|F$, graded by transcendence degree.

More precisely, we let $\G^s(L|F)$ denote the collection of relatively algebraically closed subextensions of $L|F$ of transcendence degree $s$ over $F$.
Then one has $\G^*(L|F) = \coprod_{s \geq 0} \G^s(L|F)$, and the ordering on $\G^*(L|F)$ is the partial ordering given by inclusion of subextensions of $L|F$.

Note that one has a canonical isomorphism of graded lattices
\[ \G^*(L|F) \xrightarrow{\cong} \G^*(L^i|F) \]
defined by $M \mapsto M^i$, with inverse given by $M^i \mapsto M^i \cap L$.
In particular, this canonical isomorphism induces an isomorphism of automorphism groups:
\[ \Aut^*(\G(L|F)) \xrightarrow{\cong} \Aut^*(\G(L^i|F)). \]

We will denote by $\Aut_{\{F\}}(L^i)$ the set of automorphisms $\phi$ of $L^i$ which satisfy $\phi F = F$.
Note that any element of $\Aut_{\{F\}}(L^i)$ induces an automorphism of $\G^*(L^i|F)$ in the obvious way.
Thus we obtain a canonical homomorphism
\[ \Aut_{\{F\}}(L^i) \rightarrow \Aut^*(\G(L^i|F)) \cong \Aut^*(\G^*(L|F)). \]
The work of {\sc Evans-Hrushovski} \cite{Evans:1991} \cite{Evans:1995} and {\sc Gismatullin} \cite{Gismatullin:2008} considers the surjectivity/injectivity of this map.
We summarize their main results in the following theorem.

\begin{theorem}[see \cite{Gismatullin:2008}*{Theorem 4.2}]
\label{theorem: lattice / acl / fund-thm-field-acl}
  Let $F$ be a perfect field and let $L$ be an extension of finite transcendence degree over $F$, such that $F$ is relatively algebraically closed in $L$.
  Assume furthermore that $\trdeg(L|F) \geq 5$.
  Then the canonical map 
  \[ \Aut_{\{F\}}(L^i) \rightarrow \Aut^*(\G^*(L|F)) \]
  is surjective, and its kernel is the subgroup generated by $\Frob_{L^i}$.
\end{theorem}
\begin{proof}
  The only part of this theorem which doesn't follow directly from \cite{Gismatullin:2008}*{Theorem 4.2} is that loc.~cit. uses the \emph{combinatorial geometry} associated to $\Gbb^*(L|F)$ as its basic structure, whereas we use the whole lattice $\Gbb^*(L|F)$.
  These two formulations are easily seen to be equivalent, as follows.

  First, note that the combinatorial geometry $\Gbb(L|F)$ associated to $L|F$ is defined to be the set $\Gbb^1(L|F)$ with a \emph{closure operation} $\cl$ on subsets $S \subset \Gbb^1(L|F)$.
  This closure operation is easily interpretable in the lattice $\Gbb^*(L|F)$ as 
  \[ \cl(S) = \{A \in \Gbb^1(L|F) \ : \ A \leq \vee S\}. \]

  Conversely, given the set $\Gbb^1(L|F)$ endowed with the closure operation $\cl$ defined above, it is easy to see that $\Gbb^*(L|F)$ is canonically isomorphic to the \emph{lattice of flats} associated to the combinatorial geometry $(\Gbb^1(L|F),\cl)$.
  In other words, the lattice $\Gbb^*(L|F)$ can be identified with the lattice of \emph{closed subsets} of $\Gbb^1(L|F)$ with respect to the closure operation $\cl$ described above, and the grading of this lattice is given by the \emph{combinatorial dimension} of closed subsets.

  With these observations, it is easy to see that restricting $\phi \in \Aut^*(\Gbb^*(L|F))$ to $\Gbb^1(L|F)$ induces an isomorphism of automorphism groups
  \[ \phi \mapsto \phi|_{\G^1(L|F)} : \Aut^*(\Gbb^*(L|F)) \cong \Aut^{\cl}(\Gbb^1(L|F)) \]
  where $\Aut^{\cl}(\Gbb^1(L|F))$ denotes the set of permutations of $\Gbb^1(L|F)$ which are compatible with the closure operation $\cl$.
  This isomorphism of automorphism groups is clearly compatible with the canonical morphisms originating from $\Aut_{\{F\}}(L^i)$.
  Hence, the assertion of the theorem follows immediately from \cite{Gismatullin:2008}*{Theorem 4.2}.
\end{proof}

\subsection{The Galois Action on the Geometric Lattice}
\label{subsection: lattice / galois-acl}

We now consider the canonical Galois action of $\Galk$ on $\G^*(K|k)$.
First, recall that $\Galk$ can be canonically identified with the Galois group $\Gal(K|K_0)$, since $K_0|k_0$ is a regular function field.
In particular, $\Galk$ acts on the (relatively algebraically closed) subextensions of $K|k$ in the obvious way, and this action preserves transcendence degree.
In other words, $\Galk$ acts on the graded lattice $\Gbb^*(K|k)$.

Using the notation introduced above, we have
\[ \G^*(K|k) = \{\Kbb(S) \ : \ S \subset K \} \]
and $\G^r(K|k)$ consists of all those $\Kbb(S)$ such that $\trdeg(\Kbb(S)|k) = r$.
We will consider the subgroup $\Aut^*(\G^*(K|k)|K_0)$ of $\Aut^*(\G^*(K|k))$ consisting of all automorphisms $\Phi \in \Aut^*(\G^*(K|k))$ such that $\Phi \Kbb(S) = \Kbb(S)$ for all subsets $S \subset K_0$.

We can also describe $\Aut^*(\Gbb^*(K|k)|K_0)$ as a ``Galois Group,'' as follows.
First, observe that one has a canonical \emph{injective} map 
\[ \G^*(K_0|k_0) \rightarrow \G^*(K|k) \]
which sends $F_0 \in \G^*(K_0|k_0)$ to $F := \Kbb(F_0)$.
Therefore, we see that $\Aut^*(\G^*(K|k)|K_0)$ is the Galois group of $\G^*(K|k)$ over $\G^*(K_0|k_0)$, i.e. the set of automorphisms of $\G^*(K|k)$ which fix $\G^*(K_0|k_0)$ pointwise.

Note that $\Galk$ fixes $\Kbb(S)$ \emph{setwise} (as a subset of $K$) for any subset $S \subset K_0$.
Therefore, the action of $\Galk$ on $\Gbb^*(K|k)$ induces a canonical Galois representation 
\[ \rho_{k_0} : \Galk \rightarrow \Aut^*(\G^*(K|k)|K_0). \]
We now use Theorem \ref{theorem: lattice / acl / fund-thm-field-acl} to prove that this map is an isomorphism.

\begin{proposition}
\label{proposition: lattice / galois-acl / galois-acl}
  Assume that $\trdeg(K|k) \geq 5$.
  Then the canonical map 
  \[ \rho_{k_0} : \Galk \rightarrow \Aut^*(\G^*(K|k)|K_0) \]
  is an isomorphism.
\end{proposition}
\begin{proof}
  We first use Theorem \ref{theorem: lattice / acl / fund-thm-field-acl} to make the following observation:
  Any automorphism $\Phi$ of $\G^*(K|k)$ arises from an automorphism $\tilde\Phi$ of $K^i$ which is unique up-to composition with some power of $\Frob_{K^i}$.
  Assume moreover that $\Phi$ restricts to an automorphism $\Phi_0$ of $\G^*(K_0|k_0)$.
  Then Theorem \ref{theorem: lattice / acl / fund-thm-field-acl} implies that $\Phi_0$ arises from an automorphism $\tilde\Phi_0$  of $K_0^i$ which is unique up-to composition with some power of $\Frob_{K_0^i}$.
  It therefore follows from the functoriality of the situation that this automorphism $\tilde\Phi$ restricts to an automorphism of $K_0^i$, which is precisely $\tilde\Phi_0$, up-to composition with a power of $\Frob_{K_0^i}$.

  Consider the following subgroup of $\Aut_{\{K_0^i\}}(K^i)$:
  \[ G:= \{\sigma \in \Aut_{\{K_0^i\}}(K^i) \ : \ \sigma|_{K_0^i} \in \langle \Frob_{K_0^i} \rangle \}.\]
  By Theorem \ref{theorem: lattice / acl / fund-thm-field-acl} and the observations made above, we see that the canonical map
  \[ G \rightarrow \Aut^*(\G^*(K|k)|K_0) \]
  is \emph{surjective}, and its kernel is the subgroup generated by $\Frob_{K^i}$.

  On the other hand, it is easy to see that one has a canonical isomorphism
  \[ \Gal(K|K_0) = \Gal(K^i|K_0^i) \xrightarrow{\cong} G/\langle \Frob_{K^i} \rangle. \]
  Finally, we conclude the proof by recalling that the canonical map $\rho_{k_0} : \Galk \rightarrow \Gal(K|K_0)$ is an isomorphism, since $K_0|k_0$ is a regular function field.
  As the isomorphisms 
  \[ \Gal(K|K_0) \cong G/\langle \Frob_{K^i} \rangle \cong \Aut^*(\G^*(K|k)|K_0) \]
  are all compatible with $\rho_{k_0}$, the assertion follows.
\end{proof}

\subsection{The Mod-$\ell$ Geometric Lattice}
\label{subsection: lattice / mod-ell}

As a first step towards the mod-$\ell$ context, we will prove that there is an isomorphic copy of $\Gbb^*(K|k)$ inside of the \emph{lattice of subgroups} of $\k_1(K)$.
We write $\Gfrak^*(K|k)$ for the collection of \emph{geometric subgroups} of $\k_1(K)$, ordered by inclusion in $\k_1(K)$, and graded by $\dimm$.
Namely, as a partially-ordered set, one has
\[ \Gfrak^*(K|k) = \{\Kfrak(S) \ : \ S \subset K \}\]
where the partial-ordering is given by inclusion of subgroups of $\k_1(K)$.
The grading of 
\[ \Gfrak^*(K|k) = \coprod_{s \geq 0} \Gfrak^s(K|k) \]
is defined by the condition that $\Gfrak^s(K|K)$ consists of all geometric subgroups $A$ of $\k_1(K)$ such that $\dimm(A) = s$.
The next proposition shows that we can identify $\Gfrak^*(K|k)$ with $\G^*(K|k)$, and that the \emph{Milnor-supremum} defined in \S\ref{subsection: milnorff / milnorsup} corresponds to the least-upper-bound in $\G^*(K|k)$, via this identification.

\begin{proposition}
\label{proposition: lattice / mod-ell / mod-ell-lattice}
  The following hold:
  \begin{enumerate}
    \item The set $\Gfrak^*(K|k)$, endowed with the ordering by inclusion in $\k_1(K)$ and the grading by $\dimm$, is a graded lattice.
    \item The map $\Kbb(S) \mapsto \Kfrak(S)$ induces a canonical isomorphism of graded lattices $\G^*(K|k) \rightarrow \Gfrak^*(K|k)$.
    \item For a collection $(\Kfrak_i)_i$ of elements of $\Gfrak^*(K|k)$, the least-upper-bound of $(\Kfrak_i)_i$ in the lattice $\Gfrak^*(K|k)$ is precisely the Milnor-supremum $\msup(\bigcup_i \Kfrak_i)$.
  \end{enumerate}
\end{proposition}
\begin{proof}
  Note that the map $\Kbb(S) \mapsto \Kfrak(S)$ yields a canonical surjective map $\Gbb^*(K|k) \rightarrow \Gfrak^*(K|k)$.
  By Fact \ref{fact: milnorff / milnordim / geometric-dim}, this surjective map restricts to a surjective map 
  \[ \G^s(K|k) \rightarrow \Gfrak^s(K|k) \]
  on the graded pieces, and it is clear from the definition that this map respects the partial ordering.
  Since $\G^*(K|k)$ is a graded lattice, in order to prove (1) and (2), it suffices to prove that the map above is an isomorphism on the level of partially ordered sets. 
  With this in mind, suppose that $S_1,S_2$ are two subsets of $K$, and assume that $\Kfrak(S_1) \subset \Kfrak(S_2)$.
  We must show that one has $\Kbb(S_1) \subset \Kbb(S_2)$.

  Assume the contrary, and, for $i = 1,2$, put $d_i := \trdeg(\Kbb(S_i)|k)$.
  Recall that, by Fact \ref{fact: milnorff / milnordim / geometric-dim}, one has $d_i = \dimm(\Kfrak(S_i))$.
  As $\Kbb(S_i)$ is relatively algebraically closed in $K$, it follows that
  \[ \trdeg(\Kbb(S_1 \cup S_2)|k) > \max(d_1,d_2) =: d. \]
  Hence, there exist $t_0,\ldots,t_{d} \in S_1 \cup S_2$ which are algebraically independent over $k$.
  By Lemma \ref{lemma: milnorff / nonvanishing / alg-indep-coords}, there exist $a_0,\ldots,a_d \in k$ such that 
  \begin{align}
  \label{align: mod-ell-lattice / non-vanishing}
    \{t_0-a_0,\ldots,t_d-a_d\}_K \neq 0.
  \end{align}
  However, since $\Kfrak(S_1) \subset \Kfrak(S_2) =: \Kfrak$, we have $\{t_i-a_i\}_K \in \Kfrak$ for all $i = 0,\ldots,d$.
  But then (\ref{align: mod-ell-lattice / non-vanishing}) provides a contradiction to the fact that $\dimm(\Kfrak) = d_2 \leq d$.
  
  This proves, in particular, that the map $\Gbb^*(K|k) \rightarrow \Gfrak^*(K|k)$ is an isomorphism of partially ordered sets, while the compatibility with the grading was already noted above. 
  Assertion (3) follows immediately from assertion (2) and Lemma \ref{lemma: milnorff / milnordim / milnorsup}.
\end{proof}

\subsection{Galois Action on the Mod-$\ell$ Lattice}
\label{subsection: lattice / galois-mod-ell}

By Proposition \ref{proposition: lattice / mod-ell / mod-ell-lattice}(1)(2), the partially ordered set $\Gfrak^*(K|k)$ is a graded lattice which is canonically isomorphic to $\Gbb^*(K|k)$.
Moreover, the definition of this isomorphism immediately shows that it is compatible with the action of $\Galk$.
In any case, since $\Gfrak^*(K|k)$ is a graded lattice, we can define $\Aut^*(\Gfrak^*(K|k)|K_0)$ similarly to the way we defined $\Aut^*(\G^*(K|k)|K_0)$.
Namely, $\Aut^*(\Gfrak^*(K|k)|K_0)$ consists of all automorphisms $\Phi$ of $\Gfrak^*(K|k)$ (as a graded-lattice) such that $\Phi\Kfrak(S) = \Kfrak(S)$ for all $S \subset K_0$.
As before, one has a canonical Galois representation 
\[ \rho_{k_0} : \Galk \rightarrow \Aut^*(\Gfrak^*(K|k)|K_0). \]

On the other hand, Proposition \ref{proposition: lattice / mod-ell / mod-ell-lattice}(2) immediately implies that we have a canonical isomorphism of automorphism groups:
\[ \Aut^*(\Gbb^*(K|k)|K_0) \xrightarrow{\cong} \Aut^*(\Gfrak^*(K|k)|K_0) \]
and it is easy to see that this isomorphism is compatible with $\rho_{k_0}$.
Thus, by combining Propositions \ref{proposition: lattice / galois-acl / galois-acl} and \ref{proposition: lattice / mod-ell / mod-ell-lattice}, we immediately deduce the following corollary.

\begin{corollary}
\label{corollary: lattice / galois-mod-ell / galois-lattice-iso}
  Assume that $\trdeg(K|k) \geq 5$.
  Then the canonical map
  \[ \rho_{k_0} : \Galk \rightarrow \Aut^*(\Gfrak^*(K|k)|K_0) \]
  is an isomorphism.
\end{corollary}

\section{Essentially Unramified Points and Fibers}
\label{section: essram}


A key difficulty which arises by working in the mod-$\ell$ context is that the presence of ramification can make certain valuations ``invisible.''
More precisely, suppose that $t$ is a non-constant element of $K$.
While every divisorial valuation of $\Kbb(t)|k$ is the restriction of some divisorial valuation of $K|k$, there may be some divisorial valuations $v$ of $\Kbb(t)|k$ such that $\{\U_v\}_K$ is not of the form $\Ufrak_w \cap \Kfrak(t)$ for any divisorial valuation $w$ of $K|k$.
Dualizing using the Kummer pairing, there may be some minimized inertia subgroups of $\Gc_{\Kbb(t)}^a$ arising from divisorial valuations of $\Kbb(t)|k$ which are not the image of any minimized inertia subgroup of $\Gc_K^a$ that arise from a divisorial valuation of $K|k$.
This ``difficulty'' is clearly intimately tied to ramification (specifically, ramification indices which are divisible by $\ell$).
In this section, we introduce some general conditions which will suffice to prevent such problems.

\subsection{The Flag Associated to Regular Parameters}
\label{subsection: essram / regular-params-flags}

We will use a basic construction in algebraic geometry which associates an $r$-divisorial valuation and/or a flag of divisorial valuations of length $r$ to a system of regular parameters at a point of codimension $r$ in a regular $k$-variety.
More precisely, let $X$ be a regular $k$-variety, and let $x$ be a (possibly non-closed) regular point of $X$.
Let $\Oc_{X,x}$ be the regular local ring at $x$.
For a subset $S \subset \Oc_{X,x}$, we use the usual notation
\[ V(S) = \Spec \Oc_{X,x}/(S) \]
to denote the closed subscheme of $\Spec \Oc_{X,x}$ associated to the ideal $(S)$ of $\Oc_{X,x}$.

Let $(f_1,\ldots,f_r)$ be a system of regular parameters for the (maximal ideal of the) regular local ring $\Oc_{X,x}$, and put $L := k(X)$.
We will abuse the terminology and also say that $(f_1,\ldots,f_r)$ is a {\bf system of regular parameters at $x$ in $X$.}

In any case, to the system $(f_1,\ldots,f_r)$ of regular parameters of $\Oc_{X,x}$, we associate a flag $(v_1,\ldots,v_r)$ of divisorial valuations of $L|k$, by letting $v_i/v_{i-1}$ be the divisorial valuation of
\[ Lv_{i-1} = k(V(f_1,\ldots,f_{i-1})) \]
associated to the prime Weil-divisor $V(f_1,\ldots,f_i)$ on $V(f_1,\ldots,f_{i-1})$.
As before, $v_0$ stands for the trivial valuation of $L$ by convention.
The following is a summary of the basic facts concerning the relationship between $(f_1,\ldots,f_r)$ and $(v_1,\ldots,v_r)$.

\begin{fact}
\label{fact: essram / regular-params-flags / regular-params}
  In the context above, the following hold:
  \begin{enumerate}
    \item For all $i = 1,\ldots,r$, one has $\U_{v_{i-1}}/\U_{v_{i}} \cong \Z$.
    Moreover, one has $f_i \in \U_{v_{i-1}}$ and its image in $\U_{v_{i-1}}/\U_{v_i}$ is a generator.
    \item For all $i = 1,\ldots,r$, one has $\{\U_{v_{i-1}}\}_{L}/\{\U_{v_{i}}\}_{L} \cong \Z/\ell$.
    Moreover, one has $\{f_i\}_{L} \in \{\U_{v_{i-1}}\}_{L}$ and its image in $\{\U_{v_{i-1}}\}_{L}/\{\U_{v_{i}}\}_{L}$ is a generator.
  \end{enumerate}
\end{fact}

\subsection{$\ell$-Unramified Prolongations}
\label{subsection: essram / ell-unram}
Suppose now that $k(X) = L$ is a subextension of $K|k$, and $(v_1,\ldots,v_r)$ is a flag of divisorial valuations of $L|k$.
Let $(w_1,\ldots,w_r)$ be a flag of divisorial valuations of $K|k$ which \emph{prolongs} $(v_1,\ldots,v_r)$; i.e. $w_i$ is a prolongation of $v_i$ for $i = 1,\ldots,r$.
Note that this condition ensures that the canonical map $\k_1(L) \rightarrow \k_1(K)$ restricts to a morphism $\{\U_{v_i}\}_L \rightarrow \Ufrak_{w_i}$ for all $i = 1,\ldots,r$.
We will say that $(w_1,\ldots,w_r)$ is an {\bf $\ell$-unramified prolongation} of $(v_1,\ldots,v_r)$, provided that the canonical map $\k_1(L) \rightarrow \k_1(K)$ induces isomorphisms
\[ \frac{\{\U_{v_{i-1}}\}_L}{\{\U_{v_i}\}_L} \xrightarrow{\cong} \frac{\Ufrak_{w_{i-1}}}{\Ufrak_{w_i}} \]
for all $i = 1,\ldots,r$ (recall that $v_0$ resp. $w_0$ denote the trivial valuations).

The notion of an $\ell$-unramified prolongation and Fact \ref{fact: essram / regular-params-flags / regular-params} lead to the following particularly useful observation.
If the flag $(v_1,\ldots,v_r)$ arises from a system of regular parameters $(f_1,\ldots,f_r)$ at a regular point $x \in X$ as described in \S\ref{subsection: essram / regular-params-flags}, and $(w_1,\ldots,w_r)$ is an $\ell$-unramified prolongation of $(v_1,\ldots,v_r)$ to $K$, then $\{f_i\}_K$ is a generator of $\Ufrak_{w_{i-1}}/\Ufrak_{w_i} \cong \Z/\ell$ for all $i = 1,\ldots,r$ by Fact \ref{fact: essram / regular-params-flags / regular-params}(2).

\subsection{The Essential Branch Locus}
\label{subsection: essram / essential-branch-locus}

Let $X$ be a regular $k$-variety, and suppose that $K$ is a \emph{finite extension} of $k(X)$.
Let $K'$ denote the maximal separable subextension of $K|k(X)$, and let $Y \rightarrow X$ denote the normalization of $X$ in $K'$.
Thus, $Y \rightarrow X$ is a finite separable (possibly branched) cover of $k$-varieties.
The branch locus of this finite separable cover $Y \rightarrow X$ will be called the {\bf essential branch locus of $X$ in $K$}.
Recall that the essential branch locus of $X$ in $K$ is actually a divisor of $X$.
However, we will only be interested in the \emph{support} of this divisor.
Namely, we will consider the essential branch locus of $X$ in $K$ only as a closed subvariety of $X$ which has codimension one.

Along the same lines, if $x$ is a (scheme-theoretic) point of $X$, we say that $x$ is {\bf essentially unramified} in $K$ if $x$ is not contained in the essential branch locus of $X$ in $K$.
Similarly, if $Z$ is an integral closed subvariety of $X$, we say that $Z$ is {\bf essentially unramified} in $K$ if the generic point of $Z$ is essentially unramified in $K$.
Since the essential branch locus of $X$ in $K$ is \emph{closed} in $X$, we note that $Z$ is essentially unramified in $K$ if and only if $Z$ is not contained in the essential branch locus of $X$ in $K$.

The concept of essentially unramified points and the essential branch locus allows us to produce ``many'' $\ell$-unramified prolongations of flags of divisorial valuations, as follows.
In the context above, suppose that $x$ is an essentially unramified regular point of $X$, and let $y \in Y$ be a point in the fiber of $x$.
Moreover, let $(f_1,\ldots,f_r)$ be a system of regular parameters at $x$, and let $\vbf := (v_1,\ldots,v_r)$ be the flag of divisorial valuations of $k(X)|k$ associated to $(f_1,\ldots,f_r)$.
Since $y$ is unramified over $x$, it follows that $(f_1,\ldots,f_r)$ is a system of regular parameters at $y$ as well, and we let $\wbf' := (w_1',\ldots,w_r')$ be the flag of divisorial valuations of $k(Y) = K'$ associated to $(f_1,\ldots,f_r)$.
Note that the canonical map $\k_1(k(X)) \rightarrow \k_1(k(Y))$ induces isomorphisms
\[ \frac{\{\U_{v_{i-1}}\}_{k(X)}}{\{\U_{v_i}\}_{k(X)}} \xrightarrow{\cong} \frac{\{\U_{w'_{i-1}}\}_{k(Y)}}{\{\U_{w_i'}\}_{k(Y)}} \]
for all $i = 1,\ldots,r$ (recall that $v_0$ resp. $w'_0$ denote the trivial valuations).

Finally, let $\wbf = (w_1,\ldots,w_r)$ be \emph{any} prolongation of $\wbf'$ to $K$.
Then $\wbf$ is a flag of divisorial valuations, and it follows from Lemma \ref{lemma: milnor / insep / insep-val} that $\wbf$ is an \emph{$\ell$-unramified prolongation} of $\vbf$ as defined in \S\ref{subsection: essram / ell-unram}.
We summarize this discussion for later use in the following fact.

\begin{fact}
\label{fact: essram / essential-branch-locus / essential-branch-fact}
  Suppose that $X$ is a regular $k$-variety and that $K$ is a finite extension of $k(X)$.
  Let $K'$ denote the maximal separable subextension of $K|k(X)$, and let $Y$ denote the normalization of $X$ in $K'$.
  Let $x$ be a regular point of $X$ which is essentially unramified in $K$.
  Let $(f_1,\ldots,f_r)$ be a system of regular parameters at $x$ with associated flag $\vbf$ of divisorial valuations, and let $y$ be any point of $Y$ in the fiber above $x$.
  Then the following hold:
  \begin{enumerate}
    \item The system $(f_1,\ldots,f_r)$ is a system of regular parameters at $y$.
    \item Any prolongation $\wbf$ to $K$ of the flag of divisorial valuations of $k(Y) = K'$ associated to $(f_1,\ldots,f_r)$, considered as a system of regular parameters at $y$, is an $\ell$-unramified prolongation of $\vbf$.
  \end{enumerate}
\end{fact}

\subsection{Essentially Unramified Fibers}
\label{subsection: essram / essential-unram}

We will primarily use Fact \ref{fact: essram / essential-branch-locus / essential-branch-fact} in the case where $x \in X$ is the generic point of a fiber of some smooth morphism.
More precisely, suppose that $X \rightarrow S$ is a dominant smooth morphism of regular $k$-varieties with geometrically integral fibers, and let $s \in S$ be a closed point in the image of $X \rightarrow S$.
Let $\eta_s \in X$ be the generic point of the fiber of $X \rightarrow S$ over $s$.

With the setup above, if $(f_1,\ldots,f_r)$ is a system of regular parameters at $s \in S$, then $(f_1,\ldots,f_r)$ is also a system of regular parameters at $\eta_s \in X$.
Thus, if $K$ is a finite extension of $k(X)$ and $\eta_s$ is essentially unramified in $K$, then we may apply Fact \ref{fact: essram / essential-branch-locus / essential-branch-fact} to $\eta_s \in X$ endowed with a system of regular parameters $(f_1,\ldots,f_r)$, which arises from $s \in S$.

\subsection{Mod-$\ell$ Divisors}
\label{subsection: essram / mod-ell-divs}

In this subsection, we will use the concept of \emph{essentially unramified fibers}, as discussed in \S\ref{subsection: essram / essential-unram}, to compare the divisorial valuations on a field of the form $\Kbb(t)$ with the divisorial valuations which can be detected in the mod-$\ell$ setting from $K$.

For $t \in K \smallsetminus k$, we denote by $\Dcal_t$ the set of divisorial valuations of $\Kbb(t)|k$.
Since $\Kbb(t)$ is a one-dimensional function field over $k$, the elements of $\Dcal_t$ are in canonical bijection with the closed points of the unique complete normal model of $\Kbb(t)|k$.
As for the mod-$\ell$ analogue of $\Dcal_t$, we define $\Dfrak_t$ to be the set
\[ \Dfrak_t := \{ \Kfrak(t) \cap \Ufrak_v \ : \ \Kfrak(t) \not\subset \Ufrak_v \} \]
where $v$ varies over the divisorial valuations of $K|k$.

Note that one has $\Kbb(t)/\U_v \cong \Z$ for all $v \in \Dcal_t$, and that $\Kfrak(t)/\Vfrak \cong \Z/\ell$ for all $\Vfrak \in \Dfrak_t$.
Our primary goal in this section is to compare $\Dcal_t$ and $\Dfrak_t$.
First, we show that $\Dfrak_t$ can be embedded canonically in $\Dcal_t$.

\begin{lemma}
\label{lemma: essram / mod-ell-divs / mod-ell-div-injection}
  Let $t \in K \smallsetminus k$ be given.
  For every $\Vfrak \in \Dfrak_t$, there exists a unique $v \in \Dcal_t$ such that $\Vfrak = \{\U_v\}_K$.
  In particular, one has a canonical injective map $\Dfrak_t \hookrightarrow \Dcal_t$ which satisfies $\Vfrak \mapsto v$ if and only if $\Vfrak = \{\U_v\}_K$.
\end{lemma}
\begin{proof}
  Let $w$ be a divisorial valuation of $K|k$ such that $\Vfrak = \Ufrak_w \cap \Kfrak(t)$.
  Since $\Kfrak(t) \not\subset \Ufrak_w$, we deduce that $w$ is non-trivial on $\Kbb(t)$, hence $w|_{\Kbb(t)}$ is a divisorial valuation on $\Kbb(t)$.
  Denote this divisorial valuation of $\Kbb(t)$ by $v$.
  Then clearly one has $\{\U_v\}_K \subset \Ufrak_w \cap \Kfrak(t) = \Vfrak$.
  On the other hand, both $\Kfrak(t)/\Vfrak$ and $\Kfrak(t)/\{\U_v\}_K$ are isomorphic to $\Z/\ell$, hence $\Vfrak = \{\U_v\}_K$.

  Concerning the uniqueness of $v$, suppose that $v'$ is another divisorial valuation of $\Kbb(t)$ such that $\Vfrak = \{\U_{v'}\}_K$.
  Then one has $\{\U_v\}_K = \{\U_{v'}\}_K$, and since $\k_1(\Kbb(t)) \rightarrow \k_1(K)$ is injective, we deduce that $\{\U_v\}_{\Kbb(t)} = \{\U_{v'}\}_{\Kbb(t)}$.
  In particular, $v$ and $v'$ must be dependent as valuations of $\Kbb(t)$, for otherwise $\U_v \cdot \U_{v'} = \Kbb(t)^\times$ by the approximation theorem for independent valuations.
  Since both $v$ and $v'$ have value groups isomorphic to $\Z$, it follows that $v = v'$.
  This proves the uniqueness of $v$, as required.
\end{proof}

A primary difficuly which arises in the mod-$\ell$ case is that the canonical map $\Dfrak_t \hookrightarrow \Dcal_t$ described in Lemma \ref{lemma: essram / mod-ell-divs / mod-ell-div-injection} need not be \emph{surjective} in general.
Nevertheless, we can use the notion of \emph{essentially-unramified fibers} to give a sufficient condition for an element of $\Dcal_t$ to arise from $\Dfrak_t$.
Although we will not need it later, we note that an argument similar to the proof of Lemma \ref{lemma: essram / mod-ell-divs / fibers} below shows that all but finitely many of the elements of $\Dcal_t$ arise from some element of $\Dfrak_t$ via this injection.

\begin{lemma}
\label{lemma: essram / mod-ell-divs / fibers}
  Let $t \in K \smallsetminus k$ be given and let $S$ be the unique proper normal model of $\Kbb(t)|k$.
  Let $v \in \Dcal_t$ be given, and let $s \in S$ be the (unique) center of $v$.
  Assume that there exists a regular $k$-variety $X$ and a smooth dominant morphism $X \rightarrow S$ with geometrically integral fibers, such that the following hold:
  \begin{enumerate}
    \item $K$ is a finite extension of $k(X)$.
    \item $s$ is in the image of $X \rightarrow S$.
    \item The fiber $X_s$ of $X \rightarrow S$ over $s$ is essentially unramified in $K$.
  \end{enumerate}
  Then there exists some $\Vfrak \in \Dfrak_t$ such that $\Vfrak = \{\U_v\}_K$, i.e. $v$ is in the image of the canonical injective map $\Dfrak_t \hookrightarrow \Dcal_t$ from Lemma \ref{lemma: essram / mod-ell-divs / mod-ell-div-injection}.
\end{lemma}
\begin{proof}
  Let $\pi \in \Kbb(t)$ be a uniformizer for $v$, and let $s \in S$ be the (unique) center of $v$ on $S$.
  Furthermore, let $\eta \in X$ be the generic point of the fiber of $X \rightarrow S$ over $s$.
  Following the discussion of \S\ref{subsection: essram / essential-unram} with the morphism $X \rightarrow S$, it follows that $\eta$ is a regular point of codimension one, and that $\pi$ is a local parameter at $\eta$.
  Let $w_0$ be the divisorial valuation of $k(X)$ associated to $\pi$ at $\eta$.
  By Fact \ref{fact: essram / essential-branch-locus / essential-branch-fact}, there exists an $\ell$-unramified prolongation $w$ of $w_0$ to $K|k$.
  The lemma follows by taking $\Vfrak = \Ufrak_w \cap \Kfrak(t)$ and noting that the image of $\{\pi\}_K$ is a generator of $\k_1(K)/\Ufrak_w \cong \Z/\ell$.
\end{proof}

\section{Strongly General Elements and Birational Bertini}
\label{section: strgen}


In this section, we recall the notion of a \emph{general element} of a regular function field.
We also introduce the notion of a \emph{strongly-general element} which is related to the notion of a general element, but has further assumptions which are motivated by the discussion of \S\ref{section: essram}.
The primary goal of this section is to prove so-called \emph{Birational Bertini} results for both general and strongly-general elements, which will show that there are ``many'' such elements in higher-dimensional function fields.

\subsection{General Elements}
\label{subsection: strgen / general}

Let $L|F$ be a regular field extension and let $x \in L \smallsetminus F$ be given.
We say that $x$ is {\bf separable} in $L|F$ if $x \notin F \cdot L^p$ where $p = \Char F$.
If $\Char F = 0$, then every element of $L$ is separable in $L|F$ by convention.
We say that $x$ is {\bf general} in $L|F$ provided that $L$ is a regular extension of $F(x)$.
In particular, if $x$ is general in $L|F$ then $F(x)$ is relatively algebraically closed in $L$.

The following lemma is our so-called \emph{Birational Bertini} result for general elements.
The first assertion of this lemma can be found in \cite{Lang:1958}*{Ch. VIII, pg. 213}.
The second assertion of this lemma has essentially the same proof as in loc.~cit., but since it hasn't explicitly appeared in the literature, we provide a detailed proof below.
\begin{lemma}
\label{lemma: strgen / general / gen-birat-bertini}
  Let $F$ be an infinite field, and let $L$ be a regular function field over $F$.
  Let $x,y \in L$ be algebraically independent over $F$ with $y$ separable in $L|F$.
  Then the following hold:
  \begin{enumerate}
    \item For all but finitely many $a \in F$, the element $x+ay$ is general in $L|F$.
    \item There exists a non-empty open subset $U$ of $\Abb^2_F$ such that for all $(a,b) \in U(F)$, the element $(x-a)/(y-b)$ is general in $L|F$.
  \end{enumerate}
\end{lemma}
\begin{proof}
  As mentioned above, we only need to prove assertion (2) since assertion (1) can be found in \cite{Lang:1958}*{Ch. VIII, pg. 213}.
  To simplify the notation, for $P = (a,b) \in F^2$, we write
  \[ t_{a,b} = t_P = \frac{x-a}{y-b}. \]
  Since $y$ is separable in $L|F$, there exists an $F$-derivation $D$ of $L$ such that $D(y) \neq 0$.
  Now we may calculate:
  \[ D(t_{a,b}) = \frac{D(x) \cdot (y-b) - D(y) \cdot (x-a)}{(y-b)^2}. \]
  Hence $D(t_{a,b}) = 0$ if and only if $t_{a,b} = D(x)/D(y)$, so this can only happen for at most one pair $P_0 \in F^2$.
  Thus, for any pair $(a,b)$ different from this (possible) exceptional one $P_0$, the element $t_{a,b}$ is separable in $L|F$, which implies that $L|F(t_{a,b})$ has a separating transcendence base.

  For each $P \in F^2$, write $E_P := F(t_P)$ and $E'_P$ for the relative algebraic closure of $E_P$ in $L$.
  If $P,Q$ are two different points of $F^2$, note that one has $E_P \cdot E_Q = F(x,y)$.
  Note that $E_P'$ and $E_Q'$ are regular of transcendence degree $1$ over $F$, since they are contained in $L$ which is regular over $F$.
  This implies that $E_P'$ and $E_Q'$ are linearly disjoint over $F$, since $L|F$ is regular.

  Let $M$ denote the relative algebraic closure of $F(x,y)$ in $L$, and note that there are only finitely many intermediate subextensions of $M|F(x,y)$.
  In particular, there are finitely many subextensions of $M|F(x,y)$ of the form $E'_P(x,y)$.
  Let $P_1,\ldots,P_n$ be finitely many points in $F^2$ such that $E'_{P_i}(x,y)$ exhaust all such subextensions.
  Now suppose that $Q$ is any point of $F^2$ which is different from $P_0,\ldots,P_n$.
  Since $Q \neq P_0$, we see that $E'_Q$ is the separable algebraic closure of $E_Q$ in $M$.
  Moreover, by the discussion above, $E'_Q(x,y)$ must be linearly disjoint from $E'_{P_i}(x,y)$ over $F(x,y)$, for all $i = 1,\ldots,n$.
  This forces $E'_Q(x,y)$ to be precisely $F(x,y)$, and therefore $E'_Q = E_Q$.
  In other words, $E_Q = F(t_Q)$ is algebraically closed in $L$, and since $Q \neq P_0$, we see that $L$ is regular over $E_Q = F(t_Q)$, as required.
\end{proof}

\subsection{Strongly-General Elements}
\label{subsection: strgen / strgen}

Let $t \in K$ be a non-constant element.
We say that $t$ is a {\bf strongly-general element} in $K|k$ if the following two conditions hold:
\begin{enumerate}
  \item The element $t$ is general in $K|k$.
  \item The canonical injective map $\Dfrak_t \rightarrow \Dcal_t$ from Lemma \ref{lemma: essram / mod-ell-divs / mod-ell-div-injection} is surjective (hence bijective).
\end{enumerate}
In the spirit of Lemma \ref{lemma: strgen / general / gen-birat-bertini}, we now prove the following \emph{Birational Bertini} result for strongly-general elements of $K|k$.

\begin{proposition}
\label{proposition: strgen / strgen / strgen-birat-bertini}
  Let $x,y \in K$ be algebraically independent over $k$ such that $y$ is separable in $K|k$.
  Then there exists a non-empty open subset $U$ of $\Abb^2_k$ such that, for all $(a,b) \in U(k)$, the element $(x-a)/(y-b)$ is strongly-general in $K|k$.
\end{proposition}
\begin{proof}
  By Lemma \ref{lemma: strgen / general / gen-birat-bertini}(2), there exists a non-empty open subset $U_1$ of $\Abb^2_k$ such that for all $(a,b) \in U_1(k)$, the element $t_{a,b} = (x-a)/(y-b)$ is general in $K|k$.
  We must therefore prove that the second condition for a strongly-general element is also an open condition on $(a,b) \in \Abb^2_k(k)$ as above.

  Put $x = t_1$, $y = t_2$ and extend $t_1,t_2$ to a transcendence base $\tbf = (t_1,\ldots,t_d)$ for $K|k$.
  Furthermore, consider the canonical coordinate projection $\Abb^d_\tbf \rightarrow \Abb^2_{t_1,t_2}$.
  Also, identify $\Abb^2_k$ as above with $\Abb^2_{t_1,t_2}$, so that we may consider $U_1$ as an open subset of $\Abb^2_{t_1,t_2}$.
  Finally, note that $K$ is a finite extension of $k(\tbf) = k(\Abb^d_\tbf)$.
  Thus, there exists a non-empty open subset $U$ of $U_1$ such that for all $(a,b) \in U(k)$, the fiber of $\Abb^d_\tbf \rightarrow \Abb^2_{t_1,t_2}$ over the point $(a,b)$ is essentially unramified in $K$.
  We will show that this open set $U$ satisfies the required assertion.

  Let $(a,b) \in U(k)$ be given, and consider the (rational) projection $\Abb^2_{t_1,t_2} \rightarrow \Pbb^1_t$ induced by sending $t$ to $(x-a)/(y-b)$.
  The fibers of this morphism are all lines in $\Abb^2_{t_1,t_2}$ given by equations of the form
  \[ c \cdot (x-a) = d \cdot (y-b) \]
  for $(c:d) \in \Pbb^1(k)$ (written in homogeneous coordinates).
  All such lines pass through the point $(a,b)$ in $\Abb^2_{t_1,t_2}$ hence the fibers of the composition
  \[ \Abb^d_\tbf \rightarrow \Abb^2_{t_1,t_2} \rightarrow \Pbb^1_t \]
  are all non-empty, geometrically integral, and essentially unramified in $K$.
  Let $Z$ be the preimage of $(a,b)$ under the coordinate projection $\Abb^d_\tbf \twoheadrightarrow \Abb^2_{t_1,t_2}$, and put $X:= \Abb^d_\tbf \smallsetminus Z$.
  Then the surjective rational morphism $\Abb^d_\tbf \twoheadrightarrow \Pbb^1_t$ defined above is regular on $X$, and the induced (regular) map $X \rightarrow \Pbb^1_t$ is a smooth surjective morphism with geometrically integral fibers.
  
  With this set-up, the fact that $\Dfrak_t \rightarrow \Dcal_t$ is surjective follows from Lemma \ref{lemma: essram / mod-ell-divs / fibers}, by taking $X$ as above, $S = \Pbb^1_t$, and $X \rightarrow S$ the map defined above.
\end{proof}

The main benefit of Proposition \ref{proposition: strgen / strgen / strgen-birat-bertini} is that it can be used to show that strongly-general elements multiplicatively generate higher-dimensional function fields.
The following lemma is a precise formulation of this fact which we will use later.

\begin{lemma}
\label{lemma: strgen / strgen / strgen-generate}
  Let $L$ be a relatively-algebraically closed subfield of $K$, such that $k_0 \subset L$ and $\trdeg(L|k_0) \geq 2$.
  Then the multiplicative group $L^\times$ is generated by elements $t \in L^\times$ which are strongly-general in $K|k$.
\end{lemma}
\begin{proof}
  It suffices to prove that every \emph{transcendental} element $x$ of $L$ is a product of finitely many elements of $L$ which are strongly-general in $K|k$.
  Since $L$ is relatively algebraically closed in $K$, it follows that every non-constant element of $L$ is a power of some element of $L$ which is separable in $K|k$.
  Thus, we may assume without loss of generality that $x$ is separable.
  Moreover, since $\trdeg(L|k \cap L) = \trdeg(L|k_0) \geq 2$, there exists another element $y \in L$ which is algebraically independent from $x$ over $k$, and such that $y$ is separable in $K|k$.

  Note that if $(t-ab)/(u-b)$ is strongly-general in $K|k$ with $a,b \in k$, then so is
  \[ \frac{t-ab}{u-b} - a = \frac{t-au}{u-b}. \]
  Thus, using Proposition \ref{proposition: strgen / strgen / strgen-birat-bertini}, the fact that $k_0$ is infinite, and the assumption that $k_0 \subset L$, we can find $a,b \in k_0$ such that both
  \[ \frac{x-a}{y-b}, \ \frac{xy-bx}{x-a} \in L^\times \]
  are strongly-general in $K|k$.
  But then we see that
  \[ x = \frac{x-a}{y-b} \cdot \frac{xy-bx}{x-a} \]
  is a product of two elements of $L$ which are strongly-general in $K|k$.
  This concludes the proof of the lemma.
\end{proof}

\subsection{Rational-Like Collections}
\label{subsection: strgen / rational-like}

Let $t \in K \smallsetminus k$ be given.
A particularly useful consequence of Lemma \ref{lemma: essram / mod-ell-divs / mod-ell-div-injection} is that every element $x \in \Kfrak(t)$ is contained in all but finitely many of the $\Vfrak \in \Dfrak_t$.
In particular, the projection maps $\Kfrak(t) \twoheadrightarrow \Kfrak(t)/\Vfrak$ for $\Vfrak \in \Dfrak_t$ together induce a canonical morphism
\[ \Kfrak(t) \rightarrow \bigoplus_{\Vfrak \in \Dfrak_t} \Kfrak(t)/\Vfrak. \]
Moreover, recall that $\Kfrak(t)/\Vfrak \cong \Z/\ell$ for all $\Vfrak \in \Dfrak_t$.
By choosing a collection $\Phi = (\Phi_\Vfrak)_{\Vfrak \in \Dfrak_t}$ of isomorphisms
\[ \Phi_\Vfrak : \Kfrak(t)/\Vfrak \xrightarrow{\cong} \Z/\ell \]
we thereby obtain a map
\[ \divv_\Phi : \Kfrak(t) \rightarrow \bigoplus_{\Vfrak \in \Dfrak_t} \Z/\ell \cdot [\Vfrak] \]
defined by $\divv_\Phi(x) = \sum_{\Vfrak \in \Dfrak_t} \Phi_\Vfrak(x \cdot \Vfrak) \cdot [\Vfrak]$.
As the notation $\divv_\Phi$ suggests, this morphism should be considered as a ``mod-$\ell$'' analogue of the usual divisor map
\[ \divv : \Kbb(t) \rightarrow \bigoplus_{v \in \Dcal_t} \Z \cdot [v] \]
albeit with respect to the collection of isomorphisms $\Phi$.

Suppose now that $t$ is \emph{strongly general} in $K|k$, so that the canonical injective map $\Dfrak_t \rightarrow \Dcal_t$ given by Lemma \ref{lemma: essram / mod-ell-divs / mod-ell-div-injection} is actually a \emph{bijection}.
In this context, we say that a collection of isomorphisms $\Phi = (\Phi_\Vfrak : \Kfrak(t)/\Vfrak \xrightarrow{\cong} \Z/\ell)_{\Vfrak \in \Dfrak_t}$ as above is a {\bf rational-like collection} provided that the induced map
\[ \divv_\Phi : \Kfrak(t) \rightarrow \bigoplus_{\Vfrak \in \Dfrak_t} \Kfrak(t)/\Vfrak \xrightarrow{(\Phi_\Vfrak)_{\Vfrak}} \bigoplus_{\Vfrak \in \Dfrak_t} \Z/\ell \cdot [\Vfrak] \]
fits into a short exact sequence
\[ 0 \rightarrow \Kfrak(t) \xrightarrow{\divv_\Phi} \bigoplus_{\Vfrak \in \Dfrak_t} \Z/\ell \cdot [\Vfrak] \xrightarrow{\text{sum}} \Z/\ell \rightarrow 0. \]

For every strongly-general element $t \in K \smallsetminus k$, there is a {\bf canonical rational-like collection} $\Psi = (\Psi_\Vfrak)_{\Vfrak \in \Dfrak_t}$ associated to the field $\Kbb(t) = k(t)$, which is defined as follows.
For each $\Vfrak \in \Dfrak_t$ with associated element $v \in \Dcal_t$ (i.e. $\Vfrak = \{\U_v\}_K$), the isomorphism $\Psi_\Vfrak$ is the unique one making the following diagram commute:
\[
\xymatrix{
  \Kbb(t)^\times \ar[r]^{v}\ar[d]_{\{\bullet\}_K} & \Z \ar@{->>}[r] & \Z/\ell \ar@{=}[d] \\
  \Kfrak(t) \ar@{->>}[r] & \Kfrak(t)/\Vfrak \ar[r]_{\Psi_\Vfrak} & \Z/\ell.
}
\]
The following Lemma shows that any rational collection is obtained from the canonical one by multiplying by some element of $(\Z/\ell)^\times$.

\begin{lemma}
\label{lemma: strgen / rational-like / rational-like}
  Let $t \in K \smallsetminus k$ be strongly general in $K|k$, and let $\Psi = (\Psi_\Vfrak)_{\Vfrak \in \Dfrak_t}$ be the canonical rational-like collection associated to $\Kbb(t) = k(t)$.
  Also, let $(\Phi_\Vfrak)_{\Vfrak \in \Dfrak_t}$ be another rational-like collection.
  Then there exists a unique $\epsilon \in (\Z/\ell)^\times$, such that $\Phi_\Vfrak = \epsilon \cdot \Psi_\Vfrak$ for all $\Vfrak \in \Dfrak_t$.
\end{lemma}
\begin{proof}
  For each $\Vfrak \in \Dfrak_t$, there exists some $\epsilon_\Vfrak \in (\Z/\ell)^\times$ such that
  \[ \Phi_\Vfrak = \epsilon_\Vfrak \cdot \Psi_\Vfrak. \]
  We must show that $\epsilon_\Vfrak$ is independent of $\Vfrak$.
  So let $\Vfrak_i$, $i = 1,2$ be two elements of $\Dfrak_t$, and put $\epsilon_i := \epsilon_{\Vfrak_i}$.

  Since $\Psi$ is a rational-like collection, there exists some $a \in \Kfrak(t)$ such that
  \[ \divv_\Psi(a) = [\Vfrak_1]-[\Vfrak_2]. \]
  This implies that
  \[ \divv_\Phi(a) = \epsilon_1 \cdot [\Vfrak_1] - \epsilon_2 \cdot [\Vfrak_2]. \]
  But the fact that $\Phi$ is a rational-like collection implies that $\epsilon_1 = \epsilon_2$, as required.
\end{proof}

We now use the notions of strongly-general elements and rational-like collections to prove a proposition which will be useful in several steps of the proof of Theorem \ref{maintheorem: intro / milnor-variant / milnor-main}.

\begin{proposition}
\label{proposition: strgen / rational-like / ident-criterion}
  Assume that $\trdeg(K|k) \geq 2$, and let $\sigma$ be an automorphism of $\k_1(K)$.
  Then the following are equivalent:
  \begin{enumerate}
    \item There exists an $\epsilon \in (\Z/\ell)^\times$ such that $\sigma = \epsilon \cdot \one_{\k_1(K)}$.
    \item For all $\Kfrak \in \Gfrak^1(K|k)$, one has $\sigma \Kfrak = \Kfrak$.
  \end{enumerate}
\end{proposition}
\begin{proof}
  The implication $(1) \Rightarrow (2)$ is trivial.
  The proof of the non-trivial direction $(2) \Rightarrow (1)$ has three main steps:
  \begin{enumerate}
    \item First, we show that for all $t \in K \smallsetminus k$ and for all $\Vfrak \in \Dfrak_t$, one has $\sigma \Vfrak = \Vfrak$.
    \item Second, we show that for all $t$ which is strongly-general in $K|k$, the restriction of $\sigma$ to $\Kfrak := \Kfrak(t)$ is of the form $\epsilon_\Kfrak \cdot \one_\Kfrak$, for some $\epsilon_\Kfrak \in (\Z/\ell)^\times$, which \emph{a priori} might depend on $\Kfrak$.
    \item Finally, we show that $\epsilon_\Kfrak$ from step (2) doesn't actually depend on $\Kfrak$, and then conclude the proof of the proposition by using Lemma \ref{lemma: strgen / strgen / strgen-generate}.
  \end{enumerate}

  \vskip 5pt
  \noindent\emph{Step (1):} Let $v$ be a divisorial valuation of $K|k$.
  Since $Kv$ is a function field of transcendence degree $\geq 1$ over $k$, it is easy to see that $\U_v$ is multiplicatively generated by elements $x \in \U_v$ whose image in $Kv$ is transcendental over $k$.
  For all such $x \in \U_v$, one has $\Kbb(x)^\times \subset \U_v$.
  Thus, $\Ufrak_v$ is generated by subgroups of the form $\Kfrak(x)$ such that $\Kfrak(x) \subset \Ufrak_v$.
  Since $\sigma \Kfrak(x) = \Kfrak(x)$ for all $x \in K \smallsetminus k$, it follows that $\sigma \Ufrak_v = \Ufrak_v$.
  For any $t \in K \smallsetminus k$, it follows that $\sigma \Vfrak = \Vfrak$ for all $\Vfrak \in \Dfrak_t$ by the definition of the elements of $\Dfrak_t$.

  \vskip 5pt
  \noindent\emph{Step (2):}
  Let $t$ be strongly-general in $K|k$ and put $\Kfrak = \Kfrak(t)$.
  Consider the canonical rational-like collection $\Psi = (\Psi_\Vfrak)_{\Vfrak \in \Dfrak_t}$ associated to $\Kbb(t) = k(t)$.
  Since $\sigma \Vfrak = \Vfrak$ for all $\Vfrak \in \Dfrak_t$, we obtain an induced rational-like collection $\Phi = (\Phi_\Vfrak)_{\Vfrak \in \Dfrak_t}$, where $\Phi_\Vfrak := \Psi_\Vfrak \circ \sigma$.
  By Lemma \ref{lemma: strgen / rational-like / rational-like}, there exists some $\epsilon_\Kfrak \in (\Z/\ell)^\times$ depending only on $\Kfrak$ and $\sigma$, such that $\Phi_\Vfrak = \epsilon_\Kfrak \cdot \Psi_\Vfrak$ for all $\Vfrak \in \Dfrak_t$.
  In other words, for all $x \in \Kfrak$, one has
  \[ \divv_\Psi(\sigma x) = \divv_\Phi(x) = \epsilon_\Kfrak \cdot \divv_\Psi(x) = \divv_\Psi(\epsilon_\Kfrak \cdot x). \]
  The injectivity of $\divv_\Psi$ in the definition of $\Psi$ being a rational-like collection implies that $\sigma|_\Kfrak = \epsilon_\Kfrak \cdot \one_\Kfrak$.

  \vskip 5pt
  \noindent\emph{Step (3):}
  Let $t_1,t_2$ be strongly-general in $K|k$, and put $\Kfrak_i := \Kfrak(t_i)$ and $\epsilon_i := \epsilon_{\Kfrak_i}$ for $i = 1,2$.
  If $t_1,t_2$ are algebraically dependent over $k$, then $\Kfrak(t_1) = \Kfrak(t_2)$, so that $\epsilon_1 = \epsilon_2$.

  Assume, on the other hand, that $t_1,t_2$ are algebraically independent over $k$.
  Since $t_i$ is general in $K|k$, we see that $\{\{t_i-a\}_K \ : \ a \in k\}$ is a linearly-independent subset of $\k_1(K)$.
  Thus, by Proposition \ref{proposition: strgen / strgen / strgen-birat-bertini}, we may choose $a,b \in k$ such that $\{t_1-a\}_K$ and $\{t_2-b\}_K$ are $\Z/\ell$-independent in $\k_1(K)$ and such that $t_0 := (t_1-a)/(t_2-b)$ is strongly-general in $K|k$.
  Put $\Kfrak_0 = \Kfrak(t_0)$ and $\epsilon_0 = \epsilon_{\Kfrak_0}$.

  We can now calculate:
  \[ \epsilon_0 \cdot (\{t_1-a\}_K - \{t_2-b\}_K) = \sigma \left\{\frac{t_1-a}{t_2-b}\right\}_K = \epsilon_1 \cdot \{t_1-a\}_K - \epsilon_2 \cdot \{t_2-b\}_K.\]
  Since $\{t_1-a\}_K$ and $\{t_2-b\}_K$ are linearly independent in $\k_1(K)$, it follows that $\epsilon_1 = \epsilon_2 = \epsilon_0$.
  This proves that $\epsilon_\Kfrak$ doesn't depend on $\Kfrak = \Kfrak(t)$ for $t$ strongly-general in $K|k$.

  Letting $\epsilon = \epsilon_\Kfrak$ for some (hence any) $\Kfrak = \Kfrak(t)$ with $t$ strongly-general, we deduce that $\sigma|_\Kfrak = \epsilon \cdot \one|_\Kfrak$ for all $\Kfrak = \Kfrak(t)$ with $t$ strongly-general.
  By Lemma \ref{lemma: strgen / strgen / strgen-generate}, we see that $\k_1(K)$ is generated by its subgroups of the form $\Kfrak(t)$ for $t$ strongly-general in $K|k$. 
  Hence $\sigma = \epsilon \cdot \one_{\k_1(K)}$, as required.
\end{proof}

\subsection{Faithfulness of the Galois Action}
\label{subsection: strgen / galois-action}

We conclude this section by proving that the Galois action of $\Galk$ on the mod-$\ell$ Milnor K-ring of $K$ is faithful.
Although there are many ways to prove this fact, we can use geometric subgroups and the Birational-Bertini results to prove this for function fields of dimension $\geq 2$.
The result also holds for function fields of dimension $1$, but a different argument is needed in that case.

\begin{lemma}
\label{lemma: strgen / galois-action / faithfulness}
  Suppose that $\trdeg(K|k) \geq 1$.
  Then the canonical Galois representation
  \[ \rho_{k_0} : \Galk \rightarrow \UAutm(\k_1(K)) \]
  is injective.
  In particular, the Galois representation $\rho_{k_0} : \Galk \rightarrow \Autm(\k_1(K))$ is injective as well.
\end{lemma}
\begin{proof}
  First if $K = k(t)$ with $t \in K_0$, then the claim is clear, simply because the set 
  \[\{\{t-a\}_{k(t)} \ : \ a \in k\}\]
  is linearly independent in $\k_1(k(t))$, and since the action of $\tau \in \Galk$ on $\{t-a\}_{k(t)}$ satisfies $\tau \{t-a\}_{k(t)} = \{t-\tau a\}_{k(t)}$.

  Next, if $K$ has transcendence degree $\geq 2$, we note that the action of $\tau \in \Galk$ on $\k_1(K)$ restricts to an automorphism on any geometric subgroup of the form $\Kfrak(S)$ for $S \subset K_0$.
  By Lemma \ref{lemma: strgen / general / gen-birat-bertini}, there exists some $t \in K_0$ which is general in $K|k$.
  In this case, the map $\k_1(k(t)) \rightarrow \k_1(K)$ is injective with image $\Kfrak(t)$.
  This injection is compatible with the action of $\Galk$, so the assertion follows from the argument above.

  Finally, suppose that $\trdeg(K|k) = 1$ and suppose that $\tau$ is in the kernel of the canonical map 
  \[ \rho_{k_0} : \Galk \rightarrow \UAutm(\k_1(K)). \]
  Note that in order to show $\tau = \one$, it suffices to prove that $\tau a = a$ for all but finitely many $a \in k$.

  We now choose some $t \in K_0$ such that $K$ is finite and separable over $k(t)$.
  Let $C$ denote the complete normal model of $K|k$, and consider the finite (possibly branched) separable cover $C \rightarrow \Pbb^1_t$ induced by the inclusion $k(t) \hookrightarrow K$.

  Now let $a \in k$ be given such that the point $t = a$ in $\Pbb^1_t$ is unramified in the cover $C \rightarrow \Pbb^1_t$, and let $v$ be a divisorial valuation of $K|k$ whose center on $C$ lies above this point.
  By our assumption on $\tau$, it follows that 
  \[ \tau\{t-a\}_K = \{t-\tau a\}_K \in (\Z/\ell)^\times \cdot \{t-a\}_K. \]
  However, note that $\{t-a\}_K \notin \Ufrak_v$ while $\tau \Ufrak_v = \Ufrak_v$.
  In particular, we see that $\{t-\tau a\}_K \notin \Ufrak_v$, which directly implies that $\tau a = a$.
  Since this condition holds for all but finitely many $a \in k$, we deduce that $\tau = \one$, as required.

  Since the morphism $\rho_{k_0} : \Galk \rightarrow \UAutm(\k_1(K))$ factors through $\Autm(\k_1(K))$, we deduce that the morphism $\rho_{k_0} : \Galk \rightarrow \Autm(\k_1(K))$ is injective as well.
\end{proof}

\section{The Main Proof}
\label{section: mainproof}


We now turn to the proof of Theorems \ref{maintheorem: intro / galois-variant / galois-main} and \ref{maintheorem: intro / milnor-variant / milnor-main}, which is the main focus of this paper.
The primary focus will be on Theorem \ref{maintheorem: intro / milnor-variant / milnor-main}, since we have been primarily working with mod-$\ell$ Milnor K-theory, while Theorem \ref{maintheorem: intro / galois-variant / galois-main} will follow by applying Theorem \ref{theorem: cohom / kummer / galois-to-milnor}.

Using the notation from Theorem \ref{maintheorem: intro / milnor-variant / milnor-main}, recall that $\abf = (a_1,\ldots,a_r)$ is an arbitrary (possibly empty) finite tuple of elements of $k_0^\times$.
We start off the proof by working with a fixed element $\tau \in \Autm_\abf(\k_1(K))$, although we will eventually replace $\tau$ by another element $\sigma \in \Autm_\abf(\k_1(K))$ of the form $\epsilon \cdot \tau$ for some $\epsilon \in (\Z/\ell)^\times$.
In particular, $\sigma$ and $\tau$ represent the same element of $\UAutm_\abf(\k_1(K))$, but this $\sigma$ will have some further special properties which we will need.
In any case, if $A$ is any subgroup of $\k_1(K)$ and $\tau,\sigma$ are as above, then one has $\sigma A = \tau A$.
Since the primary goal of the proof is to show that $\tau$ induces an automorphism of the lattice of geometric subgroups of $\k_1(K)$, this observation shows that it doesn't actually matter if we replace $\tau$ with $\sigma = \epsilon \cdot \tau$.
We now fix this initial element $\tau \in \Autm_\abf(\k_1(K))$.

Recall that the definition of $\Autm_\abf(\k_1(K))$ says that $\tau$ is an automorphism of $\k_1(K)$ which satisfies the following two properties:
\begin{enumerate}
  \item $\tau$ extends to an automorphism of $\k_*(K)$.
  \item For all $t \in K_0 \smallsetminus k_0$, $\tau$ restricts to an automorphism of the subgroup
  \[ \langle \{t\}_K,\{t-a_1\}_K,\ldots,\{t-a_r\}_K \rangle. \]
\end{enumerate}
The following fact summarizes the formulation of these two conditions which will form the starting point of our proof.
\begin{fact}
\label{fact: mainproof / autm-condition}
  In the context above, the following hold:
  \begin{enumerate}
    \item Let $x_1,\ldots,x_r \in \k_1(K)$ be given.
    Then one has $\{x_1,\ldots,x_r\}_K = 0$ if and only if one has $\{\tau x_1,\ldots,\tau x_r\}_K = 0$.
    \item For all $t \in K_0 \smallsetminus k_0$, one has $\tau \{t\}_K \in \langle \{t\}_K,\{t-a_1\}_K,\ldots,\{t-a_r\}_K \rangle$.
  \end{enumerate}
\end{fact}

\subsection{Acceptable Subsets}
\label{subsection: mainproof / acceptable}

As discussed above, the primary goal of the proof is to show that the action of $\tau$ on the subgroups of $\k_1(K)$ induces an automorphism of the lattice $\Gfrak^*(K|k)$ of geometric subgroups.
With this in mind, we say that a subset $S \subset K$ is {\bf $\tau$-acceptable} if there exists a subset $T$ of $K$ such that
\[ \tau\Kfrak(S) = \Kfrak(T). \]
Thus, the primary goal of the proof is to show that \emph{every} subset of $K$ is $\tau$-acceptable, and we will then conclude the proof by applying Corollary \ref{corollary: lattice / galois-mod-ell / galois-lattice-iso}.
The following fact follows immediately from Lemma \ref{lemma: milnorff / milnordim / milnorsup} and Fact \ref{fact: mainproof / autm-condition}(1), and it essentially reduces our primary goal to showing that every \emph{element} of $K$ is acceptable.

\begin{fact}
\label{fact: mainproof / acceptable / acceptable}
  Suppose that $(S_i)_i$ is a collection of $\tau$-acceptable subsets of $K$, and for each $i$, let $T_i$ be a subset of $K$ such that
  \[ \tau \Kfrak(S_i) = \Kfrak(T_i). \]
  Put $S = \bigcup_i S_i$ and $T = \bigcup_i T_i$.
  Then one has $\tau\Kfrak(S) = \Kfrak(T)$.
  In particular, $S$ is $\tau$-acceptable.
\end{fact}

\subsection{Fixing $K_0$}
\label{susbection: mainproof / K0}

We begin by showing that every subset of $K_0$ is $\tau$-acceptable.
In fact, since our goal will be to apply Corollary \ref{corollary: lattice / galois-mod-ell / galois-lattice-iso}, we must show the stronger property that $\tau \Kfrak(S) = \Kfrak(S)$ for all subsets $S \subset K_0$.
This assertion is the starting point of the proof and is accomplished in the following lemma.

\begin{lemma}
\label{lemma: mainproof / K0 / fix-Kfrakt0}
  Let $S \subset K_0$ be a subset.
  Then one has $\tau\Kfrak(S) = \Kfrak(S)$.
  In particular, every subset of $K_0$ is $\tau$-acceptable.
\end{lemma}
\begin{proof}
  By Fact \ref{fact: mainproof / acceptable / acceptable}, it suffices to prove that $\tau\Kfrak(t) = \Kfrak(t)$ for all $t \in K_0 \smallsetminus k_0$.
  By Corollary \ref{corollary: milnorff / geometric / infinite-subsets}, one has
  \[ \Kfrak(t) = \bigcap_{c \in k_0} \ker\{t-c,\bullet\}_K. \]
  By Fact \ref{fact: mainproof / autm-condition}(1), we see that
  \[ \tau\Kfrak(t) = \bigcap_{c \in k_0} \ker\{\tau\{t-c\}_K,\bullet\}_K. \]

  On the other hand, Fact \ref{fact: mainproof / autm-condition}(2) implies that $\tau\{t-c\}_K \in \Kfrak(t)$ for all $c \in k_0$.
  Moreover, for all $x \in \Kfrak(t)$, one has $\Kfrak(t) \subset \ker\{x,\bullet\}_K$ by Fact \ref{fact: milnorff / vanishing / vanishing-B-K}.
  In particular, we deduce that $\Kfrak(t) \subset \tau \Kfrak(t)$.
  Repeating this argument with $\tau^{-1} \in \Autm_\abf(\k_1(K))$ in place of $\tau$, we deduce that $\Kfrak(t) = \tau\Kfrak(t)$, as required.
\end{proof}

Although Lemma \ref{lemma: mainproof / K0 / fix-Kfrakt0} will be used in the final steps of the proof, we will need a \emph{stronger} variant of this result.
Namely, we will need to prove that there exists a $\sigma$ of the form $\sigma = \epsilon \cdot \tau$ for some $\epsilon \in (\Z/\ell)^\times$ such that $\sigma \{t\}_K = \{t\}_K$ for all $t \in K_0^\times$.
This stronger variant appears in Proposition \ref{proposition: mainproof / K0 / fix-K0} below, and proving this proposition is the main goal of this subsection.
Naturally, this property can be seen as a crude approximation to our goal of proving that $\tau$ arises from some element of $\Galk$.

Let $(v_1,\ldots,v_r)$ be a flag of divisorial valuations of $K|k$ of length $r < \trdeg(K|k)$.
Since $\tau$ is an element of $\Autm(\k_1(K))$, we recall from \S\ref{subsection: localthy / qpd} (specifically Theorem \ref{theorem: localthy / qpd / main-qpd}) that there exists a unique flag $(v_1^\tau,\ldots,v_r^\tau)$ of \emph{quasi-divisorial valuations} of $K|k$ which satisfies
\[ \tau \Ufrak_{v_i} = \Ufrak_{v_i^\tau}, \ \ \tau \Ufrak_{v_i}^1 = \Ufrak_{v_i^\tau}^1 \]
for all $i = 1,\ldots,r$.
We now use Lemma \ref{lemma: mainproof / K0 / fix-Kfrakt0} to prove that this induced flag actually consists of \emph{divisorial valuations}.

\begin{lemma}
\label{lemma: mainproof / K0 / divs-local-thy}
  Let $(v_1,\ldots,v_r)$ be a flag of divisorial valuations of $K|k$, with $r < \trdeg(K|k)$.
  Then $(v_1^\tau,\ldots,v_r^\tau)$ is a flag of divisorial valuations of $K|k$.
\end{lemma}
\begin{proof}
  By Theorem \ref{theorem: localthy / qpd / main-qpd}, the valuation $v_i^\tau$ is the $i$-quasi-divisorial valuation of $K|k$ which is uniquely determined by the fact that
  \[ \tau\Ufrak_{v_i} = \Ufrak_{v_i^\tau}, \ \ \tau\Ufrak_{v_i}^1 = \Ufrak_{v_i^\tau}^1. \]
  On the other hand, by Proposition \ref{proposition: localthy / pd / pd-detection}, there exists some $t \in K_0 \smallsetminus k_0$ such that $\Kfrak(t) \cap \Ufrak_{v_i}^1$ is finite, since $v_i$ is $i$-divisorial.
  But then by Lemma \ref{lemma: mainproof / K0 / fix-Kfrakt0}, we see that
  \[ \tau(\Kfrak(t) \cap \Ufrak_{v_i}^1) = (\tau\Kfrak(t)) \cap (\tau \Ufrak_{v_i}^1) = \Kfrak(t) \cap \Ufrak_{v_i^\tau}^1 \]
  is finite as well.
  Hence $v_i^\tau$ is an $i$-divisorial valuation by Proposition \ref{proposition: localthy / pd / pd-detection}.
\end{proof}

To simplify the exposition for the rest of the proof, we will introduce some notation to label the elements of $\Dcal_t$ and $\Dfrak_t$ for a strongly-general element $t$.
If we fix an element $t \in K \smallsetminus k$ which is strongly-general in $K|k$, then the canonical map
\[ \Dfrak_t \rightarrow \Dcal_t \]
defined in Lemma \ref{lemma: essram / mod-ell-divs / mod-ell-div-injection} is a \emph{bijection}.
Since $\Kbb(t) = k(t)$ (as $t$ is, in particular, general), the elements of $\Dcal_t$ are in bijection with the closed points of $\Pbb^1_t$.
Thus $\Dfrak_t$ is also parametrized by the closed points of $\Pbb^1_t$, via the bijection $\Dfrak_t \rightarrow \Dcal_t$.
By fixing the parameter $t$, we can label the closed points of $\Pbb^1_t$ as $k \cup \{\infty\}$ in the usual way.
Namely, for $c \in k$, the corresponding point of $\Pbb^1_t$ is the closed point given by the equation $t-c = 0$, and the point associated to $\infty$ is the closed point given by $1/t = 0$.
We will denote the element of $\Dfrak_t$ associated to $c \in k \cup \{\infty\}$ by $\Vfrak[t;c]$, and we will denote the element of $\Dcal_t$ associated to $c$ by $v[t;c]$.
Using this notation, we recall from Lemma \ref{lemma: essram / mod-ell-divs / mod-ell-div-injection} that for all $c \in k \cup \{\infty\}$, one has
\[ \Vfrak[t;c] = \{\U_{v[t;c]}\}_K \]
as subgroups of $\k_1(K)$, and that one has $\Kfrak(t)/\Vfrak[t;c] \cong \Z/\ell$ for all $c \in k \cup \{\infty\}$.

A change of the parameter $t$ by some fractional-linear transformation yields a corresponding change in the associated element of $k \cup \{\infty\} = \Pbb^1(k)$.
More precisely, if we let ${\rm GL}_2(k)$ act on $\Pbb^1(k)$ and on the generators of $k(t)|k$ by fractional-linear transformations, as usual, and if $M \in {\rm GL}_2(k)$ is given, then for all $c \in \Pbb^1(k)$, one has
\[ \Vfrak[M \cdot t;c] = \Vfrak[t;M^{-1} \cdot c] \]
as subgroups of $\k_1(K)$.

For an element $x$ of $\Kfrak(t)$, we define the {\bf support of $x$ in $\Dfrak_t$} as
\[ \Supp_{\Dfrak_t}(x) = \{\Vfrak \in \Dfrak_t \ : \ x \notin \Vfrak \}.\]
This is completely analogous to the usual notion of the support of a function $f \in k(t)^\times$ in $\Pbb^1(k) = \Dcal_t$.
In particular, using the notation introduced above, note that one has
\[ \Supp_{\Dfrak_t}(\{t-c\}_K) = \{\Vfrak[t;c],\Vfrak[t;\infty]\}\]
for all $c \in k$.

The following lemma shows that $\tau$ fixes subgroups of the form $\Vfrak[t;c]$ which arise from $K_0$, i.e. such that $t \in K_0 \smallsetminus k_0$ and $c \in \Pbb^1(k_0)$.
This is a key step towards proving Proposition \ref{proposition: mainproof / K0 / fix-K0} below.
The proof of this lemma essentially follows by considering the $\Dfrak_t$-supports of various elements associated to $t$ and using Fact \ref{fact: mainproof / autm-condition}(2).

\begin{lemma}
\label{lemma: mainproof / K0 / fix-K0-strgen-divs}
  Let $t \in K_0 \smallsetminus k_0$ be strongly general in $K|k$, and let $c \in \Pbb^1(k_0)$ be given.
  Then one has $\tau\Vfrak[t;c] = \Vfrak[t;c]$.
\end{lemma}
\begin{proof}
  By replacing $t$ with another element of $K_0$ which generates $\Kbb(t) = k(t)$ over $k$, it suffices to prove that $\tau\Vfrak[t;\infty] = \Vfrak[t;\infty]$.
   By Fact \ref{fact: mainproof / autm-condition}(2), we know that
   \[ \tau\{t\}_K \in \langle\{t\}_K,\{t-a_1\}_K,\ldots,\{t-a_r\}_K \rangle. \]

   By combining Lemmas \ref{lemma: mainproof / K0 / fix-Kfrakt0} and \ref{lemma: mainproof / K0 / divs-local-thy} with the definition of $\Dfrak_t$, it follows that for all $\Vfrak \in \Dfrak_t$, one has $\tau \Vfrak \in \Dfrak_t$ as well.
   In other words, $\Vfrak \mapsto \tau\Vfrak$ can be considered as a permutation of $\Dfrak_t$.

   Note that the $\Dfrak_t$-support of any element of
   \[ \langle\{t\}_K,\{t-a_1\}_K,\ldots,\{t-a_r\}_K \rangle\]
   is contained in the set
   \[ \{\Vfrak[t;0],\Vfrak[t;a_1],\ldots,\Vfrak[t;a_r],\Vfrak[t;\infty]\}. \]
   And since the $\Dfrak_t$-support of $\{t\}_K$ is $\{\Vfrak[t;0],\Vfrak[t;\infty]\}$, it follows that one has
   \[ \tau\Vfrak[t;0] \in \{\Vfrak[t;0],\Vfrak[t;a_1],\ldots,\Vfrak[t;a_r],\Vfrak[t;\infty]\}.\]

   For $c \in k_0^\times$ and $a \in k$, one has $\Vfrak[t/c;a] = \Vfrak[t;ca]$; in particular $\Vfrak[t;0] = \Vfrak[t/c;0]$ for all $c \in k_0^\times$.
   By varying $c \in k_0^\times$ and using the fact that $k_0$ is infinite, we deduce that
   \[ \tau\Vfrak[t;0] \in \bigcap_{c \in k_0^\times}  \{\Vfrak[t;0],\Vfrak[t;ca_1],\ldots,\Vfrak[t;ca_r],\Vfrak[t;\infty]\} = \{\Vfrak[t;0],\Vfrak[t;\infty]\}. \]
   By replacing $t$ with $t^{-1}$, it follows similarly that
   \[ \tau\Vfrak[t;\infty] \in \{\Vfrak[t;0],\Vfrak[t;\infty]\}\]
   as well.
   Hence $\tau$ restricts to a permutation of $\{\Vfrak[t;0],\Vfrak[t;\infty]\}$.

   Finally, for $c \in k_0$, one has $\Vfrak[t-c;0] = \Vfrak[t;c]$.
   Repeating the argument above with $t-c$ for $c \in k_0$, we deduce that
   \[ \tau\Vfrak[t;c] \in \{\Vfrak[t;c],\Vfrak[t;\infty]\}\]
   for all $c \in k_0$.
   But since $\Vfrak[t;c] \neq \Vfrak[t;0]$ for $c \in k_0^\times$ (since $t$ is strongly-general in $K|k$), it follows that $\tau\Vfrak[t;\infty] = \Vfrak[t;\infty]$, as required.
\end{proof}

As mentioned above, the following proposition is the primary goal of this subsection, and it can be seen as the first major step towards the proof of Theorem \ref{maintheorem: intro / milnor-variant / milnor-main}.

\begin{proposition}
\label{proposition: mainproof / K0 / fix-K0}
  Assume that $\trdeg(K|k) \geq 2$.
  Then there exists a unique $\epsilon \in (\Z/\ell)^\times$ such that the following hold:
  \begin{enumerate}
    \item For all $t \in K_0^\times$, one has $\tau\{t\}_K = \epsilon \cdot \{t\}_K$.
    \item For all $t \in K_0$ which is strongly general in $K|k$ and for all $b \in k$, there exists a unique $c \in k$ such that $\tau\{t-b\}_K = \epsilon \cdot \{t-c\}_K$.
  \end{enumerate}
\end{proposition}
\begin{proof}
  First, let us assume that $t \in K_0\smallsetminus k_0$ is strongly-general in $K|k$.
  Recall that the following hold:
  \begin{enumerate}
    \item First, $\tau \Kfrak(t) = \Kfrak(t)$ by Lemma \ref{lemma: mainproof / K0 / fix-Kfrakt0}.
    \item Second, one has $\tau \Vfrak \in \Dfrak_t$ for all $\Vfrak \in \Dfrak_t$ by Lemma \ref{lemma: mainproof / K0 / divs-local-thy} and the definition of $\Dfrak_t$.
    \item Third, for all $c \in k_0 \cup \{\infty\} = \Pbb^1(k_0)$, one has $\tau\Vfrak[t;c] = \Vfrak[t;c]$ by Lemma \ref{lemma: mainproof / K0 / fix-K0-strgen-divs}.
  \end{enumerate}
  In particular, $\tau$ induces isomorphisms $\tau : \Kfrak(t)/\Vfrak \xrightarrow{\cong} \Kfrak(t)/\tau\Vfrak$ for every $\Vfrak \in \Dfrak_t$.

  Let $\Psi = (\Psi_\Vfrak)_{\Vfrak \in \Dfrak_t}$ be the canonical rational-like collection associated to $\Kbb(t)|k$, as defined in \S\ref{subsection: strgen / rational-like}.
  Consider the induced rational-like collection $\Psi^\tau = (\Psi^\tau_{\Vfrak})_{\Vfrak \in \Dfrak_t}$, where the isomorphism $\Psi^\tau_{\tau\Vfrak} : \Kfrak(t)/\tau\Vfrak \xrightarrow{\cong}\Z/\ell$ indexed by $\tau\Vfrak$ is the unique one making the following diagram commute:
  \[
  \xymatrix{
    \Kfrak(t) \ar[d]_\tau \ar@{->>}[r] & \Kfrak(t)/\Vfrak \ar[d]_\tau \ar[r]^-{\Psi_\Vfrak} & \Z/\ell \ar@{=}[d] \\
    \Kfrak(t) \ar@{->>}[r] & \Kfrak(t)/\tau\Vfrak \ar[r]_-{\Psi^\tau_{\tau\Vfrak}} & \Z/\ell.
  }
  \]
  By Lemma \ref{lemma: strgen / rational-like / rational-like}, there exists an $\epsilon_t \in (\Z/\ell)^\times$ such that $\epsilon_t \cdot \Psi_{\tau\Vfrak} = \Psi^\tau_{\tau\Vfrak}$ for all $\Vfrak \in \Dfrak_t$.
  Since it will be used later on, recall that this $\epsilon_t$ only depends on $\Kfrak(t)$ and $\tau$.
  
  \vskip 5pt
  \noindent\emph{Proof of (1):} 
  Now let $c \in k_0$ be given.
  Then one has $\divv_\Psi(\{t-c\}_K) = [\Vfrak[t;c]] - [\Vfrak[t;\infty]]$ since $\Psi$ is the canonical rational-like collection associated to $\Kbb(t)$.
  
  By Lemma \ref{lemma: mainproof / K0 / fix-K0-strgen-divs}, one has $\tau\Vfrak[t;c] = \Vfrak[t;c]$ and $\tau\Vfrak[t;\infty] = \Vfrak[t;\infty]$.
  And since $\Psi^\tau_{\tau\Vfrak} = \epsilon_t \cdot \Psi_{\tau\Vfrak}$ for all $\Vfrak \in \Dfrak_t$, we deduce that
  \[ \divv_\Psi(\tau\{t-c\}_K) = \divv_{\Psi^\tau}(\{t-c\}_K) = \epsilon_t \cdot [\Vfrak[t;c]] - \epsilon_t \cdot [\Vfrak[t;\infty]] = \divv_\Psi(\epsilon_t \cdot \{t-c\}_K). \]
  The injectivity of $\divv_\Psi$ implies that $\tau\{t-c\}_K = \epsilon_t \cdot \{t-c\}_K$.

  In order to conclude the proof of assertion (1), we will apply Lemma \ref{lemma: strgen / strgen / strgen-generate} to the subfield $K_0$ of $K$.
  In light of this, it suffices to show that the $\epsilon_t$ which appears in the argument above is independent of $\Kfrak(t)$ for $t \in K_0$ which is strongly-general in $K|k$.

  With this in mind, suppose that $u,t \in K_0$ are both strongly-general in $K|k$.
  Let $\epsilon_t$ resp. $\epsilon_u$ be as above, and assume for a contradiction that $\epsilon_t \neq \epsilon_u$.
  Since the $\epsilon_t$ resp. $\epsilon_u$ depends only on $\Kfrak(t)$ resp. $\Kfrak(u)$ and $\tau$, this implies that $\Kfrak(t) \neq \Kfrak(u)$, hence $\Kbb(t) \neq \Kbb(u)$ by Proposition \ref{proposition: lattice / mod-ell / mod-ell-lattice}.
  In particular, $u,t$ are algebraically independent over $k$.

  By Proposition \ref{proposition: strgen / strgen / strgen-birat-bertini}, there exists a non-empty open subset $U$ of $\Abb^2_k$ such that $(a,b) \in U(k)$ implies that $(t-a)/(u-b)$ is strongly-general in $K|k$.
  Moreover, since $t,u$ are general in $K|k$, it follows that the sets $\{\{t-a\}_K\}_{a \in k_0}$, $\{\{u-a\}_K\}_{a \in k_0}$ are $\Z/\ell$-independent in $\k_1(K)$.
  Since $k_0$ is infinite, there exists some $(a,b) \in U(k_0)$ such that $\{t-a\}_K$ and $\{u-b\}_K$ are $\Z/\ell$-independent in $\k_1(K)$.

  Put $x = (t-a)/(u-b)$, and recall that $x \in K_0$ and that $x$ is strongly-general in $K|k$.
  Let $\epsilon_x$ be again as above.
  Now we calculate, similarly to the proof of Proposition \ref{proposition: strgen / rational-like / ident-criterion}:
  \[\epsilon_x \cdot (\{t-a\}_K - \{u-b\}_K) = \tau\left\{\frac{t-a}{u-b}\right\}_K = \epsilon_t \cdot \{t-a\}_K-\epsilon_u \cdot \{u-b\}_K.\]
  Since $\{t-a\}_K$ and $\{u-b\}_K$ are linearly-independent in $\k_1(K)$, it follows that $\epsilon_t = \epsilon_x = \epsilon_u$.
  This concludes the proof of (1), and we let $\epsilon$ be the unique element of $(\Z/\ell)^\times$ such that $\tau\{t\}_K = \epsilon \cdot \{t\}_K$ for all $t \in K_0$.
  I.e. $\epsilon = \epsilon_t$ for some/all $t \in K_0$ which is strongly-general in $K|k$.

  \vskip 5pt
  \noindent\emph{Proof of (2):} Let $b \in k$ be given, assume that $t \in K_0$ is strongly-general in $K|k$, and let $\epsilon$ be as in the proof of (1) above.
  By Lemma \ref{lemma: mainproof / K0 / fix-K0-strgen-divs}, we see that $\tau\Vfrak[t;\infty] = \Vfrak[t;\infty]$.
  Since $\Vfrak[t;b] \neq \Vfrak[t;\infty]$, it follows that there exists a unique $c \in k$ such that $\tau\Vfrak[t;b] = \Vfrak[t;c]$.
  Arguing similarly to the above, we have
  \begin{align*}
    \divv_{\Psi}(\tau\{t-b\}_K) = \divv_{\Psi^\tau}(\{t-b\}_K) &= \epsilon \cdot [\tau\Vfrak[t;b]] - \epsilon \cdot [\tau\Vfrak[t;\infty]] \\
                  &= \epsilon \cdot [\Vfrak[t;c]] - \epsilon \cdot [\Vfrak[t;\infty]] \\
                  &= \divv_\Psi(\epsilon \cdot \{t-c\}_K).
  \end{align*}
  The injectivity of $\divv_\Psi$ implies that $\tau\{t-b\}_K = \epsilon \cdot \{t-c\}_K$, as required.
  Finally, the uniqueness of $c$ follows from the fact that $c \mapsto \Vfrak[t;c] : \Pbb^1(k) \rightarrow \Dfrak_t$ is a bijection, and that $c$ is uniquely determined by the condition that $\tau\Vfrak[t;b] = \Vfrak[t;c]$.
\end{proof}

At this point, we will fix $\sigma := \epsilon^{-1} \cdot \tau$ with $\epsilon \in (\Z/\ell)^\times$ as in Proposition \ref{proposition: mainproof / K0 / fix-K0}.
In particular, $\sigma$ has the following crucial properties:
\begin{enumerate}
  \item $\sigma$ and $\tau$ represent the same element of $\UAutm_\abf(\k_1(K))$.
  In particular $\sigma A = \tau A$ for all subgroups $A$ of $\k_1(K)$.
  Thus, one has $\sigma\Kfrak(S) = \Kfrak(S)$ for all $S \subset K_0$ by Lemma \ref{lemma: mainproof / K0 / fix-Kfrakt0}.
  \item One has $\sigma\{t\}_K = \{t\}_K$ for all $t \in K_0^\times$ by Proposition \ref{proposition: mainproof / K0 / fix-K0}(1).
  \item If $t \in K_0$ is strongly-general in $K|k$, and $b \in k$ is arbitrary, then there exists a unique element $c \in k$ such that $\sigma\{t-b\}_K = \{t-c\}_K$, by Proposition \ref{proposition: mainproof / K0 / fix-K0}(2).
\end{enumerate}

At this stage in the proof, the main missing step is that we have no apparent control over the $c$ which appears in point (3) above.
Namely, this $c$ might \emph{a priori} vary as we change the element $t$.
The next proposition gives a crucial sufficient condition for the independence of this $c$ from $t$.

\begin{proposition}
\label{proposition: mainproof / K0 / independence-proposition}
  Let $x,y \in K_0$ be two elements which are algebraically independent over $k$.
  Let $e,f \in k_0$ be such that $t_e := (x-e)/y$ and $t_f := (x-f)/y$ are strongly-general in $K|k$, and let $b \in k$ be given.
  Assume that there exists a regular $k$-variety $X$ and a dominant smooth morphism $X \rightarrow \Pbb^2_{x,y}$ with geometrically integral fibers, such that the following hold:
  \begin{enumerate}
    \item The closed point $(0:b:1) \in \Pbb^2_{x,y}$ is in the image of $X \rightarrow \Pbb^2_{x,y}$.
    \item $K$ is a finite extension of $k(X)$.
    \item The fiber of $X \rightarrow \Pbb^2_{x,y}$ over $(0:b:1)$ is essentially unramified in $K$.
  \end{enumerate}
  Then there exists a unique $c \in k$ such that $\sigma\{t_e-b\}_K = \{t_e-c\}_K$ and $\sigma\{t_f-b\}_K = \{t_f-c\}_K$.
\end{proposition}
\begin{proof}
  For $f_1,\ldots,f_s \in k[x,y]$, we denote by $V(f_1,\ldots,f_s)$ the (reduced) closed subvariety of $\Abb^2_{x,y}$ given by the ideal $\sqrt{(f_1,\ldots,f_s)}$.
  We will abuse the notation, and also write $V(f_1,\ldots,f_r)$ for the closure of this subvariety in $\Pbb^2_{x,y}$, also considered as a (reduced) subvariety.
  Following our notational conventions, we denote the closed points of $\Pbb^2_{x,y}$ by their associated $k$-rational point, written in homogeneous coordinates.
  In particular, the closed points of $\Abb^2_{x,y}$ are written as $(1:a_1:a_2)$, $a_1,a_2 \in k$, and the points of $\Pbb^2_{x,y} \smallsetminus \Abb^2_{x,y}$ are written as $(0:a_1:a_2)$, $a_1,a_2 \in k$, $(a_1,a_2) \neq (0,0)$.

  Let $v$ denote the divisorial valuation of $k(x,y)|k$ associated to the prime Weil-divisor $Z := \Pbb^2_{x,y} \smallsetminus \Abb^2_{x,y}$ on $\Pbb^2_{x,y}$.
  For any $a \in k$, let $v_a$ be the $2$-divisorial valuation of $k(x,y)|k$ refining $v$ which is associated to the rational point $(0:a:1)$ on $Z$.
  Thus $(v,v_a)$ is a flag of divisorial valuations of $k(x,y)|k$ of length $2$.

  Let $\eta$ denote the generic point of $Z$.
  To simplify the exposition later in the proof, we will say that $a \in k$ is {\bf allowable} if the fiber of $X \rightarrow \Pbb^2_{x,y}$ over the point $(0:a:1)$ is non-empty and essentially unramified in $K$.
  Since $(0:b:1)$ lies on $Z$, assumptions (1) and (3) imply that the fiber $X_\eta$ of $X \rightarrow \Pbb^2_{x,y}$ over $\eta$ is non-empty and essentially unramified in $K$.
  Therefore, we see that all but finitely many $a \in k$ (including $b$) are allowable.

  For $h \in k$, we write $t_h := (x-h)/y$.
  We also write $t := t_0 = x/y$.
  Note that for all $a \in k$, the system $(1/x,t-a)$ is a system of regular parameters for the regular closed point $(0:a:1)$, and the flag of divisorial valuations associated to this system is precisely $(v,v_a)$.
  Thus, by the discussion of \S\ref{subsection: essram / essential-unram} and using Fact \ref{fact: essram / essential-branch-locus / essential-branch-fact}, we see that there exists an $\ell$-unramified prolongation $w$ of $v$ to $K$.
  In other words, $w$ is a divisorial valuation of $K|k$ which prolongs $v$, and the canonical map
  \[ \k_1(k(x,y))/\{\U_v\}_{k(x,y)} \rightarrow \k_1(K)/\Ufrak_w \cong \Z/\ell \]
  is an isomorphism.
  Using the same argument, for any allowable $a \in k$, there exists an $\ell$-unramified prolongation $w_a$ of $v_a$ to $K$ which \emph{refines} $w$.
  In other words, for such $a \in k$, $w_a$ is a $2$-divisorial valuation of $K|k$ which refines $w$ and prolongs $v_a$, and one has a canonical isomorphism
  \[ \{\U_v\}_{k(x,y)}/\{\U_{v_a}\}_{k(x,y)} \rightarrow \Ufrak_w/\Ufrak_{w_a}. \]
  Note also that one has $\{t_h-a\}_{k(x,y)} \in \{\U_v\}_{k(x,y)}\smallsetminus\{\U_{v_a}\}_{k(x,y)}$ for all $h \in k$.
  To summarize, for any allowable $a \in k$, the flag $(w,w_a)$ is an $\ell$-unramified prolongation of $(v,v_a)$ to $K$, and the following conditions hold:
  \begin{enumerate}
    \item One has $\{1/x\}_K \in \k_1(K) \smallsetminus \Ufrak_w$ hence $\{x\}_K \in \k_1(K) \smallsetminus \Ufrak_w$.
    \item For all $h \in k$, one has $\{t_h-a\}_K \in \Ufrak_w \smallsetminus \Ufrak_{w_a}$.
  \end{enumerate}

  By Lemma \ref{lemma: mainproof / K0 / divs-local-thy}, there exists a unique divisorial valuation $w^\sigma$ of $K|k$, and for every allowable $a \in k$ as above, there exists a unique $2$-divisorial valuation $w_a^\sigma$ of $K|k$, such that the following conditions hold:
  \begin{enumerate}
    \item $w_a^\sigma$ refines $w^\sigma$.
    \item $\sigma\Ufrak_w = \Ufrak_{w^\sigma}$ and $\sigma \Ufrak_w^1 = \Ufrak_{w^\sigma}^1$.
    \item $\sigma\Ufrak_{w_a} = \Ufrak_{w_a^\sigma}$ and $\sigma \Ufrak_{w_a}^1 = \Ufrak_{w_a^\sigma}^1$.
  \end{enumerate}
  Finally, we let $v^\sigma$ resp. $v_a^\sigma$ be the restrictions of $w^\sigma$ resp. $w_a^\sigma$ to $k(x,y)$.

  \begin{claim*}
    One has $v^\sigma = v$.
  \end{claim*}
  \begin{proof}
    Recall that $\eta$ denotes the generic point of $Z$ in $\Pbb^2_{x,y}$.
    As $\Pbb^2_{x,y}$ is proper over $k$, the valuation $v^\sigma$ has a unique center on $\Pbb^2_{x,y}$.
    Since $\eta$ is a regular codimension $1$ point of $\Pbb^2_{x,y}$, it suffices to show that the center of $v^\sigma$ on $\Pbb^2_{x,y}$ is $\eta$.

    Recall that one has $\sigma\{x\}_K = \{x\}_K$ by Proposition \ref{proposition: mainproof / K0 / fix-K0}, and that $\sigma\Ufrak_w = \Ufrak_{w^\sigma}$ as noted above.
    Since $\{x\}_K \notin \Ufrak_w$, we see that
    \[ \{x\}_K = \sigma \{x\}_K \notin \Ufrak_{w^\sigma}. \]
    In particular, $v^\sigma(x) \neq 0$.
    This implies that the center of $v^\sigma$ on $\Pbb^2_{x,y}$ is contained in the support of the function $x$, which is $V(x) \cup Z$.

    On the other hand, note that one also has $\{x-1\}_{k(x,y)} \notin \{\U_v\}_{k(x,y)}$, hence $\{x-1\}_K$ is also a generator of $\k_1(K)/\Ufrak_w$.
    By repeating the same argument with $x-1$ in place of $x$, and noting that $\sigma\{x-1\}_K = \{x-1\}_K$ by Proposition \ref{proposition: mainproof / K0 / fix-K0}, we deduce that the center of $v^\sigma$ on $\Pbb^2_{x,y}$ must actually be contained in $Z$, since the support of $x-1$ on $\Pbb^2_{x,y}$ is $V(x-1) \cup Z$.

    Now assume for a contradiction that the center of $v^\sigma$ is a closed point $P$ on $Z$.
    For an allowable $a \in k$, recall that $v_a^\sigma$ refines $v$.
    It therefore follows that the center of $v_a^\sigma$ on $\Pbb^2_{x,y}$ must also be this closed point $P$.

    Now let $a_0 \in k_0$ be allowable -- such an $a_0$ exists since $k_0$ is infinite.
    By Proposition \ref{proposition: mainproof / K0 / fix-K0}, we have $\{t-a_0\}_K = \sigma\{t-a_0\}_K$.
    Therefore $\{t-a_0\}_K \notin \Ufrak_{w_{a_0}^\sigma}$, hence $v_{a_0}^\sigma(t-a_0) \neq 0$.
    By considering the support of $t-a_0$ in $\Pbb^2_{x,y}$, we see that the point $P$ must be either the point $(0:a_0:1)$ or the point $(0:1:0)$.
    In fact, it follows that $P$ must be the point $(0:1:0)$ since we can repeat this argument with a different allowable $a_1 \in k_0$, $a_1 \neq a_0$, which exists since $k_0$ is infinite.

    Now consider the function $(t-a_0)/(t-a_1)$ for $a_1 \in k_0$ distinct from $a_0$.
    By Proposition \ref{proposition: mainproof / K0 / fix-K0}, we have $\{(t-a_0)/(t-a_1)\}_K = \sigma\{(t-a_0)/(t-a_1)\}_K$, while $\{t-a_1\}_K \in \Ufrak_{w_{a_0}}$.
    Therefore, $\{(t-a_0)/(t-a_1)\}_K \notin \Ufrak_{w_{a_0}^\sigma}$, hence $v_{a_0}^\sigma((t-a_0)/(t-a_1)) \neq 0$.
    By looking at the support of $(t-a_0)/(t-a_1)$ on $\Pbb^2_{x,y}$, we obtain a contradiction to the fact that $P = (0:1:0)$.

    Having obtained our contradiction, it follows that the center of $v^\sigma$ cannot be a closed point on $Z$.
    This means that the center of $v^\sigma$ must be the generic point of $Z$, which concludes the proof of the claim.
  \end{proof}

  Recall from Proposition \ref{proposition: mainproof / K0 / fix-K0}(2) that there exist unique constants $b_e,b_f \in k$ such that $\sigma\{t_e-b\}_K = \{t_e-b_e\}_K$ and $\sigma\{t_f-b\}_K = \{t_f-b_f\}_K$.
  We must show that $b_e = b_f$.

  Note that one has $\{t_e-b\}_K \in \Ufrak_w \smallsetminus \Ufrak_{w_b}$, so that by Proposition \ref{proposition: mainproof / K0 / fix-K0}, we have
  \[ \{t_e-b_e\}_K \in \Ufrak_{w^\sigma} \smallsetminus \Ufrak_{w_b^\sigma}. \]
  In particular, $v_b^\sigma$ is a proper refinement of $v^\sigma$.
  Thus, the center $Q$ of $v_b^\sigma$ on $\Pbb^2_{x,y}$ must be a closed point on the line $Z$ at infinity, for otherwise this center would be $\eta$ which would imply that $v^\sigma = v^\sigma_b$ by the claim above.

  Let $a \in k_0$ be different from $b$.
  By Proposition \ref{proposition: mainproof / K0 / fix-K0}, we see that $\sigma\{(t_e-b)/(t_e-a)\}_K = \{(t_e-b_e)/(t_e-a)\}_K$, and thus
  \[ \{(t_e-b_e)/(t_e-a)\}_K \notin \Ufrak_{w_b^\sigma}. \]
  By considering the support of the function $(t_e-b_e)/(t_e-a)$ on $\Pbb^2_{x,y}$, it follows that $Q$ cannot be the point $(0:1:0)$.
  Thus, this center $Q$ must be a point of the form $(0:c:1)$ for some $c \in k$.
  Recalling that
  \[ \{t_e-b_e\}_K \in \Ufrak_{w^\sigma} \smallsetminus \Ufrak_{w_b^\sigma}\]
  so that $v_b^\sigma(t_e-b_e) \neq 0$, we deduce from this that $b_e = c$.
  Finally, by noting that also
  \[ \{t_f-b_f\}_K \in \Ufrak_{w^\sigma} \smallsetminus \Ufrak_{w_b^\sigma}, \]
  we deduce similarly that $b_f = c$.
  Thus, one has $c = b_e = b_f$, as required.
\end{proof}

\subsection{Intersections}
\label{subsection: mainproof / intersections}

We now turn to the second main part of the proof.
In this part, we develop a condition for detecting one-dimensional geometric subgroups as intersections of certain two-dimensional geometric subgroups.
The first lemma in this direction shows that any one-dimensional geometric subgroup can actually be realized as the intersection of two two-dimensional geometric subgroups.
The second lemma in this direction gives a sufficient condition for the intersection of two two-dimensional geometric subgroups to be geometric.

\begin{lemma}
\label{lemma: mainproof / intersections / gen-intersections}
  Assume that $\trdeg(K|k) \geq 4$.
  Let $t \in K \smallsetminus k$ be given, and put $F = \Kbb(t)$.
  Let $t_1,t_2,t_3 \in K$ be algebraically independent over $F$.
  Then there exist $a,b \in k_0$ such that
  \[ \Kfrak(t) = \Kfrak(t,t_1+a t_3) \cap \Kfrak(t,t_2+b t_3). \]
\end{lemma}
\begin{proof}
  Extend $t_1,t_2,t_3$ to a transcendence base $\tbf = (t_1,\ldots,t_r)$ for $K|F$, and let $K'$ denote the maximal separable subextension of $K|F(\tbf)$.
  Note that $K'$ is then a regular extension of $F$.
  By Lemma \ref{lemma: strgen / general / gen-birat-bertini}(1), for all but finitely many $a \in k$, the field $K'$ is regular over $F(t_1+a \cdot t_3)$; since $k_0$ is infinite, we may choose such an $a$ which lies in $k_0$.
  Applying Lemma \ref{lemma: strgen / general / gen-birat-bertini}(1) again, for all but finitely many $b \in k$, the field $K'$ is regular over $F(t_1 + a \cdot t_3,t_2 + b \cdot t_3)$; again, we may choose such a $b$ which lies in $k_0$.
  We put $x = t_1 + a \cdot t_3$ and $y = t_2 + b \cdot t_3$ for $a,b \in k_0$ as above.

  Next, put $F_1 = F(x)$, $F_2 = F(y)$ and $F_{12} = F(x,y)$, and note that $F_*$ is relatively algebraically closed in $K'$ for $* = 1,2,12$.
  Also, for $* = 1,2,12$, let $M_*$ denote the relative algebraic closure of $F_*$ in $K$.
  Since $K|K'$ is purely-inseparable, the extension $M_*|F_*$ is also purely-inseparable.

  Since $F_{12} = F(x,y)$ is rational over $F$ and $F_1 = F(x)$, $F_2 = F(y)$, it is easy to see that one has $\{F_1^\times\}_{F_{12}} \cap \{F_2^\times\}_{F_{12}} = \{F^\times\}_{F_{12}}$ as subgroups of $\k_1(F_{12})$.
  Since $\k_1(F_{12}) \rightarrow \k_1(K')$ is injective, we deduce that
  \[ \{F_1^\times\}_{K'} \cap \{F_2^\times\}_{K'} = \{F^\times\}_{K'}. \]
  Finally, the fact that 
  \[ \{M_1^\times\}_K \cap \{M_2^\times\}_K = \{F^\times\}_K = \Kfrak(t) \]
  follows easily from Lemma \ref{lemma: milnor / insep / insep-iso} using the fact that the extensions $M_*|F_*$, $* = 1,2,12$, and $K|K'$ are purely-inseparable.
  This concludes the proof of the lemma since $\Kfrak(t,x) = \{M_1^\times\}_K$ and $\Kfrak(t,y) = \{M_2^\times\}_K$.
\end{proof}

\begin{lemma}
\label{lemma: mainproof / intersections / geometric-criterion}
  Let $\Kfrak$ be a subgroup of $\k_1(K)$ which is maximal among the subgroups $\Delta$ of $\k_1(K)$ such that $\dimm(\Delta) = 1$.
  Assume furthermore that there exist elements of $\Gfrak^*(K|k)$ denoted as follows
  \begin{itemize}
    \item $A \in \Gfrak^1(K|k)$
    \item $B_1,B_2,C \in \Gfrak^2(K|k)$
    \item $B_1',B_2',D \in \Gfrak^3(K|k)$
  \end{itemize}
  such that the following conditions hold:
  \begin{enumerate}
    \item $B_1 \cup B_2 \subset D$, $B_1 \neq B_2$, and $A \not\subset D$.
    \item $B_1 \cup A \subset B_1'$ and $B_2 \cup A \subset B_2'$.
    \item $C \subset B_1' \cap B_2'$ and $\Kfrak = B_1 \cap B_2$.
  \end{enumerate}
  Then $\Kfrak$ is a one-dimensional geometric subgroup of $\k_1(K)$, i.e. one has $\Kfrak \in \Gfrak^1(K|k)$.
\end{lemma}
\begin{proof}
  First suppose that $t \in K \smallsetminus k$ is given, and consider the geometric subgroup $\Kfrak(t)$.
  Then for all $a,b \in \Kfrak(t)$, one has $\{a,b\}_K = 0$ by Fact \ref{fact: milnorff / vanishing / vanishing-B-K}.
  Moreover, by Proposition \ref{proposition: milnorff / geometric / geometric-maximal}, if $c \in \k_1(K) \smallsetminus \Kfrak(t)$, then there exist (many) elements $d \in \Kfrak(t)$ such that $\{c,d\}_K \neq 0$.
  In particular, $\Kfrak(t)$ is maximal among subgroups $\Delta$ of $\k_1(K)$ such that $\dimm(\Delta) = 1$.

  Now suppose that $\Kfrak$ satisfies the assumptions on the lemma.
  The goal of this proof will be to show that there exists some $t \in K \smallsetminus k$ such that $\Kfrak(t) \subset \Kfrak$.
  Then the ``maximality'' in the observation above would imply that $\Kfrak(t) = \Kfrak$.
  We will tacitly use Proposition \ref{proposition: lattice / mod-ell / mod-ell-lattice}(1), which says that for $E_1,E_2 \in \G^*(K|k)$, one has $E_1 \subset E_2$ if and only if $\Kfrak(E_1) \subset \Kfrak(E_2)$ as subgroups of $\k_1(K)$.

  Let $F_1,F_2 \in \G^2(K|k)$ be such that $\Kfrak(F_i) = B_i$ and put $F = F_1 \cap F_2$.
  Then condition (1) implies that $F_1 \neq F_2$ hence $\trdeg(F|k) \leq 1$.
  Let $F_{12} \in \G^3(K|k)$ be such that $\Kfrak(F_{12}) = D$, then condition (1) implies that $F_1 \cdot F_2 \subset F_{12}$.
  Let $x \in K \smallsetminus k$ be such that $A = \Kfrak(x)$.
  Condition (1) implies that $x$ is transcendental over $F_{12}$, hence it is also transcendental over $F_1$ and $F_2$.
  Thus, one has $B_i' = \Kfrak(F_i,x)$ by condition (2).

  On the other hand, $x$ being transcendental over $F_{12}$, and $F_1 \cdot F_2 \subset F_{12}$, implies that 
  \[ \Kbb(F_1,x) \cap \Kbb(F_2,x) = \Kbb(F_1 \cap F_2,x) = \Kbb(F,x). \]
  On the other hand, letting $M \in \G^2(K|k)$ be such that $\Kfrak(M) = C$, we deduce from condition (3) that $M \subset \Kbb(F_1,x)$ and $M \subset \Kbb(F_2,x)$, hence $\trdeg(\Kbb(F,x)|k) \geq 2$.
  Since $\trdeg(F|k) \leq 1$ holds, we deduce that $\trdeg(\Kbb(F,x)|k) = 2$ and therefore $\trdeg(F|k) = 1$.

  Finally, one has 
  \[ \Kfrak(F) = \Kfrak(F_1 \cap F_2) \subset \Kfrak(F_1) \cap \Kfrak(F_2) = B_1 \cap B_2. \]
  Therefore, $\Kfrak(F) \subset \Kfrak$ by condition (3).
  Letting $t \in K \smallsetminus k$ be such that $\Kbb(t) = F$, we deduce that $\Kfrak(t) \subset \Kfrak$, hence $\Kfrak = \Kfrak(t)$ as noted in the beginning of the proof.
\end{proof}

\subsection{The Base Case}
\label{subsection: mainproof / base-case}

Recall that our primary goal is to show that every \emph{element} of $K$ is $\sigma$-acceptable.
Also, recall that one has $K = K_0 \otimes_{k_0} k$, and therefore every element $t$ of $K$ can be written as
\[ t = a_0 x_0 + \cdots + a_r x_r \]
for some $x_0,\ldots,x_r \in K_0$ and some $a_0,\ldots,a_r \in k$.
We have already proved in Lemma \ref{lemma: mainproof / K0 / fix-Kfrakt0} that elements of $K_0$ are $\sigma$-acceptable, hence elements of the form $a_0 x_0$ where $a_0 \in k$ and $x_0 \in K_0$ are also $\sigma$-acceptable.
The proof that every element of $K$ is $\sigma$-acceptable follows by induction on the length $r$ of the expression $a_0 x_0 + \cdots + a_r x_r$ above.
The base case for our induction is the case $r = 1$, which is the focus of this subsection.

We begin by proving that ``many'' elements of the form $a_0 x_0 + a_1 x_1$ are $\sigma$-acceptable, and we will then use the ``intersection'' results proved in the previous subsection to deduce the full base case.
\begin{lemma}
\label{lemma: mainproof / base-case / pre-base-case}
  Let $t_1,t_2,t_3 \in K_0$ be algebraically independent over $k$.
  Then there exists a non-empty open subset $U$ of $\Abb^3_k$ such that for all $(a_1,a_2,a_3) \in U(k_0)$, and for all $b \in k$, the element
  \[ \frac{t_2-a_2}{t_1-a_1} + b \cdot \frac{t_3-a_3}{t_1-a_1}\]
  is $\sigma$-acceptable.
\end{lemma}
\begin{proof}
  Extend $t_1,t_2,t_3$ to a transcendence base $\tbf := (t_1,\ldots,t_d)$ for $K|k$ with $t_i \in K_0$.
  Let $\tilde U$ denote the (non-empty, open) complement of the essential branch locus of $\Abb^d_\tbf$ in $K$, and let $U$ denote the (non-empty, open) image of $\tilde U$ under the natural coordinate projection
  \[ \Abb^d_\tbf \rightarrow \Abb^3_{t_1,t_2,t_3}. \]
  Since the map $\tilde U \rightarrow U$ is just a coordinate projection of affine spaces, and since $k_0$ is infinite, we note that the map $\tilde U(k_0) \rightarrow U(k_0)$ is surjective.
  By Lemma \ref{lemma: strgen / general / gen-birat-bertini}(2), we may replace $U$ with a smaller non-empty open subset, if needed, and assume without loss of generality that for all $(a_1,a_2,a_3) \in U(k_0)$, the element $(t_1-a_1)/(t_3-a_3)$ is general (hence separable) in $K|k$.
  We will show that this open set $U$ satisfies the assertion of the lemma.
  
  Let $p := (a_1,a_2,a_3) \in U(k_0)$ be given, and let $P := (a_1,\ldots,a_d) \in \tilde U(k_0)$ be a point lying above $p$.
  Put $x = (t_2-a_2)/(t_1-a_1)$ and $y = (t_3-a_3)/(t_1-a_1)$, and consider the (rational) projection about $p$,
  \[ \Abb^3_{t_1,t_2,t_3} \rightarrow \Pbb^2_{x,y}, \]
  which is defined on points by $(X,Y,Z) \mapsto (X-a_1:Y-a_2:Z-a_3)$.
  Let $Z$ be the fiber of $p$ with respect to $\Abb^d_\tbf \rightarrow \Abb^3_{t_1,t_2,t_3}$, and put $X := \Abb^d_\tbf \smallsetminus Z$.
  The composition of the following maps
  \[ X \hookrightarrow \Abb^d_\tbf \rightarrow \Abb^3_{t_1,t_2,t_3} \rightarrow \Pbb^2_{x,y} \]
  is a (regular) smooth surjective morphism with geometrically integral fibers.
  Moreover, the point $P = (a_1,\ldots,a_d)$ is contained in the closure of every fiber of $X \rightarrow \Pbb^2_{x,y}$.
  Since $P$ is essentially unramified in $K$, it follows that all the fibers of $X \rightarrow \Pbb^2_{x,y}$ must be essentially unramified in $K$.
  Thus, the morphism $X \rightarrow \Pbb^2_{x,y}$ satisfies the assumptions of Proposition \ref{proposition: mainproof / K0 / independence-proposition}.

  Now suppose that $e \in k_0$ is given and put $t_e = (x-e)/y$.
  Consider the (rational) projection $\Pbb^2_{x,y} \rightarrow \Pbb^1_{t_e}$ defined by the inclusion of function fields $k(t_e) \hookrightarrow k(x,y)$.
  The fibers of this morphism are the lines in $\Pbb^2_{x,y}$ passing through the point $(1:e:0)$, and so this map is surjective onto $\Pbb^1_{t_e}$.
  Letting $Y_e$ denote the preimage of the point $(1:e:0)$ in $X$, and putting $X_e := X \smallsetminus Y_e$, we find that the composition 
  \[ X_e \hookrightarrow X \rightarrow \Pbb^2_{x,y} \rightarrow \Pbb^1_{t_e} \]
  is a (regular) smooth surjective morphism with geometrically integral fibers which are all essentially unramified in $K$.

  In particular, if $t_e = (x-e)/y$ is general in $K|k$, then we see that $t_e$ is automatically \emph{strongly-general} in $K|k$ by applying Lemma \ref{lemma: essram / mod-ell-divs / fibers} to the morphism $X_e \rightarrow \Pbb^1_{t_e}$.
  Moreover, as $1/y$ is separable in $K|k$, it follows from Lemma \ref{lemma: strgen / general / gen-birat-bertini}(1) that $t_e$ is general (hence strongly-general) for all but finitely many $e \in k_0$.

  Now let $b \in k$ be arbitrary as in the statement of the lemma.
  By Proposition \ref{proposition: mainproof / K0 / fix-K0}(2), for every $e \in k_0$ as above (i.e. such that $t_e$ is strongly-general), there exists a unique $c$ (which \emph{a priori} might depend on $e$) such that
  \[ \sigma\left\{ \frac{x-e}{y} + b \right\}_K = \left\{\frac{x-e}{y} + c\right\}_K. \]
  Fortunately, as noted above, we may directly use Proposition \ref{proposition: mainproof / K0 / independence-proposition} which implies that this $c$ doesn't depend on the given $e \in k_0$ such that $t_e$ is strongly-general.

  In other words, there exists a single $c \in k$ such that, for all $e \in k_0$ as above, one has
  \[ \sigma\left\{ \frac{x-e}{y} + b \right\}_K = \left\{\frac{x-e}{y} + c\right\}_K. \]
  Finally, by Lemma \ref{proposition: mainproof / K0 / fix-K0}(1), we may multiply both sides by $\sigma\{y\}_K = \{y\}_K$ to deduce that
  \[ \sigma\{x-e+b \cdot y\}_K = \{x-e + c \cdot y\}_K. \]
  We conclude that $\sigma\Kfrak(x+b \cdot y) = \Kfrak(x+ c \cdot y)$ by Corollary \ref{corollary: milnorff / geometric / infinite-subsets}, since the above equality holds true for all but finitely many elements $e$ of $k_0$.
\end{proof}

The following lemma concludes the base case for our induction.

\begin{lemma}
\label{lemma: mainproof / base-case / main-base-case}
  Assume that $\trdeg(K|k) \geq 5$ and let $x_0,y_0 \in K_0$ be given.
  Then for all $d \in k$, the element $x_0+d\cdot y_0$ is $\sigma$-acceptable.
\end{lemma}
\begin{proof}
  We may assume without loss of generality that $x_0,y_0$ are algebraically independent over $k$, for otherwise the claim is either trivial (if $x_0+d \cdot y_0 \in k$) or it follows from Lemma \ref{lemma: mainproof / K0 / fix-Kfrakt0}.

  Put $t = x_0+d \cdot y_0$.
  Let $t_1,t_2,t_3 \in K_0$ be algebraically independent over $\Kbb(x_0,y_0)$, and let $a,b \in k_0$ be as in Lemma \ref{lemma: mainproof / intersections / gen-intersections}, i.e., letting $x = t_1+at_3$ and $y = t_2 + b t_3$, one has
  \[ \Kfrak(t) = \Kfrak(t,x) \cap \Kfrak(t,y). \]

  Let $z \in K_0$ be algebraically independent over $\Kbb(x_0,y_0)$, and let $e,f,g \in k_0$ be such that 
  \[ \frac{x_0-e}{z-g}+d\cdot\frac{y_0-f}{z-g} \]
  is $\sigma$-acceptable; such $e,f,g \in k_0$ exist by Lemma \ref{lemma: mainproof / base-case / pre-base-case}.
  Now note that
  \[ \Kfrak(t,z) = \Kfrak\left(\frac{x_0-e}{z-g}+d \cdot \frac{y_0-f}{z-g},z\right). \]
  Since $z \in K_0$ is $\sigma$-acceptable by Lemma \ref{lemma: mainproof / K0 / fix-Kfrakt0}, it follows that $\{t,z\}$ is $\sigma$-acceptable by Lemma \ref{lemma: mainproof / base-case / pre-base-case} and Fact \ref{fact: mainproof / acceptable / acceptable}.
  Using similar arguments for various $z \in K_0$, we see that the following subsets of $K$ are all $\sigma$-acceptable as well by Fact \ref{fact: mainproof / acceptable / acceptable} and Lemma \ref{lemma: mainproof / K0 / fix-Kfrakt0}:
  \[ \{t_3\}, \ \{t,x\}, \ \{t,y\}, \ \{t,x,y\}, \ \{t,t_3\}, \ \{t,x,t_3\}, \ \{t,y,t_3\}. \]

  Moreover, recall that if $S$ is $\sigma$-acceptable and $T$ is a subset such that $\sigma\Kfrak(S) = \Kfrak(T)$, then one has 
  \[ \trdeg(\Kbb(S)|k) = \dimm(\Kfrak(S)) = \dimm(\Kfrak(T)) = \trdeg(\Kbb(T)|k) \]
  by Fact \ref{fact: milnorff / milnordim / geometric-dim} and Fact \ref{fact: mainproof / autm-condition}.
  It follows from these observations and Proposition \ref{proposition: lattice / mod-ell / mod-ell-lattice} that $\sigma\Kfrak(t)$ satisfies the assumptions of Lemma \ref{lemma: mainproof / intersections / geometric-criterion}, as follows.
  Indeed, recall that $\Kfrak(t)$ is maximal among subgroups $\Delta$ of $\k_1(K)$ such that $\dimm(\Delta) = 1$ by Fact \ref{fact: milnorff / vanishing / vanishing-B-K} and Proposition \ref{proposition: milnorff / geometric / geometric-maximal}.
  Thus, $\sigma\Kfrak(t)$ is also maximal with this property by Fact \ref{fact: mainproof / autm-condition}(1).
  Finally, in the notation of Lemma \ref{lemma: mainproof / intersections / geometric-criterion}, we can take
  \begin{enumerate}
    \item $A = \sigma\Kfrak(t_3)$ and $C = \sigma\Kfrak(t,t_3)$.
    \item $B_1 = \sigma\Kfrak(t,x)$, $B_2 = \sigma\Kfrak(t,y)$ and $D = \sigma\Kfrak(t,x,y)$.
    \item $B_1' = \sigma\Kfrak(t,x,t_3)$ and $B_2' = \sigma\Kfrak(t,y,t_3)$.
  \end{enumerate}
  Applying Lemma \ref{lemma: mainproof / intersections / geometric-criterion}, we deduce that $\sigma\Kfrak(t) \in \Gfrak^1(K|k)$, and therefore $t$ is $\sigma$-acceptable, as required.
\end{proof}

\subsection{The General Case}
\label{subsection: mainproof / general-case}

We are now ready to prove the final main step in our proof, that every element $t$ of $K$ is $\sigma$-acceptable.
As noted above, this will proceed by induction on the length of the expression
\[ t = a_0 x_0 + \cdots + a_r x_r, \ a_i \in k, \ x_i \in K_0, \]
with the base case $r = 1$ taken care of by Lemma \ref{lemma: mainproof / base-case / main-base-case}.

\begin{lemma}
\label{lemma: mainproof / general-case / all-acceptable}
  Assume that $\trdeg(K|k) \geq 5$.
  Then every element of $K$ is $\sigma$-acceptable.
\end{lemma}
\begin{proof}
  Recall that every element of $K$ is of the form $a_0 x_0 + \cdots + a_r x_r$ for $a_i \in k$ and $x_i \in K_0$.
  We proceed by induction on $r$, with the case $r = 0$ being Lemma \ref{lemma: mainproof / K0 / fix-Kfrakt0} and the case $r = 1$ being Lemma \ref{lemma: mainproof / base-case / main-base-case}.
  So assume that $r$ is fixed and that all elements of $K$ which can be written as
  \[ a_0 x_0 + \cdots + a_s x_s, \ a_i \in k, \ x_i \in K_0, \]
  with $s < r$, are $\sigma$-acceptable.
  Let $t \in K$ be an element of the form
  \[ t = a_0 x_0 + \cdots + a_r x_r \]
  with $a_i \in k$ and $x_i \in K_0$.
  As a first reduction, divide by $a_0$ to assume without loss of generality that $a_0 = 1$, so that
  \[ t = x_0 + a_1 x_1 + \cdots + a_r x_r. \]
  Also, we may assume that $a_r \in k \smallsetminus k_0$, for otherwise $x_0 + a_r x_r \in K_0$, hence $t$ is $\sigma$-acceptable by the inductive hypothesis.
  We may further assume without loss of generality that $x_r$ and $t$ are algebraically independent, for otherwise $\Kfrak(t) = \Kfrak(x_r)$ and so $t$ is $\sigma$-acceptable by Lemma \ref{lemma: mainproof / K0 / fix-Kfrakt0}.

  Choose $t_1,t_2 \in K_0$ which are algebraically independent over $\Kbb(t,x_r)$.
  By Lemma \ref{lemma: mainproof / intersections / gen-intersections}, we can choose $a,b \in k_0$ such that
  \[ \Kfrak(t) = \Kfrak(t,at_1+a_rx_r) \cap \Kfrak(t,bt_1+a_rx_r).\]
  Put $x = a t_1 + a_r x_r$ and $y = b t_1 + a_rx_r$.
  Then $t-x$, $t-y$, $t-a_r x_r$, $x$, $y$ and $x_r$ are all $\sigma$-acceptable by the inductive hypothesis, and so any subset of $\{t-x,t-y,t-a_r x_r,x,y,x_r\}$ is also $\sigma$-acceptable by Fact \ref{fact: mainproof / acceptable / acceptable}.
  Now we apply Lemma \ref{lemma: mainproof / intersections / geometric-criterion} to deduce the claim, as follows.
  First, $\sigma\Kfrak(t)$ is maximal among subgroups $\Delta$ of $\k_1(K)$ such that $\dimm(\Delta) = 1$, arguing as in the proof of Lemma \ref{lemma: mainproof / base-case / main-base-case}.
  Also, we may consider the following geometric subgroups of $\k_1(K)$ following the notation of Lemma \ref{lemma: mainproof / intersections / geometric-criterion}:
  \begin{enumerate}
    \item $A = \sigma \Kfrak(x_r) = \Kfrak(x_r)$ and $C = \sigma\Kfrak(t,x_r) = \sigma \Kfrak(t-a_r x_r,x_r)$.
    \item $B_1 = \sigma\Kfrak(t,x) = \sigma \Kfrak(t-x,x)$, $B_2 = \sigma\Kfrak(t,y) = \sigma \Kfrak(t-y,y)$.
    \item $D = \sigma\Kfrak(t,x,y) = \sigma\Kfrak(t-x,x,y)$.
    \item $B_1' = \sigma\Kfrak(t,x,x_r) = \sigma\Kfrak(t-x,x,x_r)$ and $B_2' = \sigma\Kfrak(t,y,x_r) = \sigma\Kfrak(t-y,y,x_r)$.
  \end{enumerate}
  By Proposition \ref{proposition: lattice / mod-ell / mod-ell-lattice} and the observations made above, we see that these subgroups of $\k_1(K)$ satisfy the assumptions of Lemma \ref{lemma: mainproof / intersections / geometric-criterion}.
  It therefore follows from Lemma \ref{lemma: mainproof / intersections / geometric-criterion} that $\sigma\Kfrak(t) \in \Gfrak^1(K|k)$ is geometric, hence $t$ is $\sigma$-acceptable, as required.
\end{proof}

\subsection{Concluding the Proofs of Theorems \ref{maintheorem: intro / galois-variant / galois-main} and \ref{maintheorem: intro / milnor-variant / milnor-main}}
\label{subsection: mainproof / conclusion}

For the rest of this section, we assume that $\trdeg(K|k) \geq 5$ so that we can use Corollary \ref{corollary: lattice / galois-mod-ell / galois-lattice-iso} and Lemma \ref{lemma: mainproof / general-case / all-acceptable}.

By Lemma \ref{lemma: mainproof / general-case / all-acceptable}, any element of $K$ is $\sigma$-acceptable.
Thus any \emph{subset} of $K$ is $\sigma$-acceptable by Fact \ref{fact: mainproof / acceptable / acceptable}.
Moreover, by Fact \ref{fact: milnorff / milnordim / geometric-dim} and Fact \ref{fact: mainproof / autm-condition}(1), if $S,T$ are two subsets of $K$ such that $\sigma\Kfrak(S) = \Kfrak(T)$, then one has $\trdeg(\Kbb(S)|k) = \trdeg(\Kbb(T)|k)$.
In particular, the map $A \mapsto \sigma A$ induces an automorphism of the graded lattice $\Gfrak^*(K|k)$ of geometric subgroups of $\k_1(K)$.
In other words, we obtain a canonical homomorphism
\[ \sigma \mapsto (\Kfrak \mapsto \sigma \Kfrak) : \Autm_\abf(\k_1(K)) \rightarrow \Aut^*(\Gfrak^*(K|k)) \]
which factors through $\UAutm_\abf(\k_1(K))$.
Furthermore, by Lemma \ref{lemma: mainproof / K0 / fix-Kfrakt0}, it follows that $\sigma\Kfrak(S) = \Kfrak(S)$ for subsets $S$ of $K_0$.
In other words, the image of this canonical map $\UAutm_\abf(\k_1(K)) \rightarrow \Aut^*(\Gfrak^*(K|k))$ actually lands in the subgroup $\Aut^*(\Gfrak^*(K|k)|K_0)$ which was defined in \S\ref{subsection: lattice / galois-mod-ell}.

Finally, it is easy to see that the map $\UAutm_\abf(\k_1(K)) \rightarrow \Aut^*(\Gfrak^*(K|k)|K_0)$ is compatible with $\rho_{k_0}$.
Namely, the following diagram commutes:
\[
\xymatrix{
  \Galk \ar[r]^-{\rho_{k_0}} \ar[dr]_{\rho_{k_0}} & \UAutm_\abf(\k_1(K))\ar[d] \\
  & \Aut^*(\Gfrak^*(K|k)|K_0).
}
\]
By Corollary \ref{corollary: lattice / galois-mod-ell / galois-lattice-iso}, the map $\rho_{k_0} : \Galk \rightarrow \Aut^*(\Gfrak^*(K|k)|K_0)$ is an isomorphism.
On the other hand, it immediately follows from Proposition \ref{proposition: strgen / rational-like / ident-criterion} that the map
\[ \UAutm_\abf(\k_1(K)) \rightarrow \Aut^*(\Gfrak^*(K|k)|K_0) \subset \Aut^*(\Gfrak^*(K|k)) \]
is \emph{injective}.
Hence the map
\[ \rho_{k_0} : \Galk \rightarrow \UAutm_\abf(\k_1(K)) \]
is an isomorphism as well.
This concludes the proof of Theorem \ref{maintheorem: intro / milnor-variant / milnor-main}.

Finally, it easily follows from Kummer theory that the isomorphism $\UAutc(\Gc_K^a) \cong \UAutm(\k_1(K))$ of Theorem \ref{theorem: cohom / kummer / galois-to-milnor} restricts to an isomorphism $\UAutc_\abf(\Gc_K^a) \cong \UAutm_\abf(\k_1(K))$.
Thus, Theorem \ref{maintheorem: intro / galois-variant / galois-main} follows immediately from Theorem \ref{maintheorem: intro / milnor-variant / milnor-main} by applying Theorem \ref{theorem: cohom / kummer / galois-to-milnor}.

\section{Concluding the Proof of the mod-$\ell$ I/OM}
\label{section: finalproof}


We now turn to the proof of Theorems \ref{maintheorem: intro / main-result / iom-main-opens} and \ref{maintheorem: intro / main-result / iom-main-connected}.
As we will see, Theorem \ref{maintheorem: intro / main-result / iom-main-opens} follows rather easily from Theorem \ref{maintheorem: intro / galois-variant / galois-main}.
On the other hand, Theorem \ref{maintheorem: intro / main-result / iom-main-connected} follows from Theorem \ref{maintheorem: intro / main-result / iom-main-opens} more-or-less because of our definition of a \emph{$5$-connected subcategory} of $\Var_{k_0}$.

\subsection{Proof of Theorem \ref{maintheorem: intro / main-result / iom-main-opens}}
\label{subsection: finalproof / ion-main-opens}

Let $X$ be a normal $k_0$-variety of dimension $\geq 5$, and let $\Ucal$ be a birational system of $X$.
Put $K = k(X)$.
First, since $X$ is geometrically normal, we recall that for every $U \in \Ucal$, one has canonical \emph{surjective} morphisms
\[ \Gc_K^c \twoheadrightarrow \pi^c(U), \ \ \Gc_K^a \twoheadrightarrow \pi^a(U). \]
Moreover, taking limits over $U \in \Ucal$, one has
\[ \Gc_K^c = \varprojlim_{U \in \Ucal} \pi^c(U), \ \ \Gc_K^a = \varprojlim_{U \in \Ucal} \pi^a(U). \]
Therefore, we obtain a canonical morphism
\[ \Autc(\pi^a|_\Ucal) \rightarrow \Autc(\Gc_K^a) \]
which is defined by sending a system of automorphisms $(\phi_U)_{U \in \Ucal}$ in $\Autc(\pi^a|_\Ucal)$ to its projective limit $\varprojlim_U \phi_U \in \Autc(\Gc_K^a)$.
This map clearly induces a corresponding map on $(\Z/\ell)^\times$-classes of automorphisms
\[ \UAutc(\pi^a|_\Ucal) \rightarrow \UAutc(\Gc_K^a). \]

Since the maps $\Gc_K^a \rightarrow \pi^a(U)$ are all surjective for $U \in \Ucal$, it follows that the map $\Autc(\pi^a|_\Ucal) \rightarrow \Autc(\Gc_K^a)$ is injective, hence $\UAutc(\pi^a|_\Ucal) \rightarrow \UAutc(\Gc_K^a)$ is injective as well.
Finally, for a finite tuple $\abf$ of elements of $k_0^\times$, it follows immediately from the definitions that the image of the canonical map
\[ \Autc(\pi^a|_{\Ucal_\abf}) \xrightarrow{\phi \mapsto \phi|_\Ucal} \Autc(\pi^a|_\Ucal) \hookrightarrow \Autc(\Gc_K^a)\]
lands in the subgroup $\Autc_\abf(\Gc_K^a)$.

To conclude, we note that if we choose $t \in K_0^\times$ which is general in $K|k$ (such a $t$ exists by Lemma \ref{lemma: strgen / general / gen-birat-bertini}), then the canonical map $\Gc_K^a \rightarrow \pi^a(\U_\abf)$ induced by $t$ is \emph{surjective}.
Moreover, this map factors through $\pi^a(U)$ for some $U \in \Ucal$ which is sufficiently small.
From this observation, it follows that the map $\Autc(\pi^a|_{\Ucal_\abf}) \rightarrow \Autc(\pi^a|_\Ucal)$ is injective, hence the induced maps
\[ \Autc(\pi^a|_{\Ucal_\abf}) \rightarrow \Autc_\abf(\Gc_K^a), \ \ \UAutc(\pi^a|_{\Ucal_\abf}) \rightarrow \UAutc_\abf(\Gc_K^a) \]
are both injective as well.
These injections are clearly compatible with $\rho_{k_0}$, and therefore Theorem \ref{maintheorem: intro / main-result / iom-main-opens} follows from Theorem \ref{maintheorem: intro / galois-variant / galois-main}.

\subsection{Proof of Theorem \ref{maintheorem: intro / main-result / iom-main-connected}}
\label{subsection: finalproof / ion-main-connected}

Before we conclude the proof of Theorem \ref{maintheorem: intro / main-result / iom-main-connected}, we need a small lemma concerning the {\bf domination} condition between two birational systems, as defined in \S\ref{subsection: intro / main-result-connected}.

\begin{lemma}
\label{lemma: finalproof / iom-main-connected / dominant-lemma}
  Let $\Vc$ be a subcategory of $\Var_{k_0}$, and let $\Ucal_1$ and $\Ucal_2$ be two positive-dimensional birational systems, such that $\Ucal_1$ dominates $\Ucal_2$ in $\Vc$.
  Let $\phi \in \Autc(\pi^a|_\Vc)$ be given, and assume that $\phi|_{\Ucal_1} \in \Autc(\pi^a|_{\Ucal_1})$ is defined by $\tau \in \Galk$.
  Then the following hold:
  \begin{enumerate}
    \item The restriction $\phi|_{\Ucal_2} \in \Autc(\pi^a|_{\Ucal_2})$ is also defined by $\tau$. 
    \item If there exists some $\epsilon \in (\Z/\ell)^\times$ and some $\tau' \in \Galk$ such that $\epsilon \cdot \phi|_{\Ucal_2}$ is defined by $\tau'$, then $\tau' = \tau$ and $\epsilon = 1$.
  \end{enumerate}
\end{lemma}
\begin{proof}
  Put $K = k(X)$ for some $X \in \Ucal_2$ and $L = k(Y)$ for some $Y \in \Ucal_1$.
  Since $\Ucal_1$ dominates $\Ucal_2$ in $\Vc$, we see that $K$ is a subfield of $L$.
  \vskip 5pt
  \noindent\emph{Proof of (1):}
  Arguing as in \S\ref{subsection: finalproof / ion-main-opens}, it suffices to show that the image of $\phi$ under the map
  \[ \Autc(\pi^a|_\Vc) \rightarrow \Autc(\pi^a|_{\Ucal_2}) \hookrightarrow \Autc(\Gc_K^a) \]
  is defined by $\tau$.
  Let $\phi_L$ resp. $\phi_K$ denote the image of $\phi$ in $\Autc(\Gc_L^a)$ resp. $\Autc(\Gc_K^a)$.
  Since $\phi|_{\Ucal_1}$ is defined by $\tau$, we see that $\phi_L$ is defined by $\tau$ as well.
  
  Note that the assumptions of the lemma imply that $\phi_L$ and $\phi_K$ are compatible with the canonical morphism $\Gc_L^a \rightarrow \Gc_K^a$, i.e., the following diagram commutes:
  \[ 
  \xymatrix{
    \Gc_L^a \ar[d] \ar[r]^{\phi_L} & \Gc_L^a \ar[d] \\
    \Gc_K^a \ar[r]_{\phi_K} & \Gc_K^a.
  }
  \]
  We will now deal with the cases $\dim \Ucal_2 \neq 1$ and $\dim \Ucal_2 = 1$ separately.
  \vskip 5pt
  \noindent\underline{Case $\dim \Ucal_2 = 1$:} In this case, our assumptions ensure that $\Gc_L^a \rightarrow \Gc_K^a$ is \emph{surjective}.
  Since $\phi_L$ is defined by $\tau$, and since $\phi_L$ and $\phi_K$ are compatible with the projection $\Gc_L^a \twoheadrightarrow \Gc_K^a$, it follows that $\phi_K$ must also be defined by $\tau$.

  \vskip 5pt
  \noindent\underline{Case $\dim \Ucal_2 > 1:$} Let $\psi = \tau^{-1} \cdot \phi_K$ denote the composition of $\phi_K$ with the element of $\Autc(\Gc_K^a)$ induced by $\tau^{-1}$.
  Furthermore, let $\psi^* \in \Autm(\k_1(K))$ be the element associated to $\psi$ via the Kummer pairing (see Theorem \ref{theorem: cohom / kummer / galois-to-milnor}).
  Finally, note that the image of $\Gc_L^a \rightarrow \Gc_K^a$ is an open subgroup of $\Gc_K^a$, on which $\psi$ acts as the identity.
  Thus, there exists a \emph{finite} subgroup $H_0$ of $\k_1(K)$ such that for every $x \in \k_1(K)$, one has $\psi^* x \in x + H_0$.
  Therefore, there is a subgroup $H$ of $\k_1(K)$ such that $\psi^*$ acts as the identity on $H$, and such that $H$ has finite index in $\k_1(K)$.

  We will show that $\psi^*$ is some $(\Z/\ell)^\times$-multiple of the identity on $\k_1(K)$.
  First, suppose that $t \in K \smallsetminus k$ is given and consider the one-dimensional geometric subgroup $\Kfrak(t)$.
  The inclusion $k(t) \hookrightarrow K$ induces a (possibly non-injective) map $\k_1(k(t)) \rightarrow \k_1(K)$ and we let $M$ denote the preimage of $H$ in $\k_1(k(t))$.
  Since $H$ has finite index in $\k_1(K)$, it follows that $M$ must have finite index in $\k_1(k(t))$.

  \begin{claim*}
    There exists some $x \in k(t) \smallsetminus k$ such that $\{x-c\}_{k(t)} \in M$ for infinitely many $c \in k$.
  \end{claim*}
  \begin{proof}
    Since $M$ has finite index in $\k_1(k(t))$, there exists some $f \in k(t)^\times$ such that the set 
    \[ \mathcal{S} := \{a \in k \ : \ \{f \cdot (t-a)\}_{k(t)} \in M\} \]
    is infinite.
    Note that for all $a,b \in \mathcal{S}$ one has $\{(t-a)/(t-b)\}_{k(t)} \in M$.
    Let $a,b \in \mathcal{S}$ be two different elements with $b \neq 0$.
    Put $x := (t-a)/(t-b)$, and let $c \in k$ be a constant such that $c \neq 1$.
    Since $\{1-c\}_{k(t)} = 0$, we find that
    \[ \{x-c\}_{k(t)} = \left\{\frac{t-a-ct+cb}{t-b}\right\}_{k(t)} = \left\{\frac{t-\frac{a-cb}{1-c}}{t-b}\right\}_{k(t)}.\]
    In particular, $\{x-c\}_{k(t)} \in M$ for every $c \in k \smallsetminus \{1\}$ such that $(a-cb)/(1-c) \in \mathcal{S}$.
    Since $\mathcal{S}$ was infinite, we see that there are infinitely many such constants $c$, which proves the claim.
  \end{proof}

  Let $x \in k(t) \smallsetminus k$ be as in the claim above.
  Then one has $\psi^*\{x-c\}_K = \{x-c\}_K$ for infinitely many $c \in k$.
  As $\Kfrak(t) = \Kfrak(x)$, it follows from Corollary \ref{corollary: milnorff / geometric / infinite-subsets} that $\psi^* \Kfrak(t) = \Kfrak(t)$.
  Finally, since $t \in K \smallsetminus k$ was arbitrary, Proposition \ref{proposition: strgen / rational-like / ident-criterion} shows that $\psi^* \in (\Z/\ell)^\times \cdot \one_{\k_1(K)}$.
  In particular, $\psi = \tau^{-1} \cdot \phi_K$ is contained in $(\Z/\ell)^\times \cdot \one_{\Gc_K^a}$.

  Now note that we have a commutative diagram:
  \[ 
  \xymatrix{
    \Gc_L^a \ar[d] \ar@{=}[r] & \Gc_L^a \ar[d] \\
    \Gc_K^a \ar[r]_{\tau^{-1} \cdot \phi_K} & \Gc_K^a
  }
  \]
  and recall that the vertical arrows in this diagram have an open (hence non-trivial) image in $\Gc_K^a$.
  Since $\tau^{-1} \cdot \phi_K \in (\Z/\ell)^\times \cdot \one_{\Gc_K^a}$, it follows that $\tau^{-1} \cdot \phi_K = \one_{\Gc_K^a}$.
  This proves assertion (1).
  \vskip 5pt
  \noindent\emph{Proof of (2):}
  Recall that $\phi_K \in \Autc(\Gc_K^a)$ is defined by $\tau$, and the assumption of (2) implies that $\epsilon \cdot \phi_K \in \Autc(\Gc_K^a)$ is defined by $\tau'$.
  Since $\phi_K$ and $\epsilon \cdot \phi_K$ represent the same element of $\UAutc(\Gc_K^a)$, and since one has a $\rho_{k_0}$-compatible isomorphism 
  \[ \UAutc(\Gc_K^a) \cong \UAutm(\k_1(K)) \]
  by Theorem \ref{theorem: cohom / kummer / galois-to-milnor}, it follows from Lemma \ref{lemma: strgen / galois-action / faithfulness} that $\tau = \tau'$.
  Finally, since the map 
  \[ \rho_{k_0} : \Galk \rightarrow \Autc(\Gc_K^a) \twoheadrightarrow \UAutc(\Gc_K^a) \cong \UAutm(\k_1(K)) \]
  is injective by Lemma \ref{lemma: strgen / galois-action / faithfulness}, we see that $\epsilon = 1$, as required.
\end{proof}

We will now assume that $\Vc$ is a $5$-connected subcategory of $\Var_{k_0}$, and we will conclude the proof of Theorem \ref{maintheorem: intro / main-result / iom-main-connected}, which states that the canonical map 
\[ \rho_{k_0} : \Galk \rightarrow \UAutc(\pi^a|_\Vc) \]
is an isomorphism.

First we show that this map is injective.
Since $\Vc$ contains a positive-dimensional object, it follows from the definition of $\Vc$ being $5$-connected that $\Vc$ contains a positive-dimensional birational system $\Ucal$.
By first restricting to $\Ucal$ and then taking limits over the objects of $\Ucal$ as above, we obtain a canonical map,
\[ \Autc(\pi^a|_\Vc) \rightarrow \Autc(\pi^a|_\Ucal) \rightarrow \Autc(\Gc_{k(U)}^a), \]
where $U$ is some object of $\Ucal$.
This map is clearly compatible with $\rho_{k_0}$.
Finally, recall that one has a $\rho_{k_0}$-compatible isomorphism 
\[ \UAutc(\Gc_{k(U)}^a) \cong \UAutm(\k_1(k(U))) \]
by Theorem \ref{theorem: cohom / kummer / galois-to-milnor}.
Thus, injectivity follows from Lemma \ref{lemma: strgen / galois-action / faithfulness}.

The proof of the surjectivity of this map is more difficult, but it essentially follows from the technical definition of $\Vc$ being $5$-connected and Theorem \ref{maintheorem: intro / main-result / iom-main-opens}, as follows.
First, let us fix an element $\phi$ of $\Autc(\pi^a|_\Vc)$.
Recall that this element $\phi$ is represented by a system $(\phi_X)_{X \in \Vc}$ with $\phi_X \in \Autc(\pi^a(X))$ which is compatible with the morphisms from $\Vc$.
To prove surjectivity, we must show that there exists some $\tau \in \Galk$ and some $\epsilon \in (\Z/\ell)^\times$, such that for all positive-dimensional $X \in \Vc$, $\epsilon \cdot \phi_X$ is defined by $\tau$.

By the definition of $\Vc$ being $5$-connected, we see that there exists some birational system $\Ucal$ and some finite tuple $\abf$ of elements of $k_0^\times$ such that $\dim \Ucal \geq 5$ and such that $\Ucal_\abf$ is contained in $\Vc$.
Thus, by Theorem \ref{maintheorem: intro / main-result / iom-main-opens}, there exists a unique $\tau \in \Galk$ and an $\epsilon \in (\Z/\ell)^\times$ such that $\epsilon \cdot \phi|_{\Ucal_\abf}$ is defined by $\tau$, and thus $\epsilon \cdot \phi|_\Ucal$ is defined by $\tau$ as well.
To simplify the notation, we replace $\phi$ by $\epsilon \cdot \phi$, so we must prove that $\phi$ itself is defined by $\tau$.

Let $X$ be any positive-dimensional object of $\Vc$ and put $\Ucal = \Ucal_0$.
Since $\Vc$ is $5$-connected, we see that there exists a birational system $\Ucal_{2r} = \Ucal_X^+$ which contains $X$, and birational systems $\Ucal_1,\ldots,\Ucal_{2r-1}$ such that, for all $i = 0,\ldots,r-1$, the following conditions hold:
\begin{enumerate}
  \item One has $\dim \Ucal_{2i+1} \geq 5$. 
  \item The birational system $\Ucal_{2i+1}$ attaches $\Ucal_{2i}$ to $\Ucal_{2i+2}$ in $\Vc$.
\end{enumerate}
Applying Theorem \ref{maintheorem: intro / main-result / iom-main-opens} again, we find that there exist $\tau_0,\ldots,\tau_{r-1} \in \Galk$ and $\epsilon_0,\ldots,\epsilon_{r-1} \in (\Z/\ell)^\times$ such that $\epsilon_i \cdot \phi|_{\Ucal_{2i+1}}$ is defined by $\tau_i$ for all $i = 0,\ldots,r-1$.
By Lemma \ref{lemma: finalproof / iom-main-connected / dominant-lemma}(1), we see that for all $i = 0,\ldots,r-1$, $\epsilon_i \cdot \phi|_{\Ucal_{2i}}$ and $\epsilon_i \cdot \phi|_{\Ucal_{2i+2}}$ are both defined by $\tau_i$.
Since $\phi|_{\Ucal_0}$ was defined by $\tau$, we see that $\tau_0 = \tau$ and that $\epsilon_0 = 1$ by Lemma \ref{lemma: finalproof / iom-main-connected / dominant-lemma}(2).
Hence $\phi|_{\Ucal_2}$ is also defined by $\tau$ by Lemma \ref{lemma: finalproof / iom-main-connected / dominant-lemma}.
Proceeding inductively on $i = 0,\ldots,r-1$, we find that $\tau_i = \tau$ and that $\epsilon_i = 1$ for all $i = 0,\ldots,r-1$.
In particular, $\phi|_{\Ucal_{2r-1}}$ is defined by $\tau$, so that $\phi|_{\Ucal_{2r}}$ is defined by $\tau$ as well by Lemma \ref{lemma: finalproof / iom-main-connected / dominant-lemma}.

Since $X$ is contained in $\Ucal_{2r} = \Ucal_X^+$, we deduce that $\phi_X$ is defined by $\tau$ as well.
As $X \in \Vc$ was arbitrary, we see that the original element $\phi$ is defined by $\tau$.
In other words, the map
\[ \rho_{k_0} : \Galk \rightarrow \UAutc(\pi^a|_\Vc) \]
is surjective.
This concludes the proof of Theorem \ref{maintheorem: intro / main-result / iom-main-connected}.


\bibliography{TOPAZ_elliom_refs}

\end{document}